\documentclass[12pt, reqno]{amsart}
\usepackage{amssymb}
\usepackage{graphicx}
\usepackage{xcolor} 
\usepackage{tensor}
\usepackage{enumitem}
\usepackage{hyperref}

\usepackage{fullpage} 

\usepackage{array}

\definecolor{green}{rgb}{0,0.8,0} 

\newtheorem{theorem}{Theorem}[section]
\newtheorem{corollary}[theorem]{Corollary}
\newtheorem{lemma}[theorem]{Lemma}
\newtheorem{proposition}[theorem]{Proposition}
\theoremstyle{definition}
\newtheorem{definition}[theorem]{Definition}

\theoremstyle{remark}
\newtheorem{remark}[theorem]{Remark}
\numberwithin{equation}{section}
\newcommand{\nrm}[1]{\Vert#1\Vert}
\newcommand{\abs}[1]{\vert#1\vert}
\newcommand{\brk}[1]{\langle#1\rangle}
\newcommand{\set}[1]{\{#1\}}

\newcommand{\tr}{\mathrm{tr}\,}

\newcommand{\aeq}{\sim}
\newcommand{\aleq}{\lesssim}
\newcommand{\ageq}{\gtrsim}

\newcommand{\lap}{\Delta}

\newcommand{\ud}{\mathrm{d}}
\newcommand{\rd}{\partial}
\newcommand{\nb}{\nabla}

\newcommand{\bb}{\Big}

\newcommand{\alp}{\alpha}
\newcommand{\bt}{\beta}
\newcommand{\gmm}{\gamma}
\newcommand{\Gmm}{\Gamma}
\newcommand{\dlt}{\delta}

\newcommand{\eps}{\epsilon}
\newcommand{\veps}{\varepsilon}
\newcommand{\kpp}{\kappa}
\newcommand{\lmb}{\lambda}
\newcommand{\Lmb}{\Lambda}
\newcommand{\sgm}{\sigma}

\newcommand{\tht}{\theta}
\newcommand{\Tht}{\Theta}

\newcommand{\omg}{\omega}
\newcommand{\Omg}{\Omega}

\newcommand{\bfd}{{\bf d}}
\newcommand{\bfe}{{\bf e}}

\newcommand{\bfg}{{\bf g}}

\newcommand{\bfm}{{\bf m}}

\newcommand{\bfA}{{\bf A}}

\newcommand{\bfD}{{\bf D}}

\newcommand{\bfF}{{\bf F}}
\newcommand{\bfG}{{\bf G}}
\newcommand{\bfH}{{\bf H}}

\newcommand{\bfP}{{\bf P}}
\newcommand{\bfQ}{{\bf Q}}

\newcommand{\bbD}{\mathbb D}

\newcommand{\bbH}{\mathbb H}

\newcommand{\bbN}{\mathbb N}

\newcommand{\bbP}{\mathbb P}

\newcommand{\bbR}{\mathbb R}
\newcommand{\bbS}{\mathbb S}

\newcommand{\bbZ}{\mathbb Z}

\newcommand{\calA}{\mathcal A}

\newcommand{\calC}{\mathcal C}
\newcommand{\calD}{\mathcal D}
\newcommand{\calE}{\mathcal E}
\newcommand{\calF}{\mathcal F}
\newcommand{\calG}{\mathcal G}
\newcommand{\calH}{\mathcal H}

\newcommand{\calL}{\mathcal L}
\newcommand{\calM}{\mathcal M}

\newcommand{\calO}{\mathcal O}

\newcommand{\calQ}{\mathcal Q}

\newcommand{\calX}{\mathcal X}

\newcommand{\frka}{\mathfrak a}
\newcommand{\frkb}{\mathfrak b}

\newcommand{\weakto}{\rightharpoonup}

\newcommand{\covD}{\bfD}

\newcommand{\scovD}{\not \hskip-.3em \covD}
\newcommand{\snb}{{\hskip-.1em \not \hskip-.3em \covnb}}
\newcommand{\slap}{{\not \hskip-.25em \lap}}
\newcommand{\smet}{{\not \hskip-.25em \bfg}}

\newcommand{\uL}{\underline{L}}
\newcommand{\ualp}{\underline{\alp}}
\newcommand{\tA}{{\tilde{A}}}
\newcommand{\tF}{{\tilde{F}}}
\newcommand{\ta}{\tilde{a}}
\newcommand{\te}{\tilde{e}}
\newcommand{\g}{\mathfrak{g}}
\renewcommand{\H}{{\mathcal H}}
\newcommand{\R}{\mathbb R}
\newcommand{\mvC}[1]{{}^{(#1)} P}
\newcommand{\mfa}{\mathfrak{a}}
\newcommand{\mfb}{\mathfrak{b}}

\newcommand{\G}{\mathbf{G}}

\newcommand{\la}{\langle}
\newcommand{\ra}{\rangle}

\renewcommand{\aa}{\alpha}
\renewcommand{\bb}{\beta}

\newcommand{\pfstep}[1]{\vskip.5em \noindent {\bf #1.}}

\newcommand{\ED}{ED}					

\newcommand{\EM}{T}				
\newcommand{\defT}[1]{{}^{(#1)} \pi}
\newcommand{\vC}[1]{{}^{(#1)} P}
\newcommand{\sC}[1]{{}^{(#1)} K}
\newcommand{\EFlux}{\calF}
\newcommand{\met}{\bfm}
\newcommand{\Str}{\text{Str}}
\newcommand\DA{{\mathbf{DA}}}

\newcommand{\Egs}{{E_{\text{GS}}}}
\newcommand{\hM}{\calQ} 
\newcommand{\bhM}{{\bf \calQ}} 

\newcommand{\nE}{{\calE}} 
\newcommand{\spE}{{\calE_{e}}} 

\newcommand{\ch}{\boldsymbol{\chi}}

\newcommand{\En}{{E_{3}}}		
\newcommand{\dltnc}{\dlt_{0}}		

\newcommand{\covnb}{\boldsymbol{\nb}}

\newcommand{\epsU}{\eps_{0}} 	
\newcommand{\epse}{\eps_{\bfe}} 	
\newcommand{\epsA}{\eps_{A}} 	

\begin{document}

\title[]{The threshold conjecture for the energy critical hyperbolic Yang--Mills equation}
\author{Sung-Jin Oh}%
\address{Department of Mathematics, UC Berkeley, Berkeley, CA 94720 and KIAS, Seoul, Korea 02455}%
\email{sjoh@math.berkeley.edu}%

\author{Daniel Tataru}%
\address{Department of Mathematics, UC Berkeley, Berkeley, CA 94720}%
\email{tataru@math.berkeley.edu}%


\begin{abstract}
  This article represents the fourth and final part of a four-paper
  sequence whose aim is to prove the Threshold Conjecture as well as
  the more general Dichotomy Theorem for the energy critical $4+1$
  dimensional hyperbolic Yang--Mills equation.  The Threshold Theorem
  asserts that topologically trivial solutions
   with energy below twice the ground state energy are global
  and scatter. The Dichotomy Theorem applies to solutions in arbitrary topological class 
  with large energy, and provides two exclusive alternatives:
  Either the solution is global and scatters, or it bubbles off a
  soliton in either finite time or infinite time.

  Using the caloric gauge developed in the first paper \cite{OTYM1},
  the continuation/scattering criteria established in the second paper
  \cite{OTYM2}, and the large data analysis in an arbitrary
  topological class at optimal regularity in the third paper
  \cite{OTYM2.5}, here we perform a blow-up analysis which shows that
  the failure of global well-posedness and scattering implies either
  the existence of a soliton with at most the same energy bubbling
  off, or the existence existence of a nontrivial self-similar
  solution.  The proof is completed by showing that the latter
  solutions do not exist.

\end{abstract}
\maketitle

\tableofcontents

\section{Introduction} \label{sec:intro}

This article represents the fourth and last of a four-paper sequence
devoted to the study of finite energy solutions to the energy
critical $4+1$ dimensional hyperbolic Yang--Mills equation.  The four
installments of the series are concerned with
\begin{enumerate}[label=(\alph*)]
\item  the \emph{caloric gauge} for 
the hyperbolic Yang--Mills equation, \cite{OTYM1};
\item large data \emph{energy dispersed} caloric 
gauge solutions,   \cite{OTYM2};
\item \emph{topological classes} of connections and large data local
well-posedness,  \cite{OTYM2.5};
\item the Threshold Conjecture and soliton bubbling/scattering dichotomy for large data solutions, present article.
\end{enumerate}
A short overview of the four papers is provided in the survey paper \cite{OTYM0}.
\medskip

Our first goal in this paper is to prove the Threshold Conjecture for the hyperbolic
Yang--Mills equation; this asserts that the solution is global and
scatters for all topologically trivial data with energy below $2\Egs$,
where $\Egs$ represents energy of the ground state (i.e., lowest energy steady state) for this problem.

Secondly, we consider solutions with energy above this threshold, and 
prove the following Dichotomy Theorem: either (i) the solution is 
topologically trivial, global and scatters, or (ii) it ``bubbles off'' a soliton
in either finite time (which corresponds to blow-up) or in infinite time.
Here ``soliton bubbling off'' means that  a sequence of symmetry- 
and gauge-equivalent solutions must converge
to a soliton, namely a Lorentz transform of a steady state.   

As a main common component of both theorems, we separately state 
and prove a Bubbling Theorem, which provides a necessary and 
sufficient condition for soliton bubbling off purely in terms of the 
the energy distribution of the solution.

The paper is organized as follows. In the first section we provide
some background material on the hyperbolic Yang--Mills equation, and
then we give the statements of the main results in
Theorem~\ref{t:bubble-off} (the Bubbling Theorem), 
Theorem~\ref{t:threshold} (the Threshold Theorem), and finally
Theorem~\ref{t:no-bubble} (the Dichotomy Theorem).  In the second
section we provide a brief overview of the results in the first three
papers of the sequence \cite{OTYM1},\cite{OTYM2},\cite{OTYM2.5}. The remainder of the paper
is devoted to the proof of the main results.

\subsection{The energy critical hyperbolic Yang--Mills equation} 

\subsubsection{Lie groups and algebras}
Let $\G$ be a compact noncommutative Lie group and $\g$ its associated Lie
algebra. We denote by $Ad(O) X = O X O^{-1}$ the action of $\G$ on
$\g$ by conjugation (i.e., the adjoint action), and by $ad(X) Y = [X,
Y]$ the associated action of $\g$, which is given by the Lie
bracket. We introduce the notation $\brk{X, Y}$ for a bi-invariant inner product on $\g$,
\begin{equation*}
\la [X,Y],Z \ra = \la X, [Y,Z] \ra, \qquad X,Y,Z \in \g, 
\end{equation*}
or equivalently 
\begin{equation*}
\la X,Y \ra = \la Ad(O) X,  Ad(O) Y  \ra, \qquad X,Y \in \g, \quad O \in \G. 
\end{equation*}
If $\G$ is semisimple then one can take 
$\brk{X, Y} = -\tr(ad(X) ad(Y))$ i.e. negative of the Killing form on $\g$, which is then positive definite.
However, a bi-invariant inner product on $\g$ exists for any compact Lie group $\G$.

An important concrete case is $\bfG = SU(2)$, the group of $2 \times 2$ unitary matrices with unit determinant. In that case, $\g = su(2)$, which is the space of $2 \times 2$ anti-hermitian matrices with zero trace, $[X, Y] = XY - YX$, $Ad(O) X = O X O^{-1}$, and $\brk{X, Y} = - \tr XY$ for $X, Y \in su(2)$ and $O \in SU(2)$, with the usual matrix multiplication and trace operations.

\subsubsection{The hyperbolic Yang--Mills equation}
Let $\R^{1+4}$ be the $(4+1)$ dimensional Minkowski space  with
the standard Lorentzian metric $\met= \text{diag}(-1,1,1,1,1)$. 
Denote by $A_{\alpha}: \R^{1+4}\rightarrow \g$, $\alpha = 0, 1, \ldots,4$,
a connection $1$-form\footnote{The geometric setting for the hyperbolic Yang--Mills equation is the space of connections on a vector bundle on a Lorentzian manifold; here, for simplicity, we give a concrete formulation on $\bbR^{1+4}$. For a more geometric description, we refer the reader to \cite{OTYM2.5}.} taking values in the Lie algebra $\g$, and by $\covD_\alpha$
the associated covariant differentiation,
\[
\covD_{\alpha} B:= \partial_{\alpha} B + [A_{\alpha},B],
\]
acting on $\g$-valued functions $B$. The commutator of two covariant derivatives takes the form $(\covD_{\alp} \covD_{\bt} - \covD_{\bt} \covD_{\alp}) B = [F_{\alp \bt}, B]$, where $F$ is the curvature tensor 
\[
F_{\alpha\beta}: = \partial_{\alpha}A_{\beta}
- \partial_{\beta}A_{\alpha} +[A_\alpha,A_\beta].
\]
The curvature tensor obeys the \emph{Bianchi identity}, namely
\begin{equation*}
	\covD_{\alp} F_{\bt \gmm} + \covD_{\bt} F_{\gmm \alp} + \covD_{\gmm} F_{\alp \bt} = 0.
\end{equation*}

The (hyperbolic) {\it{Yang--Mills}} equation for $A$ is the Euler--Lagrange
equation associated with the formal Lagrangian action functional
\[
\mathcal{L}(A) = \frac{1}{2}\int_{\R^{1+4}} \la F_{\alpha\beta}, F^{\alpha\beta}\ra \,dxdt.
\]
Here we are using the standard convention of raising or lowering indices using the metric $\bfm$, as well as summing up repeated upper and lower indices. Thus,
the Yang--Mills equation takes the form
\begin{equation}\label{ym}
\covD^\alpha F_{\alpha \beta} = 0.
\end{equation}
There is a natural energy-momentum tensor associated to the Yang--Mills
equation, namely 
\[
\EM_{\alpha \beta} (A)=   2 \la \tensor{F}{_{\alp}^{\gmm}}, F_{\beta \gamma}\ra
- \frac12 \met_{\alpha \beta} \la F_{\gamma \delta}, F^{\gamma \delta}\ra.
\]
If $A$ solves the Yang--Mills equation \eqref{ym} then $\EM_{\alpha \beta}$ is divergence free,
\begin{equation}\label{divT}
\partial^\alpha \EM_{\alpha \beta} = 0.
\end{equation}
Integrating this for $\beta = 0$ yields the \emph{conserved energy}
\begin{equation}\label{energy}
\nE(A) = \nE_{\set{t} \times \bbR^{4}}(A) = \int_{\set{t} \times \bbR^{4}} \EM_{00} \, \ud x = \int_{\set{t} \times \bbR^{4}} \frac{1}{2} \brk{F_{jk}, F^{jk}} + \brk{F_{0j}, \tensor{F}{_{0}^{j}}} \, \ud x,
\end{equation}
which is constant in time.
Here we are using the convention of using roman indices $j, k, \ldots$ for the spatial indices $\set{1, \ldots, 4}$.
For a general subset $U \subset \bbR^{4}$, we define the local energy in $U$ to be 
\begin{equation*}
\nE_{\set{t} \times U}(A) = \int_{\set{t} \times U} \EM_{00} \, \ud x.
\end{equation*}

\subsubsection{Symmetries}

The  group of symmetries for the the Yang--Mills equation
play a key role in our analysis. Its components are as follows:

\begin{enumerate}

\item Translations, both in space and in time;

\item the Lorentz group of linear coordinate changes;

\item the scaling group, 
\[
A(t,x) \to \lambda A(\lambda t,\lambda x).
\]
\end{enumerate}
The conserved energy functional $\nE$ is invariant with respect to scaling precisely in dimension
$4+1$. For this reason we call the $4+1$ problem \emph{energy critical};  this is one of the
motivations for our interest in this problem.

\subsubsection{Gauge invariance and Yang--Mills solutions}
In order to study the Yang--Mills equation as a well-defined evolution
in time, we first need to address its gauge invariance. 
Given a $\G$-valued function $O$ on $\bbR^{1+4}$, we introduce the
notation
\begin{equation*}
  O_{; \alp} = \rd_{\alp} O O^{-1}.
\end{equation*}
Such a function $O$ induces the gauge transformation
   \[
A_\alpha \longrightarrow \calG(O) A := Ad(O) A_\alpha - O_{;\alp},
\]
under which equation \eqref{ym} is invariant.
In order to uniquely 
determine the solutions to the Yang--Mills equation, one needs to add an additional 
set of constraint equations which uniquely determine a gauge. This 
procedure is known as {\em gauge fixing}.  

The choice of a gauge plays a central role in the study of the
Yang--Mills equation.  There are multiple interesting classical gauge
choices, e.g.~the Lorenz gauge, the temporal gauge and the Coulomb
gauge. Neither of these is well-suited for the global (in spacetime) large data problem, and
a main goal of our first paper \cite{OTYM1} is to introduce a better
alternate gauge choice, namely the 
 \emph{caloric gauge}. We briefly return to the issue of gauge choice in Section~\ref{subsec:results}, and then give a more detailed discussion in Section~\ref{sec:review}.

\subsubsection{Initial data sets.}
In order to consider the Yang--Mills problem as an evolution equation
we need to consider its initial data sets. An \emph{initial data set} for
\eqref{ym} is a pair of $\g$-valued $1$-forms $(a_{j}, e_{j})$  on $\bbR^{4}$. 
We say that $(a_{j},e_{j})$ is the initial data set for a Yang--Mills solution $A$ if
\begin{equation*}
	(A_{j}, F_{0 j}) \vert_{\set{t=0}} = (a_{j}, e_{j}).
\end{equation*}
Note that \eqref{ym} imposes the condition that the following
equation be true for any initial data for \eqref{ym}:
\begin{equation} \label{eq:YMconstraint}
	\covD^{j} e_{j} = 0.
\end{equation}
Here, $\covD^j$ denotes the covariant derivative with respect to the $a_j$
connection.  This equation is the \emph{Gauss} (or the
\emph{constraint}) \emph{equation} for \eqref{ym}.
In what follows, we denote by $f = f_{ij}$ the curvature of $a$. We refer to Section~\ref{subsec:notation} for the notation $H^{1}(\calO)$, $H^{1}_{loc}(\calO)$ etc.~concerning function spaces.

\begin{definition}
a) A \emph{regular initial data set}  for the Yang--Mills equation  
is a pair $(a_j,e_j) \in H^N_{loc} \times H^{N-1}$, $N \geq 2$, also with $f \in H^{N-1}$,
which has finite energy and satisfies the constraint equation \eqref{eq:YMconstraint}.

b) A \emph{finite energy initial data set}  for the Yang--Mills equation  
is a pair $(a_j,e_j) \in H^1_{loc} \times L^2$, with $f \in L^2$, and
which satisfies the constraint equation \eqref{eq:YMconstraint}.
\end{definition}

We remark that the family of regular initial data sets is dense 
in the class of finite energy data. This is not entirely trivial due to
the nonlinear constraint equation.

\subsubsection{Yang--Mills solutions}  

Due to the gauge invariance properties, we need to be more careful than usual about what we call
a solution to the hyperbolic Yang--Mills equation:

\begin{definition}
a)   Let $N \geq 2$. A \emph{regular solution} to the Yang--Mills equation in an open set $\calO \subset \R^{1+4}$  
is a connection $A$ in $\calO$ obeying $(A, \rd_{t} A) \in C_{t} H^{N}_{loc} \times C_{t} H^{N-1}_{loc}(\calO)$, whose curvature satisfies $F \in C_{t} H^{N-1}_{loc}(\calO)$
and which solves the equation \eqref{ym}.

b) A \emph{finite energy solution} to the Yang--Mills equation in the open set $\calO$
is a connection $A$ obeying $(A, \rd_{t} A) \in C_{t} H^1_{loc} \times C_{t} L^{2}_{loc}(\calO)$, whose curvature satisfies $F \in C_{t} L^2(\calO)$
and which is the limit of regular solutions in this topology.
\end{definition}

We carefully remark that this definition does not require a gauge choice. Hence, at this point
solutions are still given by equivalence classes.  Corresponding to the 
above classes of solutions, we have the classes of gauge transformations
which preserve them:

\begin{definition} \label{def:gt}
a) Let $N \geq 2$. A \emph{regular gauge transformation} in an open set $\calO \subset \bbR^{1+4}$ is 
 is a map 
\[
O: \calO \to \G
\]
with the following regularity properties:
\[
(O_{;t, x}, \rd_{t} O_{;t, x}) \in C_t H^{N+1}_{loc} \times C_t H^{N}_{loc} (\calO).
\]

b) An \emph{admissible gauge transformation} in an open set $\calO \subset
\bbR^{1+4}$ is a similar map with the following regularity properties:
\[
(O_{;t, x}, \rd_{t} O_{;t, x}) \in C_t H^1_{loc} \times C_{t} L^{2}_{loc}( \calO).
\]
\end{definition}

Using this notion we can now talk about gauge-equivalent connections:

\begin{definition}
 Two  finite energy connections $A^{(1)}$ and $A^{(2)}$  in an open set $\calO \subset \bbR^{1+4}$ are gauge equivalent
if there exists an admissible gauge transformation $O$ so that
\[
A^{(2)}_{\alp} = \calG(O) A^{(1)} (= O A^{(1)}_{\alp} O^{-1} - \rd_{\alp} O O^{-1}).
\]
\end{definition}

We list  some simple properties of finite energy connections and admissible gauge transformations in an open set $\calO$
(see \cite{OTYM2.5}):

\begin{itemize}
\item If $A^{(1)}$ and $A^{(2)}$ are finite energy gauge-equivalent
  connections then the bounds for the corresponding gauge transformation $O$ 
depend only on the corresponding bounds for $A^{(1)}$ and $A^{(2)}$.
\item If $A^{(1)}$ and $A^{(2)}$ are regular gauge-equivalent
  connections then the corresponding gauge transformation $O$ is also
  regular, with uniform bounds in terms of $A^{(1)}$, $A^{(2)}$.
\item The family of regular admissible gauge transformations is dense in the family of admissible gauge
transformations.
\item If $A^{(1)}$ and $A^{(2)}$ are gauge-equivalent finite energy
  connections, then $A^{(1)}$ is a finite energy solution to the
  Yang--Mills equation \eqref{ym} if and only if $A^{(2)}$ is.
\item If $A$ is a finite energy connection then its equivalence class 
$[A]$ is closed in the corresponding topology.
\end{itemize}

In terms of local well-posedness, it is easier to work in a gauge. At this point 
we know that (see the more detailed discussion in Section~\ref{sec:review}): 

\begin{enumerate}[label=(\roman*)] 
\item Small data global well-posedness holds in the Coulomb gauge \cite{KT},
caloric gauge \cite{OTYM2} and temporal gauge  \cite{OTYM2}.
\item Large data local well-posedness holds for large caloric data 
in the caloric gauge  \cite{OTYM2} and for arbitrary large data in the temporal gauge
 \cite{OTYM2.5}.
\item Uniqueness of finite energy solutions (up to gauge transformations)  
\cite{OTYM2.5}.
\end{enumerate}

\subsubsection{Topological classes}

The family of finite energy Yang--Mills data sets $(a,e)$  is not a connected topological space
in the above topologies. Instead, they are classified according to 
 their \emph{topological class}, see
\cite{OTYM2.5} and also the discussion in Section~\ref{sec:review}.
The topological class is easily seen to be preserved dynamically for
both regular and finite energy solutions to the hyperbolic Yang--Mills equation.

A special role in the present paper is played by the class $[0]$ of
$0$, whose elements we call \emph{ topologically trivial connections}. These
have the equivalent characterization that they can be described using
$\dot H^1$ connections \cite{OTYM2.5}, see Theorem~\ref{zero-class}
below.  The topologically trivial connections are the subject of both the first two papers
\cite{OTYM1} and  \cite{OTYM2} in our four-paper series, as well as of the  Threshold Theorem below.

For our purposes here, we will use a specific topological invariant, namely the \emph{characteristic number} $\ch$
defined by 
\[
\ch(a) = \int_{\bbR^{4}} - \brk{f \wedge f} = \frac{1}{4} \int_{\bbR^{4}} - \brk{f_{ij}, f_{k \ell}} \, \ud x^{i} \wedge \ud x^{j} \wedge \ud x^{k} \wedge \ud x^{\ell}, 
\]
which depends only on the topological class $[a]$ of $a$.  Two key properties of $\ch$ are that $\ch([0])= 0$
and the pointwise bound
\begin{equation}\label{ch-vs-E}
\abs{\brk{f \wedge f}} \leq \frac{1}{2} \brk{f_{jk}, f^{jk}} \leq \EM_{00}(a),
\end{equation}
which is referred to as the \emph{Bogomoln'yi bound}; see \cite{OTYM2.5} for their proofs.

In the case $\G= SU(2)$ the topological class of $a$ is fully described by the characteristic number $\ch$, which is in fact a multiple of the second Chern number $c_2$ computed from $a$. The Chern number $c_{2}$ turns out to be an integer,
and each such integer defines a connected component in the space of
finite energy connections in $\R^4$.  For a general Lie group $\G$ the characteristic number
$\ch$  of a connection $a$ provides only a partial description of the topological
class of $a$.

\subsubsection{Harmonic Yang--Mills connections and the 
ground state}

A \emph{harmonic Yang--Mills connection} in $\R^4$ is a $\dot
H^1_{loc}$ connection $a$ which is a critical point for the (static) energy
functional
\begin{equation*}
	\spE(a) = \int_{\bbR^{4}} \frac{1}{2} \brk{f, f}.
\end{equation*}

On the one hand they are the steady states for the  hyperbolic Yang--Mills
flow, and on the other hand they are celebrated objects in geometric analysis with spectacular applications to four-dimensional topology; see \cite{DK}. 

The Euler--Lagrange equation satisfied by $a$ takes the form
\begin{equation*}
	\covD^{\ell} f_{\ell j} = 0,
\end{equation*}
which becomes an elliptic system for $a$ in a suitable gauge (e.g.~Coulomb).

The key elliptic regularity result is as follows:
\begin{theorem} [Uhlenbeck \cite{MR648356, U2}]\label{t:harmonic-reg}
 Harmonic Yang--Mills connections $a \in H^1_{loc}(\R^4)$ are smooth in a suitable gauge. 
More generally, $H^{\frac{d}2}_{loc}$ harmonic Yang--Mills connections in any $d$-dimensional 
Riemannian manifold $(M,g)$ are smooth.
\end{theorem}

As far as the energy of  harmonic Yang--Mills connections and the energy in different topological classes 
is concerned, the  key properties are as follows, see \cite{OTYM2.5}:

 \begin{theorem} \label{t:gs} Let $\G$ be a noncommutative compact Lie
    group. Let
    \begin{equation*}
      \Egs = \inf \set{\spE(a) : \hbox{$a$ is a nontrivial harmonic Yang--Mills connection on a $\G$-bundle on $\bbR^{4}$}}.
    \end{equation*}
    Then the following statements hold.
    \begin{enumerate}[labelindent=.3in, leftmargin=!]
    \item There exists a nontrivial harmonic Yang--Mills connection
      $a$ so that $\spE(a) = \Egs < \infty$.
    \item Let $a$ be any nontrivial harmonic Yang--Mills
      connection. Then either $\spE(a) \geq 2 \Egs$, or
      \begin{equation*}
	|\ch| = \spE(a) \geq \Egs.
      \end{equation*}
    \end{enumerate}
  \end{theorem}
  This result is a combination of classical results \cite{ADHM, BPST, DK} concerning energy minimizing solutions within a topological class (called \emph{instantons}), as well as a recent energy lower bound for the non-minimizing solutions proved by Gursky--Kelleher--Streets \cite{GKS}. For a derivation, see \cite[Section~6]{OTYM2.5}. When $\bfG = SU(2)$, the first instanton $a$ is given explicitly by the classical construction of BPST/ADHM \cite{ADHM, BPST}; we refer to \cite[Chapter~3]{DK} for an exposition.
  
As a corollary, Theorem~\ref{t:gs} shows that in the class of topologically trivial connections,
harmonic Yang--Mills connections must have energy at least $2\Egs$. 
Based on this, we will call \emph{subthreshold data/solution} any 
topologically trivial hyperbolic Yang--Mills  data/solution with 
energy below $2 \Egs$.

\subsection{The main results} \label{subsec:results}

We consider the Cauchy problem for the hyperbolic Yang--Mills equation \eqref{ym} with
finite energy data $(a,e)$. As discussed earlier, this problem is known to be locally well-posed \cite{OTYM2, OTYM2.5} for large data and 
globally well-posed for small data \cite{KT}. Here we are interested in the global large data problem,
and we seek to address the following two questions:
\begin{itemize}
\item {\it Global well-posedness};
\item {\it Scattering of the solution}.
\end{itemize}

Preliminary remarks on each point in relation to the notion of the topological class of solutions are in order.

\medskip
\emph{Global well-posedness.}  Because of the finite speed of
propagation and the small data result, a classical argument shows that
at the blow-up time $T$ we must have energy concentration in a backward light cone
centered at a point $(T,X)$,
\[
\underline{C}^{(T, X)} = \{ (t, x) \in \bbR^{1+4} : |x-X| <T-t\}.
\]
In particular, the question of global well-posedness is of local nature, i.e., has nothing to do with the topological class of the initial data.

\medskip

\emph{Scattering.} In a classical sense, a solution $A$ for the Yang--Mills equation would be
 scattering if as $t$ approaches infinity, $A(t)$ approaches a free wave. Such a definition is
unrealistic in our situation. In the first place, it is gauge-dependent. Secondly, the small data result
in \cite{KT} shows that, even in a favorable gauge, classical scattering cannot occur, and instead
one needs to consider some form of \emph{modified scattering}. Even so, there is no chance 
of scattering unless the solution $A$ is topologically trivial; this is due to the fact that 
any solution which decays in a scale invariant $L^p$ norm for $p > 2$ must be topologically 
trivial. We refer to Remark~\ref{rem:scat} for a description of our notion of scattering.

\medskip

We now present our main results, which are divided into two classes.
The first consists of a gauge-independent bubbling off result.
In a nutshell, it asserts that time-like energy concentration implies soliton bubbling off. To state 
it, we need some notation. Given a backward (resp. forward) light cone 
\begin{align*}
	{}^{(T, X)}\underline{C} &= \{(t, x) \in \bbR^{1+4} : |x-X| < T-t\} \\
\big( \hbox{resp. } {}^{(T, X)}C &= \{(t, x) \in \bbR^{1+4} : | x- X | < t - T \} \big),
\end{align*}	
we introduce the time slices
\begin{align*}
	{}^{(T, X)}S_{t} &= {}^{(T, X)} \underline{C} \cap (\set{t} \times \bbR^{4}) \\
	\big( \hbox{resp. } {}^{(T, X)} S_{t} &= {}^{(T, X)} \underline{C} \cap (\set{t} \times \bbR^{4}) \big), 
\end{align*}
and for $0 < \gmm < 1$, the time-like cone
\begin{align*}
{}^{(T, X)} \underline{C}_{\gamma} &= \{  (t, x) \in \bbR^{1+4} : |x-X| < \gamma(T-t)\} \\
	\big( \hbox{resp. } {}^{(T, X)} C_{\gamma} &= \{ (t, x) \in \bbR^{1+4} : |x - X| < \gamma (t - T) \big).
\end{align*}	
When the tip $(T, X)$ coincides with the spacetime origin, we omit the superscript $(T, X)$ and write $C = {}^{(0, 0)} C$, $S_{t} = {}^{(0, 0)} S$, $C_{\gmm} = {}^{(0, 0)} C_{\gmm}$ etc.

For any future time-like vector, which in general takes the form $(1, v)$ with $\abs{v} < 1$, we denote by $L_{v}$ the Lorentz transformation\footnote{More concretely, when $v \neq 0$, $L_{v}$ is the linear transformation on $\bbR^{1+4}$ that preserves $\met$, maps $(1, v)$ to $(\sqrt{1 - \abs{v}^{2}}, 0)$ and equals the identity in $\mathrm{span}\set{(1, 0), (1, v)}^{\perp}$. When $v = 0$, $L_{0}$ is simply the identity.} with velocity $v$. 

Then we have:

\begin{theorem}[Bubbling Theorem] \label{t:bubble-off}
a) Let $A$ be a finite energy Yang--Mills connection which blows up in finite time at $(T,X)$.
Assume in addition that for some $0 < \gamma < 1$ we have
\begin{equation} 
\limsup_{t \nearrow T} \nE_{{}^{(T, X)} \underline{C}_{\gamma} \cap {}^{(T, X)} S_{t}}(A) > 0.
\end{equation}
Then there exists a sequence of points ${}^{(T, X)} \underline{C} \ni (t_n,x_n) \to (T,X)$ and scales $r_n > 0$ with the 
following properties:
\begin{enumerate}
\item Time-like concentration,
\[
\limsup_{n\to \infty} \frac{x_n-X}{|t_n-T|}  = v, \qquad \hbox{ for some } |v| < 1 .
\]

\item Below self-similar scale, 
\[
\limsup_{n\to \infty} \frac{r_n}{|t_n-T|} = 0.
\]

\item Convergence to soliton:
\[
\lim_{n \to \infty} r_{n} \calG(O_n) A(t_{n} + r_{n} t, x_{n} + r_{n} x) = L_v Q(t,x) 
\quad \hbox{in} \ H^1_{loc}([-1/2, 1/2] \times \bbR^{4}) 
\]
for some sequence of admissible gauge transformations $O_n$ and finite energy harmonic Yang--Mills connection $Q$. Here, $L_{v}$ is the Lorentz transformation with velocity $v$.
\end{enumerate}

b) Let $A$ be a finite energy Yang--Mills connection which is global forward in time.
Assume in addition that for some $0 < \gamma < 1$ we have
\begin{equation} 
\limsup_{t \nearrow \infty} \nE_{C_{\gamma} \cap S_{t}}(A) > 0,
\end{equation}
where we recall that $C = {}^{(0, 0)} C$, $S_{t} = {}^{(0, 0)} S_{t}$, $C_{\gmm} = {}^{(0, 0)} C_{\gmm}$ etc\footnote{Since this part concerns the limit $t \to \infty$, the precise choice of the tip $(0, 0)$ is irrelevant; any choice leads to an equivalent statement.}.
Then there exists a sequence of points $C \ni (t_n,x_n) \to \infty$ and scales $r_n > 0$ with the 
following properties:
\begin{enumerate}
\item Time-like concentration,
\[
\limsup_{n\to \infty} \frac{x_n}{t_n}  = v, \qquad \hbox{ for some } |v| < 1 .
\]

\item Below self-similar scale, 
\[
\limsup_{n\to \infty} \frac{r_n}{t_n} = 0 .
\]

\item Convergence to soliton:
\[
\lim_{n \to \infty} r_{n} \calG(O_n) A(t_{n} + r_{n} t, x_{n} + r_{n} x) = L_v Q(t,x) 
\quad \hbox{in} \ H^1_{loc}([-1/2, 1/2] \times \bbR^{4}) 
\]
for some sequence of admissible gauge transformations $O_n$ and finite energy harmonic Yang--Mills connection $Q$. Here, $L_{v}$ is the Lorentz transformation with velocity $v$.
\end{enumerate}
\end{theorem}

Next, we turn to the second class of main results, which concern global well-posedness and scattering properties of \eqref{ym}. 
For this, we need to briefly introduce our gauge choices:

\medskip 
\emph{Caloric gauge.} This is our main choice of gauge, in which we have the strongest gauge-dependent control of solutions. 
We say that a connection $a$ on $\bbR^{4}$ is in \emph{caloric gauge} if its Yang--Mills heat flow
\begin{equation*}
	\rd_{s} A_{j}(x, s) = \covD^{\ell} F_{\ell j}(x, s), \qquad A_{j}(x, s=0) = a_{j}(x)
\end{equation*}
exists globally in heat-time $s$ and $\lim_{s \to \infty} A(s) = 0$. Denoting by $\calC$ the manifold of finite energy caloric connections, and by $T^{L^{2}} \calC$ the completion of its tangent space in $L^{2}$, a solution to the Yang--Mills equation in the caloric gauge can be interpreted as a continuous curve $(A_{x}, \rd_{t} A_{x})(t)$ in $T^{L^{2}} \calC$ (see \cite{OTYM1} and Section~\ref{subsec:OTYM1}).

The Yang--Mills equation written in this gauge has a favorable structure, akin to the classical Coulomb gauge. But in contrast to the Coulomb gauge, the caloric gauge may be imposed for all subthreshold data (to be discussed below), making it a natural setting for the Threshold Theorem.

\medskip 
\emph{Temporal gauge.} This is a classical gauge defined by the condition
\begin{equation*}
	A_{0} = 0,
\end{equation*}
which plays an auxiliary role in our work. The structure of \eqref{ym} in this gauge is less favorable, but nevertheless it has the advantage of respecting causality (i.e., finite speed of propagation) of \eqref{ym}.

Direct analysis of \eqref{ym} in the temporal gauge at energy regularity is fraught with difficulties; however, we observe a suitable structure in the caloric gauge, which allows us to transfer some (but not all) bounds to the temporal gauge. These bounds are enough to establish small energy global well-posedness, which can then be turned into large data local well-posedness in temporal gauge by causality (see \cite{OTYM2.5} and Section~\ref{subsec:OTYM2.5}). This result provides a suitable setting for considering evolution of arbitrary finite energy data, albeit with more indirect control. 

\medskip
We refer to the beginning of Section~\ref{sec:review} for a further discussion of various gauges that arise in our work.

\medskip

We now present the Threshold Theorem, which asserts global well-posedness and scattering for initial data with energy below a sharp threshold. In view of existence of solitons, which are counterexamples for scattering, the threshold may first appear to be the ground state energy $\Egs$. However, as we aim for scattering, we would need to limit ourselves to the class of topologically trivial connections, in which the ground state energy is $2\Egs$ by Theorem~\ref{t:gs}.
Thus our result  is as follows:

\begin{theorem}[Threshold Theorem]\label{t:threshold}
The Yang--Mills equation \eqref{ym} 
is globally well-posed in the caloric gauge for all topologically trivial initial data
below the energy threshold $2\Egs$ and the corresponding solutions scatter  in the following sense:

a) (Regular data) For regular data $(a_{j}, b_{0j}) \in T^{L^{2}} \calC \cap \dot{\H}^{N}$, then there exists a unique
global regular caloric solution $(A_j,\partial_0 A_j) \in C(\R, T^{L^{2}} \calC \cap \dot{\H}^{N})$, also
with $(A_0,\partial_0 A_0) \in C(\R,  \dot{\H}^{1} \cap \dot{\H}^{N})$, which has a Lipschitz
dependence on the initial data locally in time in the $\dot{\H} \cap \dot{\H}^N$ topology.

b) (Rough data) The flow map admits an extension
\[
T^{L^2} \calC  \ni (a_{j}, b_{j}) \to (A_\alpha, \partial_t A_\alpha) \in C(\R, T^{L^2} \calC )
\]
and which is continuous in the $\H \cap \dot \H^\sgm$ topology for $\sgm <1$ and close to $1$.

c) (Weak Lipschitz dependence) The flow map is globally Lipschitz in the 
$\dot \H^\sgm$ topology for $\sgm < 1$, close to $1$.

d) (Scattering) The $S^1$ norm of $A$ is finite. More precisely,
\begin{equation}
\| A_x\|_{S^1} + \nrm{\nb A_{0}}_{\ell^{1} L^{2} \dot{H}^{\frac{1}{2}}}< \infty.
\end{equation}
\end{theorem}
Here, $\dot{\H}^{\sgm} = \dot{H}^{\sgm} \times \dot{H}^{\sgm-1}$. For the norm $S^{1}$, see Remark~\ref{rem:scat}. The norm $\ell^{1} L^{2} \dot{H}^{\frac{1}{2}}$ is defined as $\nrm{u}_{\ell^{1} L^{2} \dot{H}^{\frac{1}{2}}} = \sum_{k} \nrm{P_{k} u}_{L^{2}_{t} \dot{H}^{\frac{1}{2}}_{x}}$; see Section~\ref{subsec:notation} below for our notation and conventions. 
\begin{remark} \label{rem:cal}
The preceding theorem is stated for initial data $(a, b) \in T^{L^{2}} \calC$ which are already in the caloric gauge. However, by our results on the Yang--Mills heat flow (in particular, the corresponding Threshold Theorem), any topologically trivial gauge covariant Yang--Mills data set $(\ta, e) \in \dot{H}^{1} \times L^{2}$ with energy below $2 \Egs$ admits a gauge-equivalent caloric data set $(a, b)$, with appropriate dependence properties; see Section~\ref{subsec:OTYM1} below.
\end{remark}

\begin{remark} \label{rem:scat}
  The $S^1$ norm represents, with only minor changes, the same combination
  of Strichartz, $X^{s,b}$ and null frame norms previously used in the
  study of the Maxwell-Klein-Gordon equation \cite{KST, OT2} and the small
  data problem for Yang--Mills in \cite{KT}. The $S^1$ bound on $A_x$ implies a
host of other dispersive bounds in the caloric gauge, including Strichartz bounds, renormalizability
property, elliptic bounds for $A_{0}$ etc. In particular, finiteness of the $S^{1}$ norm of $A_{x}$ can
  be viewed as a scattering statement, as it shows that the caloric
  solutions decay in Strichartz and other norms. 
\end{remark}

Our final result, which both extends and complements the Threshold Theorem, allows for data which are either topologically nontrivial,
or are topologically trivial but above the $2\Egs$ threshold.
It aims to establish the full dichotomy between the bubbling off property on the one hand, 
and the global well-posedness and scattering on the other:

\begin{theorem}[Dichotomy Theorem]\label{t:no-bubble}
The Yang--Mills equation \eqref{ym} 
is locally well-posed in the temporal gauge for arbitrary finite energy data. Further, one of the following two properties must hold for the forward maximal solution:

a) The solution is topologically trivial, global, and scatters at infinity $(t = \infty)$.

b) The solution bubbles off a soliton, in the sense that either
\begin{enumerate}
\item it blows up in finite time and the conclusion of Theorem~\ref{t:bubble-off}(a) holds; or
\item it exists globally (forward in time) and the conclusion of Theorem~\ref{t:bubble-off}(b) holds.
\end{enumerate}
\end{theorem}
Of course, by time reversibility, the same conclusion holds backward in time as well.

One can view the first two theorems as corollaries of this last
result, modulo the different gauge assumptions. However, we prefer to state them separately because each of
them represent key and largely disjoint steps in the the proof of this
last result. In addition, the Threshold Theorem represents a long
sought after goal in this field.

Some further comments are in order concerning the scattering property
in the first part of the last theorem.  As discussed in Remark~\ref{rem:scat}, in the context of
subthreshold solutions scattering means that solutions are global in
the caloric gauge with a bounded $S^1$ norm.  As it turns out, here scattering carries almost exactly the
same meaning. Precisely, we show that for large enough $T$, the
solution admits a caloric representation on the time interval
$[T,\infty)$, which has a finite $S^{1}$ norm.  One consequence of this is
that \emph{scattering solutions must always be topologically trivial}.

Also, as far as the soliton bubbling off property is concerned, blow
up solutions with this property are known to exist. The constructions in \cite{KST2, MR2929728} give such examples\footnote{The constructions in \cite{KST2, MR2929728} are for $\G = SU(2)$ in the first topological class with $c_{2} = 1$ and with energies close to $\Egs$, but a straightforward gluing argument at infinity produces the desired topologically trivial solutions.} whose energies may be arbitrarily close to the threshold $2 \Egs$, and the recent work \cite{JeLa} provides\footnote{We note that \cite{JeLa} moreover gives a complete classification of possible dynamics at the threshold energy under equivariance symmetry.} a blow up solution at exactly the threshold energy $2 \Egs$.
These solutions concentrate at the blow-up point
following a rescaled soliton profile, where the soliton scale differs
logarithmically from the self-similar scale. Similarly, solutions where bubbling occurs at infinity also exist. For Yang--Mills in the topologically trivial class, this was achieved in \cite{Je}; interestingly, the nonscattering solution in \cite{Je} has exactly the threshold energy $2 E_{GS}$ (see also \cite{JeLa}). Such solutions have been obtained
for related models such as the energy critical wave maps equation in the one-bubble case, see \cite{wm-blow-far}.

\subsection{A brief history and broader
  context} \label{subsec:literature} \ A natural point of view is to
place the present papers and results within the larger context of
geometric wave equations, which also includes wave maps (WM),
Maxwell--Klein--Gordon (MKG) and Einstein equations. Two common features
of all these problems are that they admit a Lagrangian formulation,
and have some natural gauge invariance properties. Following are some
of the key developments that led to the present work.

\medskip

{\em 1. The null condition.} A crucial early observation in the study
of both long range and low regularity solutions to geometric wave
equations was that the nonlinearities appearing in the equations have
a favorable algebraic structure, which was called {\em null
  condition}, and which can be roughly described as a cancellation
condition in the interaction of parallel waves. In the low regularity
setting, this was first explored in work of Klainerman and
Machedon~\cite{KlMa0}, and by many others later on.
 
\medskip

{\em 2. The $X^{s,b}$ spaces.}  A second advance was the introduction
of the $X^{s,b}$ spaces\footnote{The concept, and also the notation,
  is due to Bourgain, in the context of KdV and NLS type problems.},
also first used by Klainerman and Machedon~\cite{KlMa0} in the context of
the wave equation.  Their role was to provide enough structure in
order to be able to take advantage of the null condition in bilinear
and multilinear estimates. Earlier methods, based on energy bounds,
followed by the more robust Strichartz estimates, had proved inadequate 
to the task.

\medskip

{\em 3. The null frame spaces.}  To study nonlinear problems at
critical regularity one needs to work in a scale invariant
setting. However, it was soon realized that the homogeneous $X^{s,b}$
spaces are not even well defined, not to mention suitable for
this. The remedy, first introduced in work of the second
author~\cite{Tat} in the context of wave maps, was to produce a better
description of the fine structure of waves, combining frequency and
modulation localizations with adapted frames in the physical space.
This led to the {\em null frame spaces}, which played a key role
in subsequent developments for wave maps. 

\medskip

{\em 4. Renormalization.}  A remarkable feature of all semilinear
geometric wave equations is that while at high regularity
(and locally in time) the nonlinearity is perturbative, this is no longer 
the case at critical regularity. Precisely, isolating the 
non-perturbative component of the nonlinearity, one can see that this 
is of  paradifferential type; in other words, the high frequency waves 
evolve on a variable low frequency background.  To address this 
difficulty, the idea of Tao~\cite{Tao2}, also in the wave maps context, was to 
{\em renormalize} the paradifferential problem, i.e., to find a suitable 
approximate conjugation to the corresponding constant coefficient problem. In the case of wave maps, the conjugating operator is essentially a gauge transform (i.e., a Lie group-valued function), while in the case of Maxwell--Klein--Gordon and Yang--Mills one needs a Lie group-valued pseudo-differential operator; see \cite{RT, KS, KST, KT} and the discussion below.

\medskip

{\em 5. Induction of energy.} The ideas discussed so far seem to suffice
for small data critical problems. Attacking the large data problem
generates yet another range of difficulties. One first step in this
direction is Bourgain's {\em induction of energy} idea \cite{Bour},
which is a convenient mechanism to transfer information to higher and
higher energies. We remark that an alternate venue here, which
sometimes yields more efficient proofs, is the Kenig--Merle idea
\cite{MR2461508} of constructing {\em minimal blow-up
  solutions}. However, the implementation of this method in problems
which require renormalization seems to cause considerable trouble.
For a further discussion on this issue, we refer to \cite{KriSch}, where this 
method was carried out in the cases of energy critical wave maps into the hyperbolic plane.

\medskip

{\em 6. Caloric gauge.} Another difficulty arising in the context 
of large data solutions is that of finding a good gauge, which at the same time applies 
to large data and at the same time has good analytic properties. The caloric gauge,
used in our work, is a global version of a local caloric gauge previously introduced
by the first author \cite{Oh1, Oh2}, and is based on an idea proposed by Tao~\cite{Tao-caloric}
in the wave maps context.

\medskip

{\em 7. Energy dispersion.} One fundamental goal in the study of large
data problems is to establish a quantitative dichotomy between
dispersion and concentration.  The notion of {\em energy dispersion},
introduced in joint work~\cite{ST1, ST2} of the second author and Sterbenz in
the wave maps context, provides a convenient measure for pointwise
concentration. Precisely, at each energy there is an energy dispersion
threshold below which dispersion wins.  We remark that, when it can be
applied, the Kenig-Merle method \cite{MR2461508} yields more accurate
information; for instance, see \cite{KriSch}. However, the energy dispersion idea,
which is what we follow in the present series of papers, is much easier to
implement in conjunction with renormalization.

\medskip

{\em 8. The frequency gap.}  One obstacle in the transition from small
to large data in renormalizable problems is that the low frequency
background may well correspond to a large solution. Is this fatal to
the renormalized solution? The answer to that, also originating in
\cite{ST1, ST2}, is that may be a second hidden source of
smallness, namely a large {\em frequency gap} between the
high frequency wave and the low frequency background it evolves on.

\medskip

{\em 9. Morawetz estimates (monotonicity formulas).} The outcome of the ideas above is a dichotomy 
between dispersion and scattering on one hand, and very  specific concentration
patterns, e.g., solitons, self-similar solutions on the other hand. The Morawetz estimates,
first appearing in this role in the work of Grillakis~\cite{Gri},
are a convenient and relatively simple tool to identify and, if possible, eliminate such concentration scenarios. 
In the present work, in analogy with the elliptic and parabolic literature, such an estimate is interpreted as a monotonicity formula (Section~\ref{sec:energy}).

\medskip 

We now narrow the scope of discussion, and
review some earlier developments on the wave maps (WM), Maxwell--Klein--Gordon (MKG) and hyperbolic Yang--Mills (YM) equations related to the present paper.

\medskip

{\em (MKG) and (YM) above the scaling critical regularity.}  We start our discussion with
a short and incomplete survey of the (YM) problem above the scaling critical regularity. We also discuss the (MKG) problem, which has been often studied as a simpler model for (YM) with a commutative gauge group.

In the two and three dimensional cases, which are energy subcritical, global
regularity of sufficiently regular solutions was shown in the early
works \cite{MR649158, MR649159}.  These papers in fact
handled the more general Yang--Mills-Higgs system, which includes both (YM) and (MKG) as special cases.  In dimension $d=3$, local well-posedness
in the energy space of (MKG) and (YM) was proved in \cite{KlMa1} and \cite{KlMa2}, respectively.
 In the higher dimensional case $d \geq 4$, an essentially optimal local
well-posedness result for a model problem closely related to (MKG) was
 obtained in \cite{Kl-Tat}. 

However, a new difficulty arises in the large data\footnote{More
  precisely, a suitable scaling critical norm of the connection $A$
  (e.g., $\nrm{A}_{L^{d}_{x}}$) or the curvature $F$ (e.g.,
  $\nrm{F}_{L^{\frac{d}{2}}_{x}}$) is large.} problem for (YM):
Namely, the gauge transformation law is \emph{nonlinear} due to the
noncommutative gauge group. In particular, gauge transformations into the
Coulomb gauge obey a nonlinear elliptic equation, for which no
suitable large data regularity theory is available. Note, in
comparison, that such gauge transformations obey a linear Poisson
equation in the case of (MKG).  In \cite{KlMa2}, where
finite energy global well-posedness of the 3+1 dimensional (YM)
problem was proved, this issue was handled by localizing in spacetime
via the finite speed of propagation to gain smallness, and then
working in local Coulomb gauges\footnote{On the other hand, a closely related spacetime localization approach, but relying on new ``initial data surgery'' techniques, is developed in \cite{OTYM2.5}, which yields an alternative proof of the main result of \cite{KlMa2}.}.  An alternative, more robust approach
without spacetime localizations to the same problem has been put
forth by the first author in \cite{Oh1, Oh2}, inspired by
\cite{Tao-caloric}. The idea is to use an associated geometric flow,
namely the \emph{Yang--Mills heat flow}, to select a global-in-space
Coulomb-like gauge for data of any size.

\medskip

{\em The energy critical (WM) problem. } 
Before turning to the (MKG) and (YM) problems at critical regularity, we briefly recall some recent developments on the wave maps equation (WM), where many of the
methods we implement here have their roots. We confine our discussion
to the energy critical problem in $2+1$ dimensions, which is both the
most difficult and the most relevant to our present paper.  For the small data problem, 
global well-posedness was established in \cite{Tat}, \cite{Tao2},
\cite{Tat2}.  More recently, the \emph{Threshold Theorem} for large data wave maps, 
which asserts that global well-posedness and scattering hold below the ground state energy, 
was proved in \cite{ST1, ST2} in general, and independently in
\cite{KriSch} and \cite{Tao:2008wo, Tao:2008wn, Tao:2008tz, Tao:2009ta, Tao:2009ua} for specific targets (namely the hyperbolic
space). See also \cite{LO} for a sharp refinement in the case of a two-dimensional target, 
taking into account an additional topological invariant (namely, the degree of the wave map), in analogy with the refined threshold $2 \Egs$ in our work. 
Our present strategy was strongly influenced by \cite{ST1, ST2}.

For the energy critical (WM), in the important case of spherical targets, we also note the recent development due to Grinis \cite{MR3627409}, which says that along a well-chosen sequence of times, all time-like energy concentration must be in the form of a superposition of rescaled solitons. Our Bubbling Theorem (Theorem~\ref{t:bubble-off}) is a first step for proving an analogous result for (YM). In \cite{DJKM1}, this was complemented with a decay of the energy near the cone when the total energy is sufficiently close to the ground state.

\medskip

{\em The (MKG) and (YM) problems at critical regularity.} 
Next, we discuss the (YM) problems at critical regularity. As before, we simultaneously consider the corresponding problems for (MKG), which is a simpler commutative analogue of (YM). 

Before discussing history, let us clarify a key structural difference
between (WM) on the one hand and (MKG), (YM) on the other, whose
understanding is crucial for making progress on the latter two
problems.  Roughly speaking, all three equations can be written in a
form where the main `dynamic variables', which we denote by $\phi$,
obey a possibly nonlinear gauge covariant wave equation $\Box_{A} \phi
= \cdots$, and the associated curvature $F[A]$ is determined by
$\phi$.  In the case of (WM), this dependence is simply algebraic,
whereas for (MKG) and (YM) the curvature $F[A]$ obeys a wave equation
with a nonlinearity depending on $\phi$. This difference manifests in
the renormalization procedure for each equation: For (WM) it suffices
to use a physical space gauge transformation, whereas for (MKG) and
(YM) it is necessary to use a microlocal (more precisely,
\emph{pseudo-differential}) gauge transformation that exploits the
fact that $A$ solves a wave equation in a suitable gauge.

The first such renormalization argument appeared in \cite{RT}, in
which global regularity of (MKG) for small critical Sobolev data was
established in dimensions $d \geq 6$.  This work was followed by a
similar high dimensional result for (YM) in \cite{KS}.  The small data
result in the energy critical dimension $4+1$ was obtained in
\cite{KST}.  Finally, the large data result for (MKG) in dimension
$4+1$ was proved by the authors in \cite{OT1, OT2, OT3} and independently by \cite{KL}. 
Although the implementation differs in many places, the outline of the three
papers \cite{OT1, OT2, OT3} is broadly followed in the present work.  In particular we
borrow a good deal of notation, ideas and estimates from both
\cite{KST} and \cite{OT1, OT2, OT3}. On the other hand, we remark that \cite{KL} followed 
the strategy of \cite{KriSch}. For the (YM) problem in $4+1$ dimensions, 
the small data global result was only recently proved in \cite{KT}, which is another direct predecessor the present
work. 

We conclude with a remark on differences between (MKG) and (YM).
The issue of noncommutative gauge group for the large data problem
has already been discussed. Another important difference between (MKG)
and (YM) in $4+1$ dimensions is that the latter problem admits
nontrivial steady states (i.e., harmonic Yang--Mills connections). These solutions are known to lead to a finite time blow up;
see \cite{KST2,MR2929728}, so for (YM) one must prove the Threshold Theorem, instead of a unconditional result as in (MKG).
Finally, (YM) is more `strongly
coupled' as a system compared to (MKG), in the sense that the
connection $A$ itself obeys a covariant wave equation. This feature
 necessitates  a more involved renormalization procedure
compared to (MKG).

\medskip

{\em Other related works.} In related developments, one should also note the works
\cite{BH, BH1} on the closely related cubic Dirac equation, as
well as the massive Dirac--Klein--Gordon system, as well as \cite{GO} on
the {M}axwell--{D}irac equation and \cite{Gav} on the massive Maxwell--Klein--Gordon system.

\subsection{Notation, conventions and preliminaries} \label{subsec:notation}
Here we collect more notation and conventions used in the remainder of this paper.
\subsubsection*{Asymptotic notation}
\begin{itemize}
\item $A \aleq B$ and $A = O(B)$ both mean $A
  \leq C B$ for some constant $C > 0$. The dependence of $C$
  on various parameters is specified by subscripts. When $A \aleq B$ and $B \aleq A$, we write $A \aeq B$.
\end{itemize}

\subsubsection*{Tensor calculus}
\begin{itemize}
\item We employ the usual index notation, the Einstein convention of summing up repeated upper and lower indices. We use greek indices, such as $\alp, \bt, \gmm, \ldots$, for all coordinates $x^{0} = t, x^{1}, x^{2}, x^{3}, x^{4}$, and latin indices, such as $i, j, k, \ell, \ldots$, for the spatial coordinates $x = (x^{1}, \ldots, x^{4})$.
\item Given a metric (which is usually the Minkowski metric $\bfm$ on $\bbR^{1+4}$, unless otherwise stated), we write $\covnb$ for the associated Levi-Civita connection. Tensorial indices are raised and lowered using the metric.
\end{itemize}

\subsubsection*{Exterior differential calculus}
\begin{itemize}\item The wedge product $\wedge$ and the differential $\ud$ for differential forms are defined in the usual way. A $k$-form $\omg$ can be viewed as a $k$-covariant tensor; we have $\omg = \sum_{\bt_{1} < \cdots < \bt_{k}} \omg_{\bt_{1} \ldots \bt_{k}} \ud x^{\bt_{1}} \wedge \cdots \wedge \ud x^{\bt_{k}}$ where $\omg_{\bt_{1} \ldots \bt_{k}}$ is the coordinate expression of $\omg$ as a tensor. 
\item $\iota_{X}$ is the interior product with a vector field $X$, i.e., $(\iota_{X} \omg)_{\bt_{1} \ldots \bt_{k-1}} = X^{\alp} \omg_{\alp \bt_{1} \ldots \bt_{k-1}}$, and $\calL_{X}$ is the Lie derivative with respect to $X$. 
\item The pointwise inner product $(\cdot, \cdot)$ of two $k$-forms is defined so that $\set{\tht^{\alp_{1}} \wedge \cdots \wedge \tht^{\alp_{k}}}_{\set{\alp_{1} < \cdots < \alp_{k}}}$ is an orthonormal basis, where $\set{\tht^{\alp}}$ is any orthonormal basis with respect to a given metric on $T^{\ast} \bbR^{d}$ (note that this differs by the usual induced metric for tensors by $k!$). 
\item The Hodge star operator $\star$ is defined so that $\eta \wedge \star \omg = (\eta, \omg) \, \ud \hbox{Vol}$, where $\ud \hbox{Vol}$ is the volume form. The codifferential $\dlt$ is the adjoint of $\ud$ with respect to the $L^{2}$-product $\int \brk{\cdot, \cdot} \, \ud \hbox{Vol}$. The Hodge Laplacian is defined to be $-\lap = \dlt \ud + \ud \dlt$, so that it agrees with the usual Laplacian $\sum_{j} \rd_{j}^{2}$ for $0$-forms (or functions) on $\bbR^{4}$.
\end{itemize}

\subsubsection*{Subsets of $\bbR^{d}$}
\begin{itemize}
\item For a bounded open set $U \in \bbR^{d}$ and $\lmb > 0$, $\lmb U$ is defined to be rescaling
  of $U$ by the factor $\lmb$ centered at the barycenter of $U$.
  
 \item $B_{R}(x)$ is the $4$-dimensional ball of radius $r$ centered
  at $x$. When $x =0$, we simply write $B_{R} = B_{R}(0)$.
  
\item $\calA_{(R', R)}(x)$ is the $4$-dimensional annulus of inner and outer radii $R'$ and $R$, respectively, centered at $x \in \bbR^{4}$. When $x = 0$, we simply write $\calA_{(R', R)} = \calA_{(R', R)}(0)$.

\item Consider the forward light cone centered at $(0, 0)$:
\begin{equation*}
  C = \set{(t,x) : 0 < t < \infty, \abs{x} < t}.
\end{equation*}
For $t_{0} \in \bbR$ and $I
\subset \bbR$, we define
\begin{align*}
  C_{I} =& \set{(t, x) : t \in I, \abs{x} < t}, &
  \rd C_{I} =& \set{(t, x) : t \in I, \abs{x} = t}, \\
  S_{t_{0}} =& \set{(t, x) : t = t_{0}, \abs{x} < t}, & \rd
  S_{t_{0}} =& \set{(t, x) : t = t_{0}, \abs{x} = t}.
\end{align*}
For $\dlt \in \bbR$, we define the translated cones
\begin{align*}
  C^{\dlt} =& \set{(t,x) : \max\set{0,\dlt} < t < \infty, \abs{x}
    < t-\dlt}.
\end{align*}
The corresponding objects $C^{\dlt}_{I}$, $\rd C^{\dlt}_{I}$,
$S^{\dlt}_{t_{0}}$ and $\rd S^{\dlt}_{t_{0}}$ are defined in the
obvious manner.
\end{itemize}

\subsubsection*{Polar coordinate systems}
\begin{itemize}
\item {\it Hyperbolic polar coordinates on $C \subset \bbR^{1+4}$.} We parametrize the cone $C = \set{(t, x) \in \bbR^{1+4} : \abs{x} < t}$ by $(t, x) = \rho y$, where $\rho = \sqrt{t^{2} - \abs{x}^{2}} > 0$ and $y \in \bbH^{4} := \set{(t, x) \in \bbR^{1+4} : t^{2} - \abs{x}^{2} = 1, \, t > 0}$. We write $\ud y$ for the volume form on $\bbH^{4}$.

\item {\it Polar coordinates on $\bbR^{4}$.} We parametrize $\bbR^{4} \setminus \set{0}$ by $x = r \Tht$, where $r = \abs{x}$ and $\Tht \in \bbS^{3} = \set{x \in \bbR^{4} : \abs{x} = 1}$. 

	Given a covariant tensor $\omg_{\bt_{1} \ldots \bt_{k}}$ (e.g., a $k$-form) on $\bbR^{4} \setminus \set{0}$, we use the schematic notation $\omg_{\Tht \ldots \Tht}(r, \cdot)$ for its pullback onto each constant-$r$ sphere. To formulate the calculus of such objects, \emph{we view each $\omg_{\Tht \ldots \Tht}(r, \cdot)$ as a $1$-form on the unit sphere $\bbS^{3}$}. We write $\smet$ for the metric on the unit sphere $\bbS^{3}$, $\snb$ for the associated Levi--Civita connection, $\scovD = \snb + ad(A)$ for the covariant derivative, $\ud \Tht$ for the volume form and $L^{p}_{\Tht}$ for the Lebesgue spaces with respect to $\ud \Tht$. We write $\tensor{\omg}{^{\Tht}_{\Tht}}$ for the trace with respect to $\smet$ of a covariant tensor; accordingly, $\snb^{\Tht} \omg_{\Tht}$ is the divergence operator with respect to $\smet$ and $\ud \Tht$, and $\scovD^{\Tht} \omg_{\Tht}$ is the covariant divergence for a $\g$-valued $1$-form.
	
	Given a subset $U \subset \bbR^{4} \setminus \set{0}$, we \emph{define} the $\nrm{\omg_{\Tht \ldots \Tht}}_{L^{p}(U)}$ to be the $L^{p}(U)$ norm of $\abs{\omg_{\Tht \ldots \Tht}}$, where $\abs{\omg_{\Tht \ldots \Tht}}^{2} = \smet(\omg_{\Tht \ldots \Tht}, \omg_{\Tht \ldots \Tht})$. Note that $\nrm{\omg_{\Tht \ldots \Tht}}_{L^{p}(\calA_{(R', R)})} \aeq_{R', R} \nrm{\omg_{\Tht \ldots \Tht}}_{L^{p}_{r}((R', R); L^{p}_{\Tht})}$ for any $0 < R' < R$.
\end{itemize}

\subsubsection*{Functions spaces}
\begin{itemize}
\item $\rd$ (without sub- or superscripts) is the
  spatial gradient $\rd = (\rd_{1}, \rd_{2}, \ldots, \rd_{4})$, and
  $\nb$ is the spacetime gradient $\nb = (\rd_{0}, \rd_{1}, \ldots,
  \rd_{4})$. We write $\rd^{(n)}$ (resp. $\nb^{(n)}$) for the
  collection of $n$-th order spatial (resp. spacetime) derivatives,
  and $\rd^{(\leq n)}$ (resp. $\nb^{(\leq n)}$) for those up to order
  $n$.

\item $\dot{W}^{\sgm, p}(\bbR^{d}; V)$ is the homogeneous $L^{p}$-Sobolev space of order $\sgm$ for functions from
  $\bbR^{d}$ into a normed vector space $V$. In the special case $p = 2$, we write $\dot{H}^{\sgm}(\bbR^{d}; V) = \dot{W}^{\sgm, 2}(\bbR^{d}; V)$.
 The inhomogeneous counterparts are denoted by $W^{n, p}(\bbR^{d}; V)$
  and $H^{n}(\bbR^{d}; V)$, respectively. We often suppress $\bbR^{d}$ and $V$ when it is clear from the context.

\item The mixed spacetime norm $L^{q}_{t} \dot{W}^{n, r}_{x}$ of
  functions on $\bbR^{1+d}$ is often abbreviated as $L^{q} \dot{W}^{n,
    r}$.

\item Generally, a function space on an open subset $U \subseteq
  \bbR^{4}$ is defined by restriction, i.e., $\nrm{u}_{X(U)} = \inf
  \set{\nrm{\tilde{u}}_{X} : \tilde{u} \in X, \ \tilde{u}
    \restriction_{U} = u}$. A similar convention applies for a function
  space on an open subset $\calO \subseteq \bbR^{1+4}$.

\item The local function space $X_{loc}(U)$ is defined as
  \begin{equation*}
    X_{loc}(U) = \bigcap_{B_{r}(x): \overline{B}_{r}(x) \subseteq U} X(B_{r}(x)).
  \end{equation*}
\end{itemize}

\subsubsection*{Littlewood--Paley theory, dyadic function spaces and frequency envelopes}
\begin{itemize}
\item $\set{P_{k}}_{k \in \bbZ}$ denotes the usual Littlewood--Paley projections in the variable $x \in \bbR^{4}$.
\item A dyadic function space $X$ is a collection $\set{X_{k}}_{k \in \bbZ}$ of normed spaces on either $\bbR^{4}$ or $\bbR^{1+4}$. Often we use the same space for each $k$, in which case we simply write $X = X_{k}$. We define $\ell^{p} X$ by the norm $\nrm{u}_{\ell^{p}X}^{p} = \sum_{k} \nrm{P_{k} u}^{p}_{X}$, with the usual modification for $p = \infty$. An important example is the $L^{2}$-Sobolev space $H^{\sgm} = \ell^{2} H^{\sgm}$.
\item An admissible frequency envelope $c$ is a sequence $\set{c_{k}}_{k \in \bbZ}$ of positive numbers satisfying $\max\set{\frac{c_{j}}{c_{k}}, \frac{c_{k}}{c_{j}}} \leq C 2^{\dlt_{fe} \abs{k-j}}$ for some constant $C$ depending on $c$ and an absolute constant $\dlt_{fe} > 0$ inherited from \cite{OTYM1, OTYM2}. 
\item We define $\nrm{u}_{X_{c}} = \sup_{k} c_{k}^{-1} \nrm{P_{k} u}_{X_{k}}$. If $\nrm{u}_{X_{c}} \leq 1$, then we say that $c$ is a frequency envelope for $u$ in $X$. 
\end{itemize}

\subsection{Structure of the present paper} 
 The remainder of the paper is structured as follows.

\emph{ Section~\ref{sec:review}.}
 Here we review the main results in the first three papers of the series \cite{OTYM1}, \cite{OTYM2} and \cite{OTYM2.5},
 emphasizing the parts which are needed here.

\emph{ Section~\ref{sec:energy}.}
 This is where we state and prove all the conservation laws and monotonicity formulas that are
used in this paper. We also explore a few consequences of the monotonicity formulas.

\emph{ Section~\ref{sec:cpt}.} We use a ``good gauge'' representation theorem (Theorem~\ref{t:good-gauge}) for large
energy Yang--Mills connections in order to prove a strong local compactness statement,
Theorem~\ref{t:compact}, that we rely on on in our blow-up analysis. 

\emph{ Section~\ref{sec:reg}.} Here we study the regularity of connections that either stationary or self-similar, 
and show that such connections must be gauge equivalent to a smooth connection.
This is akin to elliptic theory for harmonic Yang--Mills connections.

\emph{Section~\ref{sec:self-similar}.} 
We show that there does not exist any nontrivial  self-similar solutions to Yang--Mills with finite  energy,
thus eliminating one of the main potential obstructions to our results.

\emph{Section~\ref{sec:bubble-off}.}  Here we carry out the blow-up
analysis and prove the Bubbling Theorem (Theorem~\ref{t:bubble-off}). 
This proof uses all of the previous five sections.

\emph{Section~\ref{sec:no-null}.}  In this section
 we prove that sharp energy concentration cannot occur near the null cone.
This is critical in order to be able to separate the bubbling-off scenario from the 
scattering, energy dispersed case.

\emph{Section~\ref{sec:proof}.} 
Finally, here we complete both the proof of the Threshold Conjecture, see Theorem~\ref{t:threshold}, and the  
dichotomy result in Theorem~\ref{t:no-bubble}.

\emph{Appendix~\ref{sec:gt}.}
We collect some technical tools needed for our analysis of gauge transformations, especially in Sections~\ref{sec:reg}, \ref{sec:self-similar} and \ref{sec:no-null}.

\addtocontents{toc}{\protect\setcounter{tocdepth}{-1}}
\subsection*{Acknowledgments} 
Part of the work was carried out during the semester program ``New Challenges in PDE''
held at MSRI in Fall 2015. S.-J. Oh was supported by the Miller Research Fellowship from the Miller Institute, UC Berkeley and the TJ Park Science Fellowship from the POSCO TJ Park Foundation. D. Tataru was partially
supported by the NSF grant DMS-1266182 as well as by a Simons
Investigator grant from the Simons Foundation.

\addtocontents{toc}{\protect\setcounter{tocdepth}{2}}

\section{An outline of the first three papers}
\label{sec:review}

Our aim here is to provide a brief outline of the first three papers \cite{OTYM1},\cite{OTYM2} and \cite{OTYM2.5},
to the extent necessary in order to complete the proof of the large data results in the present paper.
For a more comprehensive review of the full series of four papers we instead refer the reader to 
our survey article  \cite{OTYM0}.

Let us take as a starting point of our discussion the following small data result proved earlier  in \cite{KT}:
\begin{theorem} \label{thm:KrTa}
The hyperbolic Yang--Mills equation in $\bbR^{4+1}$   is globally well-posed in the Coulomb gauge for 
all initial data with small energy.   
\end{theorem}

Even before considering well-posedness results for the large data, the
first difficulty one faces is that the Coulomb gauge does not appear to
fully extend to large data in general, and not even to subthreshold data (see Remark~\ref{rem:coulomb} below). For this reason, our first
paper \cite{OTYM1} is devoted solely to the gauge problem; precisely,
inspired by earlier work of Tao~\cite{Tao-caloric} and of the first author
\cite{Oh1, Oh2}, we develop a new gauge for the hyperbolic Yang--Mills problem \eqref{ym},
namely the \emph{caloric gauge}. Using this gauge, the most difficult gauge-dependent analysis of the 
 Yang--Mills equation is carried out in \cite{OTYM2}.  The caloric gauge is the natural setting of our Threshold Theorem (Theorem~\ref{t:threshold}).

\begin{remark}\label{rem:coulomb}
To use the global Coulomb gauge, one would need the solution to the following open problem: the existence of a regular gauge transformation $O$ to Coulomb gauge for a general subthreshold (hence topologically trivial) connection $a$ on $\bbR^{4}$ with a quantitative control on the critical norm $\nrm{O_{;x}}_{\dot{H}^{1}}$. For the interested reader, we refer to \cite[Open~Problem~1.3]{PetRiv} and \cite{YWa}, where similar problems for other critical norms of $O_{;x}$ are studied on closed $4$-manifolds.
\end{remark}

On the other hand, for the large data analysis 
in \cite{OTYM2.5} causality plays a key role, so we shift the (soft part of the)
analysis to the \emph{temporal gauge}. The causality property of the temporal gauge allows us to formulate a local well-posedness theory of the hyperbolic Yang--Mills equation for arbitrary finite energy data (and in particular, in arbitrary topological class), which is the setting for our Dichtomy Theorem (Theorem~\ref{t:no-bubble}). We note, however, that the strong $S$-norm control on the solution is lost in the temporal gauge. 

We summarize the discussion on various gauges so far in the following table:
\begin{table}[h]
\centering
\begin{tabular}{ | c |c|m{5.5em}| m{15em}| } 
 \hline
 Gauge choice & Definition & Appearances & Remarks \\ 
 \hline
 (global) Coulomb & $\rd^{k} A_{k} = 0$ & Thm.~\ref{thm:KrTa} & Requires small initial energy; expect $S$-norm control; no direct usage in this paper. \\ 
 \hline
 Caloric & Definition~\ref{def:caloric} & Thms.~\ref{t:threshold}, \ref{t:no-bubble}${}^{\dagger}$, \ref{t:local}, \ref{t:ED} & Requires trivial topological class and finite caloric size; expect $S$-norm control. \\ 
 \hline
 Temporal & $A_{0} = 0$ & Thms.~\ref{t:no-bubble}${}^{\dagger}$, \ref{t:small-temp}, \ref{t:local-temp}, \ref{t:global-temp} & No restriction on the topological class or energy; no $S$-norm control. \\
 \hline
\end{tabular}
\smallskip
\caption{Gauge choices for the initial value problem. $\dagger$: In Theorem~\ref{t:no-bubble}, the initial value problem is posed in the temporal gauge, but the scattering statement involves the caloric gauge; see Remark~\ref{rem:scat}.} \label{tb:gauge}
\end{table}

Finally, we note that our Bubbling Theorem (Theorem~\ref{t:bubble-off}) is formulated in a gauge-independent fashion\footnote{However, the notion of scattering is formulated with the help of the caloric gauge; see Remark~\ref{rem:scat}.}. Indeed, most of the work in the present paper is carried out in a gauge-covariant fashion, while using the gauge-dependent results in \cite{OTYM2} and \cite{OTYM2.5} at critical junctures.

\subsection{The caloric gauge} \label{subsec:OTYM1}

The goal of the first paper \cite{OTYM1} is to 
\begin{itemize}
\item Introduce the caloric gauge;
\item Show that the caloric gauge is well-defined for all subthreshold data; and
\item Provide a comprehensive formulation of the hyperbolic Yang--Mills equation 
in the caloric gauge which is sufficiently accurate for the subsequent analysis.
\end{itemize}

The caloric gauge is defined using the Yang--Mills heat flow
\begin{equation}\label{ym-heat}
\partial_s A_j = \covD^k F_{jk}, \qquad A_j(s=0) = a_j
\end{equation}
which implicitly assumes the gauge condition $A_s=0$ (which we refer to as the \emph{local caloric gauge}) relative to the
fully covariant formulation of the same equation. This can be naively viewed as 
parabolic system for the curl of $A$ (or the curvature $F$), coupled with a transport equation
for the divergence of $A$. Because these two equations are strongly coupled,
this evolution displays an interesting mix of semilinear and quasilinear features.

Our interest is in initial data  $a$ for which this solution is global, with the curvature $F$
satisfying global parabolic bounds. To capture this, we use the $L^3$ norm of $F$ as 
a control norm, and call it the \emph{caloric size} of $a$,
\begin{equation}
\label{Q-def}
\hM(a) = \int_0^\infty \int_{\R^4} |F|^3 dx ds .
\end{equation}
We note that this is a gauge invariant quantity.
For solutions with finite caloric size, we have the following structure theorem:

\begin{theorem}[{\cite[Corollary~5.14]{OTYM1}}]\label{t:global-heat}
Let $a \in \dot H^1$ be a connection so that $\hM(a) < \infty$. Then 
this solution has the property that the limit
\[
\lim_{s \to \infty} A(s) = a_\infty
\]
exists in $\dot H^1$. Further, the limiting connection is flat,
$f_\infty = 0$, and the map $a \to a_\infty$ is locally Lipschitz in $
\dot H^1$, $H^{N}$ ($N \geq 1$) and $\dot H^1 \cap \dot H^N$ ($N \geq
2$).
\end{theorem}

Next, using the monotonicity formula for the energy, we prove the Dichotomy Theorem for the Yang--Mills heat flow:
\begin{theorem} [{\cite[Theorem~6.1]{OTYM1}}] \label{t:ym-dich}
One of the following two
  properties must hold for the maximally extended $\dot{H}^{1}$ Yang--Mills heat flow:

  a) The solution is global and its caloric size is finite;

 b) The solution bubbles off a nontrivial harmonic Yang--Mills connection, either
\begin{enumerate}
\item at a finite blow-up time $s < \infty$, or
\item at infinity $s = \infty$.
\end{enumerate}
\end{theorem}

Combined with topological triviality of $\dot{H}^{1}$ connections, we are led to the Threshold Theorem, with the identical threshold as in the hyperbolic case:

\begin{theorem} [{\cite[Theorem~6.6]{OTYM1}}] \label{t:ym-heat}
The Yang--Mills heat flow is globally well-posed in $\dot H^1$ for all subthreshold
initial data $a \in \dot H^1$. Precisely, there exists a nondecreasing function 
\[
\bhM: [0,2\Egs) \to \R^+
\]
so that for all subthreshold data $a$ with energy $\nE$ we have
\begin{equation}\label{bold-Q}
\hM(a) \leq \bhM(\nE) .
\end{equation}
\end{theorem}

For connections with finite caloric size, we define the \emph{(global) caloric gauge} as follows:

\begin{definition} \label{def:caloric}
  A connection $a \in \dot H^1$ is caloric if $\hM(a) < \infty$ and
  the global solution to its associated Yang--Mills heat flow has the
  property $a_{\infty}= 0$.
\end{definition}

It is easy to see that for all connections $a$ for which the conclusion of Theorem~\ref{t:ym-heat}
holds there is a unique\footnote{Up to constant conjugations} equivalent caloric gauge. 
This is because $a_{\infty}$ is flat and thus can be represented as 
\[
a_{\infty} =O^{-1}  \rd_{x} O
\]
for a suitable gauge transformation $O$. Then by gauge invariance the equivalent connection
\[
\tilde a = \calG(O) a = OaO^{-1} - \rd_{x} O O^{-1}
\]
is caloric. More precisely, we have

\begin{proposition}[{\cite[Proposition~7.2]{OTYM1}}] \label{p:cal-a} 
 For each $\dot H^1$ connection $\ta$ in $\R^4$ with $\hM(a) < \infty$, there
  exists an unique (up to constant gauge transformations) gauge-equivalent connection
  $a$, which is a caloric gauge connection.  Further, the map
$\ta \to a$ is continuous in the quotient topology defined by the distance
\[
d(a_1,a_2) = \inf_{O \in \G} \| Oa_1O^{-1} - a_2\|_{\dot H^1}.
\]
\end{proposition}

A key result in \cite{OTYM1} asserts that:

\begin{theorem}[{\cite[Proposition~7.7 and Theorem~7.8]{OTYM1}}]\label{t:calC}
The space $\calC$ of all caloric connections is a $C^1$ submanifold
of the space
\[
\bfH = \{ a \in \dot H^1: \partial^j a_j \in \ell^1 L^2\}.
\]
In addition, for all subthreshold caloric connections with energy $\nE$  and caloric size $\hM$ we have the bound
\[
\| a\|_{\bfH}^2 := \nrm{a}_{\dot{H}^{1}}^{2} + \nrm{\rd^{j} a_{j}}_{\ell^{1} L^{2}}^{2} \lesssim_{\nE,\hM} 1 .
\]
\end{theorem}
The second part of the state space norm $\bfH$ reflects the fact that caloric connections satisfy a nonlinear
form of the Coulomb gauge condition.

The second part of \cite{OTYM1} is devoted to modeling the hyperbolic Yang--Mills equation as an evolution 
on the caloric manifold. Precisely, the state space for this evolution is $T^{L^2} \calC$, which is the $L^2$
completion of the tangent space $T\calC$.  We will view the spatial components $(A_{x}, \rd_{t} A_{x}) \in T^{L^{2}} \calC$ as the dynamic variables, and the temporal parts $A_{0}, \rd_{0} A_{0}$ as auxiliary. Correspondingly, we call a pair $(a, b) \in T^{L^{2}} \calC$ (i.e., $a \in \calC$, $b \in T_{a}^{L^{2}} \calC$) a \emph{caloric gauge initial data set} for \eqref{ym}.
This is related to the gauge-covariant notion of initial data sets as follows:
\begin{theorem}[{\cite[Theorem~8.1]{OTYM1}}]\label{t:data}
a) Given any Yang--Mills initial data pair $(a, e) \in \dot H^1
    \times L^2$ with finite caloric size, there exists a caloric gauge
    data set $(\tilde a, b) \in T^{L^{2}} \calC$ and $a_0 \in
    \dot H^1$, so that $(\tilde a, \tilde e)$ is
    gauge equivalent to $(a,e)$, where
    \[
    \tilde e_k = b_k - \covD^{(\tilde{a})}_k a_0 .
    \]
b) Given any caloric gauge initial data set $(\tilde a, b) \in T^{L^{2}}
    \calC$, there exists a unique $a_0 \in \dot H^1$, with Lipschitz
    dependence on $(a, b) \in \dot H^1 \times L^2$, so that
    \[
    e_k = b_k - \covD^{(a)}_k a_0
    \]
    satisfies the constraint equation
    \eqref{eq:YMconstraint}. 
\end{theorem}

By this result, we may indeed fully describe Yang--Mills connections in the caloric gauge as continuous functions
\begin{equation*}
	I \ni t \to (A_{x}, \rd_{t} A_{x})(t) \in T^{L^{2}} \calC.
\end{equation*}
The equations for the dynamical variables $(A_{x}, \rd_{t} A_{x})$ are proved to have the form
\begin{equation}\label{main-wave}
\Box_A A_k = \bfP [A_j,\partial_k A_j] + 2\Delta^{-1} \partial_k \bfQ(\partial^\alpha A_j,\partial_\alpha A_j) + R(A,\partial_t A)
\end{equation}
together with a compatibility condition
\begin{equation}\label{main-compat}
\partial^k A_k = \DA(A) :=  \bfQ(A, A) + \DA^{3}(A),
\end{equation}
where  the temporal component $A_{0}$ and its time derivative
  $\rd_{t} A_{0}$ are uniquely determined by $(A_x,\partial_t A_x)$ and admit the expressions
  \begin{align}
    A_{0} =& \bfA_{0}(A) := \lap^{-1}[A, \rd_{t} A] + 2 \lap^{-1} \bfQ(A, \rd_{t} A) + \bfA_{0}^{3}(A), \label{eq:main-A0} \\
    \rd_{t} A_{0} =& \DA_{0}(A) := - 2 \lap^{-1} \bfQ (\rd_{t} A,
    \rd_{t} A) + \DA_{0}^{3}(A). \label{eq:main-DA0}
  \end{align}
Here $\bfP$ is the Leray projector,  $\bfQ$ is a symmetric bilinear form with symbol\footnote{Although the symbol looks anti-symmetric, it is compensated by the Lie bracket in the definition of bilinear multipliers for $\g$-valued functions; see \cite[Definition~3.1]{OTYM1}.}
\begin{equation}
\bfQ(\xi,\eta) = \frac{\xi^2 - \eta^2}{2(\xi^2+\eta^2)}.
\end{equation}
The cubic error terms are $R$, $\DA^{(3)}$, $\bfA_{0}^{3}$ and $\DA_{0}^{3}$ are ``better behaved'' in the following sense. First, we recall the following definition from \cite{OTYM1}:
\begin{definition} [Envelope preserving map; {\cite[Definition~9.1]{OTYM1}}] \label{def:env-pres} Let $X, Y$ be dyadic norms. A map $\bfF : X \to Y$ is said to be \emph{envelope-preserving
      of order $\geq n$} ($n \in \bbN$ with $n \geq 2$) if for any admissible frequency envelope $c$ for $a$ in $X$, we have
    \begin{equation*}
      \nrm{P_{k} \bfF(a)}_{Y} \aleq_{\nrm{a}_{X}} c_{k}^{n}.
    \end{equation*}
\end{definition}
The cubic error terms $R$,  $\DA^{(3)}$,  $\bfA_{0}^{3}$ and  $\DA_{0}^{3}$
 are envelope preserving maps between the following spaces:
\begin{equation}
R: \Str^1 \to  L^1 L^2 \cap L^2 \dot H^{-\frac12} ,
\end{equation}
\begin{equation}
\DA^{3}: \Str^1 \to  L^1 \dot H^1 \cap L^2 \dot{H}^{\frac{1}{2}},
\end{equation}
\begin{equation}
\bfA_0^3: \Str^1 \to  L^1 \dot H^2 \cap L^{2} \dot H^{\frac{3}{2}},
\end{equation}
\begin{equation}
\DA_0^{3}: \Str^1 \to  L^1 \dot H^1 \cap L^2 \dot{H}^{\frac{1}{2}}.
\end{equation}

Here $\Str^1$ collects several standard non-endpoint Strichartz norms
with the appropriate scaling. Since we work with solutions with $\ell^{2}$ dyadic summability, by the envelope preserving property, for each of the above bounds
we freely gain $\ell^1$ dyadic summability of the above norms. 
One should think of all these cubic  nonlinear expressions above as playing 
perturbative roles in the analysis. We remark that these expressions obey nice difference bounds as well; for details, see \cite[{Definition~9.1 and Theorem~9.2}]{OTYM1}.

\subsection{Local well-posedness in the caloric gauge and energy dispersed solutions} \label{OTYM2}

Our second paper \cite{OTYM2} aims to establish both a local well-posedness result
and a more refined continuation and scattering criteria for
subthreshold solutions to the hyperbolic Yang--Mills equation in the caloric gauge.

In what follows, we will call hyperbolic Yang--Mills connections in the caloric gauge simply \emph{caloric Yang--Mills waves}.

We begin with the  local well-posedness result. We define the $\eps$-\emph{energy
  concentration scale} $r_{c}^{\eps}$ of a finite energy Yang--Mills initial data set $(a,
e)$ to be
\begin{equation*}
	r^\eps_{c} = r_{c}(E)[a, e] = \sup \set{r > 0 : \nE_{B_{r}}(x)[a, e] \leq \eps \ \forall x \in \bbR^{4}}.
\end{equation*}
Then we have:

\begin{theorem}[Local well-posedness in caloric gauge, {\cite[Theorem~1.12]{OTYM2}}] \label{t:local}
  There exists a non-increasing function $\eps_\ast = \eps_{\ast}(\nE,\hM) > 0$ and
  a non-decreasing function $M_{\ast}(\nE,\hM)$ such that, the Yang--Mills equation in caloric gauge is
  locally well-posed on the time interval $I = [-r_{c}^{\eps_\ast}, r_{c}^{\eps_\ast}]$ for
  initial data with energy $\nE$ and initial caloric size $\hM$. More precisely, the following
  statements hold.

a) (Regular data) Let $(a, b) \in T^{L^2} \calC $ be a smooth initial data set with
  energy $\nE$ and initial caloric size $\hM$. Then there exists a
  unique smooth solution $A_{t,x}$ to the Yang--Mills equation in
  caloric gauge on $I$.

b) (Rough data) The data-to-solution map admits a continuous extension
\begin{equation*}
	T^{L^2} \calC  \ni (a, b) \mapsto (A_{x}, \rd_{t} A_{x}) \in C(I, T^{L^{2}} \calC)
\end{equation*}
within the class of initial data with energy concentration scale $\geq r_{c}$.

c) (A-priori bound) The solution defined as above obeys the a-priori bound
\begin{equation*}
	\nrm{A_{x}}_{S^{1}[I]} \leq M_{\ast}(\nE,\hM).
\end{equation*}

d) (Weak Lipschitz dependence) Let $(a', b') \in \calC \times
  L^{2}$ be another initial data set with similar bounds and energy concentration scale
  $\geq r_{c}$. Then for  $\sgm < 1$ close to $1$ we have the 
  Lipschitz bound
\begin{equation*}
	\nrm{A_{x} - A_{x}'}_{S^{\sgm}[I]} \aleq_{M_{\ast}(\nE), \hM, \sgm} \nrm{(a, f) - (a', f')}_{\dot{H}^{\sgm} \times \dot{H}^{\sgm-1}}.
\end{equation*}
\end{theorem}
We remark that bounds for the auxiliary variables $A_{0}, \rd_{0} A_{0}$ follow a-posteriori from the $S^{1}$ bound for $A_{x}$; see  \cite[Theorem~5.1]{OTYM2} for such bounds.

In particular, if the energy of the initial data set is smaller than
$\eps_{\ast} := \min \set{\eps_{\ast}(1), 1}$, then the
corresponding solution $A_{t,x}$ in caloric gauge exists globally and
obeys the bound
\[
\nrm{A_{x}}_{S^{1}[(-\infty, \infty)]} \leq M_{\ast}(\nE).
\]
Thus in particular this result also provides a caloric gauge version of the Coulomb
gauge small data result in \cite{KT}. 

One downside of using either 
the Coulomb or caloric gauge is that causality is lost. To remedy this, in \cite{OTYM2}
we prove that the well-posedness result can also be transferred to the temporal gauge
$A_0 = 0$:

\begin{theorem}[{\cite[Theorem~1.17]{OTYM2}}] \label{t:small-temp}
The hyperbolic Yang--Mills equation in $\bbR^{4+1}$   is globally well-posed in the temporal gauge for 
all initial data with small energy.   
\end{theorem}
This result includes existence, uniqueness, continuous dependence on
the initial data and propagation of higher regularity.  In particular
the finite energy solutions are identified as the unique limits of
regular solutions.  A downside of this theorem is that it does not
provide the $S^1$ regularity of solutions, or any other dispersive bounds.

The second main  result in \cite{OTYM2}  is the following theorem, which uses the energy dispersed norm 
$\ED$ defined on a time interval $I$ by 
\[
\| F\|_{\ED[I]} = \sup_{k \in \bbZ} 2^{-2k} \|P_k F\|_{L^\infty[I]}.
\]
The result asserts that caloric solutions to Yang--Mills with sufficiently small energy dispersion
are extendable and satisfy uniform bounds:

\begin{theorem} [Regularity of energy dispersed
  solutions {\cite[Theorem~1.15]{OTYM2}}] \label{t:ED} There exists a positive non-increa\-sing function
  $\eps(\nE)$ and a non-decreasing function
  $M(\nE)$ such that if $A_{t,x}$ is a solution (in the sense of
  Theorem~\ref{t:local}) to the Yang--Mills equation in caloric gauge
  on $I$ with energy $\nE$ and  that obeys
\begin{equation*}
	\nrm{F}_{\ED[I]} \leq \eps(\nE), \qquad \hM(A(0)) \leq 1, 
\end{equation*}
then it satisfies the a-priori bound
\begin{equation*}
	\nrm{A_x}_{S^{1}[I]} \leq M(\nE),
\end{equation*}
as well as
\begin{equation*}
\sup_{t \in I} \hM(A(t)) \ll 1.
\end{equation*}
Moreover, $A$ can be continued as a solution to the Yang--Mills
equation in caloric gauge past finite endpoints of $I$.
\end{theorem}

\subsection{Topological classes and large data solutions} \label{subsec:OTYM2.5}

Unlike the first two papers, the third one \cite{OTYM2.5} is concerned
with large data solutions which are not necessarily topologically
trivial, and thus cannot be directly studied using the global caloric
gauge. The goal of \cite{OTYM2.5} is two-fold:
\begin{itemize}
\item To describe finite energy initial data sets topologically and analytically.

\item To provide a good local theory for finite energy solutions.
\end{itemize}

Here we work in two settings:

\begin{enumerate}[label=\alph*)]
\item  For initial data in $\bbR^4$ and solutions in $\R^{1+4}$, or time sections thereof.

\item For initial data in a ball $B_R$ and solutions in the corresponding domain of dependence
$\calD(B_R) = \{ |x|+|t| < R\}$
or time sections thereof.
\end{enumerate}
In terms of the initial data, in addition to the energy, a key role is played by 
the $\eps$-energy concentration scale localized to the ball $B_{R}$
\[
r_c^{\eps} = \sup  \{ r > 0 : \nE_{B_{r}(x) \cap B_R}[(a,e)] \leq \eps \ \forall x \in B_R \},
\]
as well as the outer concentration radius
\[
R_c^{\eps} =  \inf  \{ r > 0 : \nE_{B_{r}(x)}[(a,e)] \leq \eps  \text{ for some } x \in \bbR^4\  \}.
\]

\subsubsection{Finite energy data sets}

We begin with an excision result, which provides small energy extensions outside an annulus:
 
\begin{proposition}[{\cite[Theorem~1.16]{OTYM2.5}}]\label{t:chop-small}
Let $(a,e)$ be a small energy data set in $B_4\setminus B_1$. Then we can find a small energy 
exterior data set $(\ta,\te)$ in $\R^4 \setminus B_1$ which agrees with $(a,e)$ in $B_2\setminus B_1$.
Furthermore, if $(a,e)$ is smooth then $(\ta,\te)$ can also be chosen to be smooth. 
\end{proposition}

For initial data sets in a ball, it is useful to work with a good gauge:

\begin{proposition} [{\cite[Theorem~1.4]{OTYM2.5}}]
 Given an initial data $(a,e)$ in $B_R$ with finite energy and $\eps$-energy concentration scale $r_C$, there exists 
a gauge-equivalent initial data $(\ta,\te)$ in $B_R$ which satisfies the bound
\begin{equation}
\| \ta\|_{\dot H^1 \cap L^4} \lesssim_{\eps, \frac{r_C}{R}} 1.
\end{equation}
\end{proposition}

Consider now finite energy initial data $(a,e)$ in $\R^4$. Here we need to distinguish 
between $a$ in different topological classes. We begin with the topologically 
trivial maps:

\begin{theorem}[\cite{OTYM2.5}] \label{zero-class}
A finite energy connection $a$ is topologically trivial if and only if it admits a representation $a \in \dot H^1$
in a suitable gauge.
\end{theorem}

Finally, for topologically nontrivial initial data in $\R^n$ we also can find a good global gauge:

\begin{theorem} [Good global gauge theorem {\cite[Theorem~1.5]{OTYM2.5}}] \label{thm:goodrep}
Let $a \in H^1_{loc}$ be a finite energy connection. 
Then there exists a global representative $a$ such that 
\begin{equation*}
  a = - \chi O_{(\infty); x} + b
\end{equation*}
where $1-\chi$ is smooth and compactly supported, $O_{(\infty)} (x)$ is a smooth $0$-homogeneous map taking values in $\G$ and $b \in \dot{H}^{1}$. 
\end{theorem}

Here one can identify the topological class of $a$ with the homotopy class $[O]$ of $O:\mathbb S^3 \to \G$.
in particular $O$ in the last theorem can be chosen arbitrarily within this homotopy class.

\subsubsection{Finite energy solutions}

A consequence of \cite{KT} and of  the first two papers in the series \cite{OTYM1, OTYM2} 
is that the small data problem for the $4+1$ dimensional 
hyperbolic Yang--Mills equation is well-posed in several gauges: Coulomb, caloric, and temporal.  
 In \cite{OTYM2.5}  we exploit  the temporal gauge small data result, combined with causality,
to obtain results for the large data problem.   

For the subsequent results, set  $r_c = r_c^{\eps_0}$ where  $\eps_0 \ll 1$ is  the energy bound for the small data result.
The main local-in-time result is as follows:
\begin{theorem}[{\cite[Theorem~1.22]{OTYM2.5}}]\label{t:local-temp}

a) For each finite energy data set $(a,e)$ in $\bbR^4$  with concentration scale $r_c$ 
there exists a unique finite energy  solution $A$ to \eqref{ym} in the time interval $[-r_c,r_c]$ in the 
temporal gauge $A_0 = 0$, depending continuously on the initial data. Furthermore, 
any other finite energy solution with the same data must be gauge equivalent to $A$.

b) The same result holds for data in a ball $B_R$ and  the solution in
the corresponding domain of uniqueness $\calD(B_{R}) \cap (I \times \bbR^{4})$.
\end{theorem}

Now we consider the continuation question. The next result asserts that temporal solutions can be continued
until energy concentration (i.e., a blow-up) occurs. Thus, temporal solutions are also maximal solutions
for the Yang--Mills equation.

\begin{theorem}[{\cite[Theorems~1.22 and 1.23]{OTYM2.5}}]\label{t:global-temp}
a) For each finite energy data set $(a,e)$ in $\bbR^4$, let $(T_{min},T_{max})$ be the maximal time interval on which the 
temporal gauge solution $A$ exists. If $T_{max}$ is finite then we have
\[
\lim_{t \to T_{max}} r_c(t) = 0.
\]
Further, there exists some $X \in \R^4$ so that the energy concentration occurs in the backward light cone 
$C = \{ |x - X| \leq T_{max} - t\}$    centered at $(T_{max},X)$, in the sense that
\begin{equation}
\lim_{t \nearrow T} \nE_{C \cap S_{t}}(A) > \eps_0.
\end{equation}
The similar result holds for  for $T_{min}$. 

b) The same result holds for data in a ball $B_R$ and  the solution in
the corresponding domain of uniqueness $\calD(B_R)$.
\end{theorem}
We remark that vanishing of $r_{c}$ follows from \cite[Theorem~1.22]{OTYM2.5}, and existence of an energy concentration point follows by a standard argument; see, e.g., \cite[Lemma~8.1]{OT3}.

The temporal gauge is convenient in order to deal with causality, but not so much in terms of regularity, as it
lacks good $S$ bounds. For this reason it is convenient to borrow the caloric gauge regularity:

\begin{theorem}[{\cite[Theorem~1.25]{OTYM2.5}}] \label{t:good-gauge}
  Let $A$ be a finite energy Yang--Mills solution in a cone section
  $C_{[t_1,t_2]}$ with energy concentration scale $r_c$.  Then in
  a suitable gauge $A$ satisfies the bound
\begin{equation}\label{good-gauge}
\| A\|_{L^\infty (\dot H^1 \cap L^4)}
 + \| \partial_t A\|_{L^\infty L^2}
 + \| \partial^j A_j \|_{\ell^1 L^{2} \dot H^\frac12} 
 + \| \nb A_0\|_{\ell^1 L^{2} \dot{H}^{\frac{1}{2}}}
+ \| \Box A_x\|_{\ell^{1} L^2 \dot H^{-\frac12}} \lesssim_{E,\frac{r_c}{t_2}} 1
\end{equation}
in the smaller cone $C_{[t_1,t_2]}^{4r_c}$ where the radius has been decreased by $4 r_c$.
\end{theorem}

For the notation $C_{[t_{1}, t_{2}]}$ and $C_{[t_{1}, t_{2}]}^{4 r_{c}}$, we refer to Section~\ref{subsec:notation}. The proof of this theorem requires a good gluing technique for local connections
with suitable regularity; see \cite{OTYM2.5} for details. 
We note that the term $\nrm{\rd^{j} A_{j}}_{\ell^{1} L^{2} \dot{H}^{\frac{1}{2}}}$ is missing in the statement of \cite[Theorem~1.25]{OTYM2.5}, but is evident from the proof (see, in particular, \cite[Eqs.~(5.6), (5.8)]{OTYM2.5}).

\section{Monotonicity formulas}\label{sec:energy}

 \subsection{The energy-momentum tensor and conservation laws}
We start by introducing the notion of a null frame and the associated null decomposition of a curvature $2$-form, which provides a very useful decomposition of the energy momentum tensor.
At each point $p= (t_{0}, x_{0}) \in \bbR^{1+4}$, we introduce the null pair
\begin{equation*}
	L = \rd_{t} + \frac{x}{\abs{x}} \cdot \rd_{x}, \qquad
	\uL= \rd_{t} - \frac{x}{\abs{x}} \cdot \rd_{x},
\end{equation*}
and also orthonormal vectors
$\set{e_{\mfa}}_{\mfa=1,\ldots, 3}$ which are orthogonal to $L$ and
$\uL$. Observe that each $e_{\mfa}$ is tangent to the sphere $\rd
B_{t_{0}, r_{0}} := \set{t_{0}} \times \rd B_{r_{0}}(0)$ where $r_{0}
= \abs{x_{0}}$. The set of vectors $\set{L, \uL, e_{1}, e_{2}, e_{3}}$
at $p$ is called a \emph{null frame at $p$ associated to $L, \uL$}.

We define the \emph{null decomposition} of the 2-form $F$ with
respect to $\set{L, \uL, e_{\mfa}}$ as
\begin{equation*}
  \alp_{\mfa} := F(L, e_{\mfa}), \quad
  \ualp_{\mfa} := F(\uL, e_{\mfa}), \quad
  \varrho := \frac{1}{2} F(L, \uL), \quad
  \sgm_{\mfa\mfb} := F(e_{\mfa}, e_{\mfb}).
\end{equation*}
Note that $\varrho$ is a $\g$-valued function, $\alp_{\mfa}, \ualp_{\mfb}$ are $\g$-valued 1-forms
on $\rd B_{t_{0}, r_{0}}$ and $\sgm_{\mfa\mfb}$ is a $\g$-valued 2-form on $\rd
B_{t_{0}, r_{0}}$. We define their pointwise absolute values as
\begin{equation*}
  \abs{\alp}^{2} := \sum_{\mfa=1, \ldots, 3} \alp_{\mfa}^{2}, \quad
  \abs{\ualp}^{2} := \sum_{\mfa=1, \ldots, 3} \ualp_{\mfa}^{2}, \quad
  \abs{\sgm}^{2} := \sum_{1 \leq \mfa < \mfb \leq 3} \sgm_{\mfa\mfb}^{2}.
\end{equation*}

Recall from Section~\ref{sec:intro} that the \emph{energy-momentum tensor} associated to a connection $A$ is
\begin{equation} 
\EM_{\aa \bb}(A) = 2 \la \tensor{F}{_{\aa}^{\gamma}}, F_{\bb \gamma}\ra - \frac{1}{2} \met_{\aa \bb }\la F_{\gamma \delta}, F^{\gamma \delta}\ra.
\end{equation}
 We observe that $\EM$ is a symmetric 2-tensor, which is gauge invariant at
each point. Moreover for each finite energy solution solution to
\eqref{ym}, the energy-momentum tensor satisfies
\begin{equation} \label{eq:div4EMT} 
\rd^{\aa} \EM_{\aa \bb}(A) = 0.
\end{equation}
This is verified directly for smooth connections, and it then transfers to finite energy Yang--Mills connections
by approximation with smooth connections.

A simple way of obtaining energy identities for Yang--Mills equation is to contract
the energy-momentum tensor with a well-chosen vector field, an then integrate 
over a suitable domain.

Given a vector field $X$ on $\calO$, we define its \emph{deformation
  tensor} to be the Lie derivative of the metric with respect to $X$,
i.e., $\defT{X} := \calL_{X} \bfm$. Using covariant derivatives,
$\defT{X}$ also takes the form
\begin{equation*}
  \defT{X}_{\aa \bb} = \covnb_{\aa} X_{\bb} + \covnb_{\bb} X_{\aa}
\end{equation*}
or expressed in coordinates
\begin{equation} \label{eq:defTcoord} \defT{X}_{\mu \nu} = X(\bfm_{\mu
    \nu}) + \rd_{\mu} (X^{\alp}) \bfm_{\alp \nu} + \rd_{\nu}
  (X^{\alp}) \bfm_{\alp \mu}
\end{equation}

Using the deformation tensor, we define the associated $1$- and
$0$-currents of $A$ as
\begin{equation} \label{eq:vsC-X}
  \begin{aligned}
    \vC{X}_{\aa} (A) :=& \EM_{\aa \bb} (A) X^{\bb}, \\
    \sC{X} (A) :=& \frac12 \EM_{\aa \bb} (A) 
    \defT{X}^{\aa \bb}.
  \end{aligned}
\end{equation}
Then by \eqref{eq:div4EMT}  we obtain
\begin{equation} \label{eq:vsC-X:div} 
\covnb^{\aa}  (\vC{X}_{\aa}(A)) = \sC{X}(A).
\end{equation}
Now energy identities for the Yang--Mills flow are obtained by integrating 
this identity over spacetime regions. Of course, this is most useful when
$ \defT{X}$ either vanishes (i.e. $X$ is Killing) or when it has a sign.

The simplest choice for $X$ is $X=T = \rd_{0}$, the unit vector in the time direction.
Then $ \defT{T} = 0$, so \eqref{eq:vsC-X:div} becomes
\begin{equation} \label{eq:vsC-T:div} 
\covnb^{\aa}  (\vC{T}_{\aa}(A)) = 0.
\end{equation}
In particular we have 
\[
\vC{T}_{0} = \frac12 \brk{F_{jk}, F^{jk}} + \brk{F_{0j}, \tensor{F}{_{0}^{j}}} = \sum_{\alp < \bt} \abs{F_{\alp \bt}}^{2},
\]
therefore integrating \eqref{eq:vsC-T:div} between time slices yields 
the well-known conservation of energy
\[
\nE_{\set{t} \times \bbR^{4}}(A) = \int \vC{T}_{0}(A) \, \ud x = \int \sum_{\alp < \bt} \abs{F_{\alp \bt}}^{2} \quad \hbox{ is constant in $t$}.
\]
In general, for $U \subset \bbR^{4}$ we introduce the notation
\begin{equation*}
	\nE_{\set{t} \times U}(A) = \int_{\set{t} \times U} \vC{T}_{0} (A)\, \ud x .
\end{equation*}

We also need to use energy estimates in sections $C_{[t_1,t_2]}$ of the cone $C$.
For this we define the  \emph{energy flux} on the lateral surface of the cone section by
\begin{equation*}
	\EFlux_{\rd C_{[t_{0}, t_{1}]}} (A)= 
\frac{1}{2} \int_{\rd C_{[t_{0}, t_{1}]}} \vC{T}_{L}(A) r^{3} \, \ud v \ud \sgm_{\bbS^{3}}
\end{equation*}
Then we have

\begin{lemma} \label{l:flux-decay}
Let $A$ be a finite energy Yang--Mills connection on $I \times \bbR^{4}$ where $I \subset \bbR^+$ is an open interval. Then for every $t_{0}, t_{1} \in I$ with $t_{0} \leq t_{1}$, the following statements hold:

a) The energy flux on $\EFlux_{\rd C_{[t_{0}, t_{1}]}}(A)$ is non-negative and additive, i.e.,
\begin{equation} \label{eq:flux-additive}
\EFlux_{\rd C_{[t_{0}, t_{1}]}}(A) = \EFlux_{\rd C_{[t_{0}, t']}}(A) + \EFlux_{\rd C_{[t', t_{1}]}}(A) \quad \hbox{ for } t' \in [t_{0}, t_{1}].
\end{equation}

b) The energy-flux relation holds:
\begin{equation} \label{energy-flux}
\nE_{S_{t_{1}}}(A) - \nE_{S_{t_{0}}}(A)= \EFlux_{\rd C_{[t_{0}, t_{1}]}}(A) .
\end{equation}
\end{lemma}

The nonnegativity is straightforward since the flux density is
expressed in terms of the curvature components in the null frame as
\[
 \vC{T}_{L}(A) = |\varrho|^2 + \abs{\sigma}^2 + \abs{\alp}^{2}.
\]
The identities are again straightforward for smooth solutions, and obtained by approximation
with smooth solutions otherwise.

A consequence of Lemma~\ref{l:flux-decay} is a simple but crucial decay result for the flux:

\begin{corollary} \label{cor:flux-decay}
Let $A$ be a finite energy Yang--Mills connection on $I \times \bbR^{4}$ where $I \subset \bbR^{4}$ is an open interval. Then the following statements hold.

a) If $(0, \dlt] \subseteq I$ for some $\dlt > 0$, then we have
\begin{equation} \label{eq:flux-decay:0}
	\lim_{t_{1} \to 0} \EFlux_{\rd C_{(0, t_{1}]}}(A) = 0,
\end{equation}

b) If $[\dlt, \infty) \subseteq I$ for some $\dlt > 0$, then we have
\begin{equation} \label{eq:flux-decay:infty}
	\lim_{t_{0}, t_{1} \to \infty} \EFlux_{\rd C_{[t_{0}, t_{1}]}}(A) = 0.
\end{equation}
\end{corollary}

\subsection{Monotonicity formulas}

Here we are interested in the case when the expression $  \sC{X}(A)$ is nonnegative definite.
Our primary vector field here will be 
\[
X_0 = \frac{1}{\rho_{0}} \left( t \rd_{t} + x \cdot \rd_{x} \right), \qquad \rho_{0} = \sqrt{t^{2} - \abs{x}^{2}}
\]
for $(t,x) \in C$. We also introduce the null coordinates
\begin{equation*}
	u_{0} = t - \abs{x}, \qquad v_{0} = t + \abs{x}.
\end{equation*}

Straightforward computations (see \cite{OT3}) lead to the relation
\begin{equation}
\sC{X}(A) = \frac{2}{\rho_{0}} | \iota_{X_0} F|^2,
\end{equation}
where $ | \iota_{X_0} F|^2 =  \met^{\alp \bt} \iota_{X_{0}} F_{\alp} \iota_{X_{0}} F_{\bt}$. Of course $\met$ is indefinite, but $ | \iota_{X_0} F|^2$ is nonnegative due to the fact that $X_0$ is time-like inside the cone $C$.
Thus the relation \eqref{eq:vsC-X:div} becomes
\begin{equation}\label{eq:vsC-X0:div} 
	\covnb^{\aa} ({}^{(X_{0})} P_{\aa}) = \frac{2}{\rho_{0}} \abs{\iota_{X_{0}} F}^{2} \geq 0
\end{equation}
where the interesting components of ${}^{(X_{0})} P$ are
\begin{align} 
	{}^{(X_{0})} P_{L}
=&	\Big(\frac{v_{0}}{u_{0}}\Big)^{\frac{1}{2}} \abs{\alp}^{2}
	+ \Big(\frac{u_{0}}{v_{0}}\Big)^{\frac{1}{2}} \Big( \abs{\varrho}^{2} + \abs{\sgm}^{2} \Big), \\
	{}^{(X_{0})} P_{\uL}
=&	\Big(\frac{u_{0}}{v_{0}}\Big)^{\frac{1}{2}} \abs{\ualp}^{2}
	+ \Big(\frac{v_{0}}{u_{0}}\Big)^{\frac{1}{2}} \Big( \abs{\varrho}^{2} + \abs{\sgm}^{2} \Big). 
\end{align}

All expressions above are singular on the cone, so we cannot\footnote{Unless the flux is zero. This is in general not the case, instead we will work in settings where the flux is merely small.} directly integrate the relation
\eqref{eq:vsC-X0:div} on sections of the cone $C$.
 To remedy this, we will translate the field $X_0$ 
downward by $\veps$. Defining 
\begin{gather*}
	\rho_{\veps} = \sqrt{(t+\veps)^{2} + \abs{x}^{2}}, \quad
	X_{\veps} = \frac{1}{\rho_{\veps}} \left( (t+\veps) \rd_{t} + x \cdot \rd_{x} \right) \\
	u_{\veps} = t + \veps - \abs{x}, \quad
	v_{\veps} = t + \veps + \abs{x}, \quad
\end{gather*}
we now have the shifted relations
\begin{equation}
	\covnb^{\aa} ({}^{(X_{\veps})} P_{\aa}) = \frac{2}{\rho_{\veps}} \abs{\iota_{X_{\veps}} F}^{2}
\end{equation}
where
\begin{align} 
	{}^{(X_{\veps})} P_{L}
=&	 \Big(\frac{v_{\veps}}{u_{\veps}}\Big)^{\frac{1}{2}} \abs{\alp}^{2}
	+  \Big(\frac{u_{\veps}}{v_{\veps}}\Big)^{\frac{1}{2}} \Big( \abs{\varrho}^{2} + \abs{\sgm}^{2} \Big),  	\label{eq:mono-L} \\
	{}^{(X_{\veps})} P_{\uL}
=&	 \Big(\frac{u_{\veps}}{v_{\veps}}\Big)^{\frac{1}{2}} \abs{\ualp}^{2}
	+  \Big(\frac{v_{\veps}}{u_{\veps}}\Big)^{\frac{1}{2}} \Big( \abs{\varrho}^{2} + \abs{\sgm}^{2} \Big). \label{eq:mono-uL} 
\end{align}
We also remind the reader that since $\rd_{t} = \frac{1}{2} (L + \uL)$, we have
\begin{equation*}
\mvC{X_{\veps}}_{0} = \frac{1}{2} (\mvC{X_{\veps}}_{L} + \mvC{X_{\veps}}_{\uL}).
\end{equation*}
Integrating now the relation \eqref{eq:vsC-X0:div} over an appropriate section of the cone $C$ yields the following:

\begin{proposition} \label{prop:monotonicity}
Let $A$ be a finite energy Yang--Mills connection on $[\veps, 1] \times \bbR^{4}$, 
where $\veps \in (0, 1)$. Suppose furthermore that $A$ satisfies
\begin{equation} \label{eq:monotonicity:hyp}
	\nE_{S_{1}}(A) \leq E, \quad \EFlux_{\rd C_{[\veps, 1]}} (A) \leq \veps^{\frac{1}{2}} E.
\end{equation}
Then
\begin{equation} \label{eq:monotonicity}
\int_{S_{1}} \mvC{X_{\veps}}_{0}(A) \, \ud x 
+ \iint_{C_{[\veps, 1]}} \frac{2}{\rho_{\veps}} \abs{\iota_{X_{\veps}} F}^{2} \ud t \ud x \aleq E,
\end{equation}
where the implicit constant is independent of $\veps, E$. 
\end{proposition}

Using Proposition~\ref{prop:monotonicity}, we can also establish a
version of \eqref{eq:monotonicity} that is localized away from
the boundary of the cone. This statement will be useful for
propagating lower bounds in a time-like region towards $(0, 0)$.

\begin{proposition} \label{prop:monotonicity:t-like}
Let $A$ a finite Yang--Mills connection in $[\veps, 1] \times \bbR^{4}$, where $\veps \in (0, 1)$. Suppose furthermore that $A$ satisfies \eqref{eq:monotonicity:hyp}. Then for $2 \veps \leq \dlt_{0} < \dlt_{1} \leq t_{0} \leq 1$, we have
\begin{equation} \label{eq:monotonicity:t-like}
	\int_{S_{1}^{\dlt_{1}}} \mvC{X_{0}}_{0}(A) \, \ud x
	\leq \int_{S_{t_{0}}^{\dlt_{0}}} \mvC{X_{0}}_{0}(A) \, \ud x + C\Big( (\dlt_{1} / t_{0})^{\frac{1}{2}}+ \abs{\log (\dlt_{1} / \dlt_{0})}^{-1} \Big) E.
\end{equation}
\end{proposition} 
The proofs of Propositions~\ref{prop:monotonicity} and \ref{prop:monotonicity:t-like} are similar to those of Propositions~5.4 and 5.5 in \cite{OT3}, respectively, and thus are omitted.

\section{A compactness result}\label{sec:cpt}

Here we establish the following compactness result for a locally uniformly bounded sequence of Yang--Mills connections which are asymptotically stationary:

\begin{theorem}\label{t:compact}
Let $A^{(n)}$ be a sequence of finite energy Yang--Mills connections in $[-2, 2] \times B_{2R}$
which is uniformly bounded, in the sense that the norms
\begin{equation} \label{eq:unif-bnd}
\| A^{(n)}\|_{L^\infty (\dot H^1 \cap L^4)}, \, 
\| \partial_t A^{(n)}\|_{L^\infty L^2}, \,
\|\nb A^{(n)}_0\|_{\ell^1 L^{2} \dot{H}^{\frac{1}{2}}}, \, 
\| \Box A^{(n)}_x\|_{\ell^{1} L^2 \dot H^{-\frac12}} 
\end{equation}
on $[-2, 2] \times B_{2R}$ are uniformly bounded in $n$. Assume moreover that 
\[
\lim_{n \to \infty} \| \iota_{V} F^{(n)} \|_{L^2([-2, 2] \times B_{2R})} = 0,
\]
where $V$ is a smooth time-like vector field (i.e., $\met(V, V) < 0$). Then on a subsequence we have 
\[
A^{(n)} \to A \qquad \text{in} \ \  H^1([-1, 1] \times B_{R})
\]
where $A$ solves the hyperbolic Yang--Mills equation (in the sense of distributions), satisfies $\iota_{V} F=0$,  and has regularity
\[
A_x \in \ell^1 H^{\frac32}([-1, 1] \times B_{R}), \qquad \nb A_0 \in \ell^1 L^2 H^\frac12 ([-1, 1] \times B_{R}).
\]
\end{theorem}
Clearly, by scaling, this result is applicable to any spacetime cylinder $(t_{0} - 2T, t_{0} + 2T)  \times B_{2R}(x_{0})$. In the sequel, Theorem~\ref{t:compact} will be applied to a sequence $A^{(n)}$ on nested open sets $\calO^{(n)} (\subset \calO^{(n+1)})$, which satisfies the hypothesis on each spacetime cylinder inside $\calO^{(n)}$ for sufficiently large $n$ after taking a suitable gauge transformation; see Section~\ref{sec:bubble-off} below. Thus, on any open set $\calO' \subset \overline{\calO'} \subset \bigcup_{n} \calO^{(n)}$, we will extract a subsequential limit via a diagonal procedure possessing the local regularity
\begin{equation*}
	A_{x} \in \ell^{1} H^{\frac{3}{2}}_{loc}(\calO'), \quad \nb A_{0} \in \ell^{1} L^{2} H^{\frac{1}{2}}_{loc}(\calO').
\end{equation*}

\begin{proof}
Let $\chi$ be a smooth cutoff supported in $[-2, 2] \times B_{2R}$, which is identically $1$ in $[-1, 1] \times B_{R}$.
Consider the sequence $\set{\chi A^{(n)}}$, which is now globally defined in $\bbR^{1+4}$. Moreover, we claim that $\set{\chi A^{(n)}}$ is uniformly bounded with respect to $n$ in the global-in-spacetime version of the norms listed in \eqref{eq:unif-bnd}; the bound depends on $R$ and the corresponding norms of $A^{(n)}$ on $[-2, 2] \times B_{2R}$. Indeed, it is straightforward to reduce the claim to the following global-in-spacetime commutator bounds:
\begin{align*}
	\nrm{[\nb, \chi] B}_{L^{\infty} L^{2}} & \aleq_{R} \nrm{B}_{L^{\infty} L^{4}}, \\
	\nrm{[\nb, \chi] B}_{\ell^{1} L^{2} \dot{H}^{\frac{1}{2}}} & \aleq_{R} \nrm{\nb B}_{L^{\infty} L^{2}}, \\
	\nrm{[\Box, \chi] B}_{\ell^{1} L^{2} \dot{H}^{-\frac{1}{2}}} & \aleq_{R} \nrm{\nb B}_{L^{\infty} L^{2}},
\end{align*}
These commutator bounds, in turn, follow from the usual Littlewood--Paley trichotomy analysis. 

Next, we microlocally split the connections $\chi A^{(n)}$ into a high modulation part and a low modulation part
\[
\chi(t, x) A^{(n)} =   A^{(n),lo} + A^{(n),hi} := \eta(D_{t,x}) \chi(t, x) A^{(n)} + (1-\eta(D_{t,x})) \chi(t,x)  A^{(n)} 
\]
i.e., with a classical order zero multiplier $\eta$ which is supported in a small neighborhood 
$\set{\abs{\abs{\tau}-\abs{\xi}} < \kpp \abs{\xi}}$ of the null cone
$\{ \tau^2 = \xi^2\}$ and which is identically $1$ in the smaller neighborhood $\set{\abs{\abs{\tau}-\abs{\xi}} < \frac{\kpp}{2} \abs{\xi}}$. We choose $\kpp$ small enough so that $\rd_{0}$ and $V^{\alp} \rd_{\alp}$ are microlocally elliptic (i.e., $\abs{\tau} \ageq \abs{\tau} + \abs{\xi}$ and $\abs{V^{0} \tau + V^{k} \xi_{k}} \ageq \abs{\tau} + \abs{\xi}$) on the support of $\eta(\tau, \xi) \chi(t,x) $, which is possible since $\rd_{0}$ and $V^{\alp} \rd_{\alp}$ are time-like.

Since $\Box$ is microlocally elliptic in the support of $1-\eta(\tau, \xi)$, the uniform bound for $\Box (\chi A^{(n)})$ implies that the high modulation parts
$A_{x}^{(n),hi}$ are uniformly bounded in $\ell^1 H^\frac32$.  The same happens
with the $\nabla (\chi A_0^{(n)})$ component in its entirety in $\ell^1 L^2 H^{\frac12}$. On a
subsequence we get convergence in $H^1([-1, 1] \times B_{R})$ for $A_x^{(n),hi}$ and $A_{0}^{(n)}$ with the limits $A_{x}^{hi} \in \ell^{1} H^{\frac{3}{2}}([-1, 1] \times B_{R})$ and $\nb A_{0} \in \ell^1 L^2
H^\frac12([-1, 1] \times B_{R})$, respectively.

It remains to consider the low modulation part of $A_x^{(n)}$. For this we expand $\iota_{V} F$ as 
\[
(\iota_{V} F)_\beta = V^\alpha F_{\alpha \beta} = V^\alpha (\partial_\alpha A_\beta - \partial_\beta A_\alpha+ [A_\alpha,A_\beta])
\]
Separating the cases $\beta = 0$ and $\beta \neq 0$, we view this as a system for $A_k$ of the form
\begin{align}
V^\alpha \partial_\alpha A_k - \partial_k (V^{j} A_{j}) = & \left( - (\iota_{V} F)_k +   V^\alpha [A_\alpha,A_k] + V^0 \partial_k A_0 \right)
- (\rd_{k} V^{j}) A_{j}, \label{eq:iX-F-k} \\
\partial_0 (V^j  A_j) = & \left( - (\iota_{V} F)_0 +  V^j (\partial_j A_0  + [A_j,A_0]) \right)
+ (\rd_{0} V^{j}) A_{j}. \label{eq:iX-F-0}
\end{align}
Here the LHS can be viewed as a system in $A_k$ which is microlocally elliptic of order $1$ on the support of $\eta(\tau, \xi) \chi (t,x)$. 
To exploit this fact, we apply $\eta(D_{t,x}) \chi(t,x)$ to the both sides and rewrite the above system as
\begin{align*}
	V^{\alp} \rd_{\alp} (\eta(D_{t,x}) \chi A_{k}) - \rd_{k} (\eta(D_{t,x}) \chi V^{j} A_{j}) & = \eta(D_{t,x}) \chi \hbox{(RHS of \eqref{eq:iX-F-k})} \\
	& \phantom{=} + [V^{\alp} \rd_{\alp}, \eta(D_{t,x}) \chi] A_{k} - [\rd_{k}, \eta(D_{t,x}) \chi] V^{j} A_{j} \\
	\rd_{0} (\eta(D_{t,x}) \chi V^{j} A_{j}) & = \eta(D_{t,x}) \chi \hbox{(RHS of \eqref{eq:iX-F-0})} + [\rd_{0}, \eta(D_{t,x}) \chi] V^{j} A_{j}.
\end{align*}
It is straightforward to check that the resulting RHS has size
\[
o_{L^2}(1) + O_{\ell^1 H^\frac12}(1),
\]
where we note that only the terms of the form $(\iota_{V} F)_{\bt}$ contribute $o_{L^{2}}(1)$.
Thus, on a subsequence, we get convergence in $H^1([-1, 1] \times B_{R})$ first for $\eta(D_{t,x}) \chi V^{j} A_{j}^{(n)}$ and then for $A_{x}^{(n),lo} = \eta(D_{t,x}) (\chi A_{x})$, with the limits in $\ell^1 H^\frac32 ([-1, 1] \times B_{R})$. Now the convergence of $F^{(n)}$ in $L^2([-1, 1] \times B_{R})$ is easy to establish. \qedhere

\end{proof}

\section{Regularity of stationary connections}\label{sec:reg}

Here  we consider the solutions provided as limiting connections in Theorem~\ref{t:compact} (see the discussion following the theorem).
These have the local regularity
\begin{equation}\label{A-reg}
A_x \in \ell^1 H^\frac32_{loc}(\calO),  \qquad  \nb A_0 \in \ell^1 L^{2} H^{\frac12}_{loc}(\calO)
\end{equation}
on some open subset $\calO$ of $\bbR^{1+4}$, and satisfy
\begin{equation}
\iota_V F = 0. \label{eq:F-stat}
\end{equation}
We further specialize to the following two cases:
\begin{enumerate}
\item[(i)] $V$ is constant and time-like, or
\item[(ii)] $V = S = x^{\alp} \rd_{\alp}$ is the scaling vector field.
\end{enumerate}
Moreover, in Case~(ii), $A$ is defined in the forward light cone, i.e., $\calO \subset C$. The goal of this section is to establish the following (qualitative) regularity result.
\begin{proposition} \label{p:stat} 
Let $A$ be a hyperbolic Yang--Mills connection on an open set $\calO \subset \bbR^{1+4}$ that satisfies the above properties. Then in any open set $\calO'$ such that $\overline{\calO'} \subset \calO$, the connection $A$ is gauge-equivalent, via continuous local gauge transformations to a smooth connection. More precisely, there exists an open covering $\set{B}$ of $\calO'$ and a continuous gauge transformation $O$ on each $B$ such that $\tA = \calG(O) A$ is smooth.
\end{proposition}

Passing to a smaller set $\calO'$ makes the statement and the proof simple. This point will not be problematic for our application, thanks to the following extension result for a smooth hyperbolic Yang--Mills connection satisfying \eqref{eq:F-stat}:
\begin{proposition} \label{p:stat-ext}
Let $A$ be a hyperbolic Yang--Mills connection on an open set $\calO \subset \bbR^{1+4}$ that satisfies the above properties. Assume furthermore that $A$ is smooth. Then $A$ is gauge-equivalent, via a smooth gauge transformation, to a smooth connection $\tilde{A} = \calG(O) A$ obeying $\iota_{V} \tilde{A} = 0$ and $\calL_{V} \tilde{A} = 0$. Moreover, the extension of $\tilde{A}$ to $\bigcup_{s \in \bbR} {}^{(V)} \Phi_{s} (\calO)$ via $\calL_{V} \tilde{A} = 0$, which we still denote by $\tilde{A}$, remains a smooth hyperbolic Yang--Mills connection satisfying \eqref{eq:F-stat}.
\end{proposition}
Here, ${}^{(V)} \Phi_{s}$ is the one-parameter family of diffeomorphisms generated by the vector field $V$ and $\calL_{V}$ is the Lie derivative with respect to $V$, i.e., $\calL_{V} \tilde{A} = \frac{\ud}{\ud s} {}^{(V)} \Phi_{s}^{\ast} \tilde{A} \vert_{s = 0}$.

In the sequel, we will apply Proposition~\ref{p:stat} so that 
\begin{equation*}
\calO' = 
\begin{cases}
(-\frac{1}{4}, \frac{1}{4}) \times \bbR^{4} & \hbox{in Case~(i)}, \\
C_{[\frac{3}{2}, \infty)}^{\frac{3}{2}} & \hbox{in Case~(ii)}.
\end{cases}
\end{equation*}
In both cases $\calO'$ is contractible so that the gauge-equivalent smooth connection may be represented by a single $\g$-valued 1-form $\tA$ on $\calO'$. By continuity of the local gauge transformations, it follows that the global gauge transformation $O$ from $A$ to $\tA$ is continuous. Moreover, by the formula
\begin{equation*} 
	O_{;\alp} = Ad(O) A - \tA,
\end{equation*}
regularity \eqref{A-reg} and smoothness of $\tA$, $O$ is admissible in the sense of Definition~\ref{def:gt}. Finally, by Proposition~\ref{p:stat-ext} the smooth connection $\tA$ extends to a smooth stationary Yang--Mills connection on $\bbR^{1+4} = \bigcup_{s \in \bbR} {}^{(V)}\Phi_{s}(\calO')$ in Case~(i), and to a smooth self-similar Yang--Mills connection on $C = \bigcup_{s \in \bbR} {}^{(V)}\Phi_{s}(\calO')$ in Case~(ii); see Section~\ref{sec:bubble-off} below.

The remainder of the section is devoted to the proofs of Propositions~\ref{p:stat} and \ref{p:stat-ext}.

\begin{proof}[Proof of Proposition~\ref{p:stat}]
We first describe the main idea.
In both cases, the basic observation is that we can use the relation $ \iota_V F = 0$ to
  change the Yang--Mills equation to be elliptic in spacetime. More precisely, $A$ obeys the Yang--Mills equation with the (inverse) metric changed from $\bfm^{-1}$, which is Lorentzian,
  to 
\begin{equation} \label{eq:stat-met}
  \bfe^{-1} = \bfm^{-1} - \frac{2}{\bfm(V, V)}V \otimes V.
\end{equation}
Indeed, by the variational formulation of the Yang--Mills equation, it suffices to show that, under the condition \eqref{eq:F-stat}, the Lagrangian with respect to $\bfm$ agrees with that with respect to $\bfe$, i.e.,
\begin{equation*}
(\bfm^{-1})^{\alp \alp'} (\bfm^{-1})^{\bt \bt'} F_{\alp \bt} F_{\alp' \bt'} \ud \hbox{Vol}_{\bfm}
= (\bfe^{-1})^{\alp \alp'} (\bfe^{-1})^{\bt \bt'} F_{\alp \bt} F_{\alp' \bt'} \ud \hbox{Vol}_{\bfe},
\end{equation*}
where $\ud \hbox{Vol}_{\bfm}$ (resp. $\ud \hbox{Vol}_{\bfe}$) is the volume form associated with $\bfm$ (resp. $\bfe$). This property is easily verified by choosing at each point an $\bfm$-orthonormal frame that includes $\hat{V} = \abs{\bfm(V, V)}^{-1/2} V$, which is also $\bfe$-orthonormal (and vice versa) by \eqref{eq:stat-met}.

Therefore, our connection $A$ is harmonic in $(\calO, \bfe)$.
  Since it has local regularity $A_{x} \in \ell^{1} H^{\frac32}_{loc}(\calO)$ and $\nb A_{0} \in \ell^{1} L^{2} H^{\frac{1}{2}}_{loc}(\calO)$, which is the critical
  regularity in $5$ dimensions, by an argument similar to Theorem~\ref{t:harmonic-reg} we conclude it is  
locally smooth in a suitable gauge, and thus globally smooth in a suitable gauge. 

In Case~(i), $\bfe$ is simply the Euclidean metric. In Case~(ii), $\bfe$ takes the form
\begin{equation*}
	\bfe = \ud \rho^{2} + \rho^{2} \ud y^{2}
\end{equation*}
in the hyperbolic polar coordinates $(\rho, y)$ on $C \subset \bbR^{1+4}$ (see Section~\ref{subsec:notation}), where $\ud y^{2}$ denotes the standard metric on the hyperbolic space $\bbH^{4}$ with sectional curvature $-1$. The one difference is that our elliptic problems now
 have variable coefficients.

For a more detailed analysis, we  implement the ideas above in three steps:

\pfstep{Step~1: Local Coulomb gauge with respect to $\bfe$} 
Let $x \in \calO'$. Our aim is to place the equations in a local Coulomb gauge with respect to the metric $\bfe$,
\begin{equation*}
	\covnb^{\alp} A_{\alp} = 0,
\end{equation*}
in a sufficiently small ball $B$ in $\calO$ centered at $x$ via a gauge transformation $O$ with regularity 
\begin{equation}\label{gauge-reg-0}
O_{;\alp} \in L^{5}(B), \quad \nb O_{;\alp} \in L^{2} H^{\frac{1}{2}}(B).
\end{equation}
Here and in the sequel, $\covnb$ is the Levi-Civita connection associated with the Riemannian metric $\bfe$ as in \eqref{eq:stat-met}, and we raise and lower greek (spacetime) indices using $\bfe$. 

Let $\bar{\bfe}_{\alp \bt} = \bfe_{\alp \bt}(x)$. We take the ball $B$ centered at $x$ to be small enough so that 
\begin{equation*}
\nrm{\nb A}_{L^{2} H^{\frac{1}{2}}_{B}} + \nrm{A}_{L^{5}(B)}, \quad
\nrm{\bfe_{\alp \bt} - \bar{\bfe}_{\alp \bt}}_{C^{2}(2B)}
\end{equation*}
are sufficiently small (here, $2B$ is the double enlargement of $B$). Then we may find extensions of $A$ and $\bfe$ from $B$ to $\bbR^{5}$ such that $\bfe - \bar{\bfe}$ is supported in $2B$ and $\nrm{\nb A}_{L^{2} \dot{H}^{\frac{1}{2}}} + \nrm{\bfe - \bar{\bfe}}_{C^{2}}$ is small. We look for a global gauge transformation $O$ into the Coulomb gauge; note that $\Omg_{\alp} = O^{-1} \rd_{\alp} O$ must solve
\begin{equation} \label{eq:div-curl-O-5d}
	\left\{
	\begin{aligned}
		\covnb^{\alp} \Omg_{\alp} &= \covnb^{\alp} A_{\alp} + [\Omg^{\alp}, A_{\alp}], \\
	\covnb_{\alp} \Omg_{\bt} - \covnb_{\bt} \Omg_{\alp} &= - [\Omg_{\alp}, \Omg_{\bt}].
\end{aligned}
	\right.
\end{equation}
By Proposition~\ref{prop:div-curl-nonlin-5d} and the smallness properties of $A$, $\bfe - \bar{\bfe}$ (as well an extra iteration procedure to include the term $[\Omg^{\alp}, A_{\alp}]$), we may find a unique solution $\Omg$ to \eqref{eq:div-curl-O-5d} such that $\Omg \in L^{5}$ and $\nb \Omg \in L^{2} \dot{H}^{\frac{1}{2}}$. Integrating the system of ODEs $O^{-1} \rd_{\alp} O = \Omg$, for which the curl condition serves as the compatibility condition needed for integrability\footnote{To make this procedure rigorous, one first approximate $A$ by smooth $1$-forms, so that the corresponding $\Omg$'s are smooth, then take the limit by compactness using the bounds on $\Omg$.}, we find a gauge transformation $O$ in $B$ with regularity 
\begin{equation}\label{gauge-reg-0}
O^{-1} \rd_{\alp} O \in L^{5}(B), \quad \nb (O^{-1} \rd_{\alp} O) \in L^{2} H^{\frac{1}{2}}(B),
\end{equation}
such that $\tA = \calG(O) A$ obeys the Coulomb condition in $B$. Note that $O_{;\alp} = Ad(O) (O^{-1} \rd_{\alp} O)$ has the regularity \eqref{gauge-reg-0} as well.

\pfstep{Step~2: Elliptic regularity} Here we examine the output connection
$\tilde A$, which a-priori has the same regularity as $O_{;\alpha}$, i.e.,
\begin{equation*}
\tilde{A} \in L^{5}(B), \quad \nb \tilde{A} \in L^{2} L^{\frac{8}{3}}(B),
\end{equation*}
solves the harmonic Yang--Mills equation
\[
\tilde{\covD}^\alpha \tF_{\alpha \beta} = 0,
\]
where $\tilde{\covD} = \covnb + ad(\tilde{A})$, and satisfies the Coulomb gauge condition
\begin{equation*}
	\covnb^{\alp} \tilde{A}_{\alp} = 0.
\end{equation*}
Together these form an elliptic system for $\tA_\alpha$ of the form
\[
\Delta_\tA \tA_\alpha - Ric[\bfe]_{\alp \bt} \tA^{\bt} = [\tA^\beta, \tilde{\covD}_\alpha \tA_\beta],
\]
where $\Delta_{\tA} = \tilde{\covD}^{\bt} \tilde{\covD}_{\bt}$ and $Ric[\bfe]$ is the Ricci curvature of $(B, \bfe)$. Since a-priori $\tilde A \in L^{5}(B)$, which is critical regularity, by (perturbative) elliptic theory it follows that the solutions 
are smooth in any smaller ball $B' \subset \overline{B'} \subset B$; for simplicity, we shrink the ball $B$ so that $\tA$ is smooth on $B$.

\pfstep{Step~3: The regularity of the gauge} 
So far, we have shown that at each point $x \in \calO'$, there exists a ball $B \ni x$ in $\calO$ and a gauge transform $O$ with regularity \eqref{gauge-reg-0} on $B$ such that $\tA = \calG(O) A$ is smooth. Here the goal is to use \eqref{A-reg} to boost this regularity to
\begin{equation} \label{eq:gauge-reg-goal}
\nb O_{;\alp} \in \ell^1 L^{2} H^{\frac{1}{2}}(B).
\end{equation}
By localization of the global-in-spacetime bound
\begin{equation} \label{eq:a0-trace}
	\nrm{u}_{\ell^{1} L^{\infty} \dot{H}^{1}} \aleq \nrm{\nb u}_{\ell^{1} L^{2} \dot{H}^{\frac{1}{2}}},
\end{equation}
which follows by applying the trace theorem to each $P_{k} u$ and summing up in $k$, it would follow that $O_{;\alp} \in \ell^{1} L^{\infty} H^{1}(B)$. By Lemma~\ref{lem:ptwise-O}, we would have $O \in C^{0}(B)$ as well, as desired.

By \eqref{A-reg} (as well as $A_{\alp} \in L^{\infty} H^{1}_{loc}(\calO)$), \eqref{gauge-reg-0}, smoothness of $\tA$ and the formula 
\begin{equation} \label{eq:O-A-tA}
	O_{;\alp} = Ad(O) A_{\alp} - \tA_{\alp},
\end{equation}
it follows that 
\begin{equation}\label{gauge-reg}
O_{;\alp} \in L^{\infty} H^{1}(B), \quad \nb O_{;\alpha} \in L^{2} H^{\frac{1}{2}}(B).
\end{equation}
To exploit the additional regularity $\nb A \in \ell^{1} L^{2} H^{\frac{1}{2}}(B)$ from \eqref{A-reg}, we note that the following global-in-spacetime bounds hold:
\begin{align*}
	\nrm{u v}_{\ell^{1} L^{2} \dot{H}^{\frac{1}{2}}} & \aleq \nrm{\nb u}_{L^{2} \dot{H}^{\frac{1}{2}}} \nrm{\nb v}_{L^{2} \dot{H}^{\frac{1}{2}}} \\
	\nrm{Ad(O) u}_{\ell^{1} L^{2} \dot{H}^{\frac{1}{2}}} & \aleq_{\nrm{O}_{L^{\infty}}, \nrm{O^{-1}}_{L^{\infty}}} \nrm{u}_{\ell^{1} L^{2} \dot{H}^{\frac{1}{2}}} + \nrm{\nb O_{;x}}_{L^{2} \dot{H}^{\frac{1}{2}}} \nrm{u}_{L^{2} \dot{H}^{\frac{1}{2}}} .
\end{align*}
Both estimates are straightforward to establish by standard Littlewood--Paley trichotomy, so we omit the proof. Localization of the above estimates, combined with \eqref{A-reg} and \eqref{gauge-reg}, imply the desired estimate \eqref{eq:gauge-reg-goal}. \qedhere 
\end{proof}

\begin{proof}[Proof of Proposition~\ref{p:stat-ext}]
Solving the nonlinear transport equation $V^{\alp} \rd_{\alp} O = O \iota_{V} A$ on $\calO$ (with arbitrary smooth data), we find a smooth gauge transformation $O$ on $\calO$ such that $\tA = \calG(O) A$ obeys $\iota_{V} \tA = 0$. By Cartan's formula, note that
\begin{equation*}
	\calL_{V} \tA =  \iota_{V} \ud \tA + \ud \iota_{V} \tA = \iota_{V} \tF - \frac{1}{2} \iota_{V} [\tA \wedge \tA] + \ud \iota_{V} \tA = 0,
\end{equation*}
where in the last equality, we used \eqref{eq:F-stat} for the first term and $\iota_{V} \tA = 0$ for the others. Thus the extension of $\tA$ to $\bigcup_{s \in \bbR} {}^{(V)} \Phi_{s}(\calO)$ by solving $\calL_{V} \tA = 0$, which amounts to solving an ODE along each integral curve of $V$, is well-defined. That the extension satisfies $\iota_{V} \tA = 0$ is clear, whereas \eqref{eq:F-stat} follows by reversing the preceding computation. By Cartan's formula applied to $\calL_{V} F$, it also follows that $\calL_{V} F = 0$.

It remains to show that the extension still solves the Yang--Mills equation. First, note that \eqref{ym} may be rewritten in the form
\begin{equation*}
	\tilde{\bfd} (\star F) = 0,
\end{equation*}
where $\tilde{\bfd}$ is the covariant exterior derivative associated with $\tA$\footnote{This operator is characterized by linearity and $\tilde{\bfd} (u \otimes \omg) = \tilde{\covD} u \wedge \omg + u \ud \omg$, where $u$ is a $\g$-valued function and $\omg$ is a $k$-form.} and $\star$ is the Hodge star operator associated with $\bfm$. By $[\ud, \calL_{V}] = 0$ and $\calL_{V} \tA = 0$, it follows that $[\calL_{V}, \tilde{\bfd}] = 0$. Moreover, since $\calL_{V} \bfm = 0$ in Case~(i) and $\calL_{V} \bfm = 2 \bfm$ in Case~(ii), it follows that $[\calL_{V},  \star] F = c \star F$, where $c = 0$ in Case~(i) and $c = 1$ (i.e., the spacetime dimension minus $4$) in Case~(ii). In conclusion, $\calL_{V} (\tilde{\bfd} (\star F)) = c \tilde{\bfd} (\star F)$, from which the desired conclusion follows. \qedhere
\end{proof}

\section{No finite energy self-similar solutions}\label{sec:self-similar}

One of the main enemies in proving the bubbling-off result is given by self-similar solutions.
Here we prove that no nontrivial finite energy self-similar Yang--Mills connections exist:

\begin{theorem}\label{t:no-self}
 There are no smooth nontrivial self-similar solutions
to the hyperbolic Yang--Mills equation (i.e., $\iota_{S} F = 0$) defined on the whole forward light cone $C$ which have finite energy.  
\end{theorem}
In our application, smoothness of the self-similar solution in $C$ (the \emph{open} forward light cone) follows from the results in Section~\ref{sec:reg}, but the only information a-priori available near the boundary is the finite energy condition. The main issue in the proof of Theorem~\ref{t:no-self} is indeed the analysis near the boundary.
\begin{proof}
We proceed in several steps.
\pfstep{Step~1}
We first recast the problem in hyperbolic polar coordinates, parametrizing the forward light
cone $C$ as
\[
C = \{ \rho y \in \bbR^{1+4}: \rho \in (0, \infty), \, y \in \bbH^4\},
\]
where we remind the reader that $\rho = \sqrt{t^{2} - \abs{x}^{2}}$ and $\bbH^{4} = \set{(t, x) \in C : \rho = 1}$ (see Section~\ref{subsec:notation}). The Minkowski metric becomes 
\[
\ud s^2 = - \ud \rho^2 + \rho^{2} \ud y^2,
\]
where $\ud y^{2}$ denotes the induced metric on $\bbH^{4}$; as is well-known, it is the standard metric on the hyperbolic space $\bbH^{4}$ with sectional curvature $-1$. 

Inside the light cone $C$, the self-similarity condition $\iota_S F = 0$ becomes $\iota_{ \partial_\rho} F = 0$.
By Proposition~\ref{p:stat-ext}, we may make a smooth gauge transformation to make $A_{\rho} = \iota_{\partial_{\rho}} A = 0$.
Then our connection is still smooth, and also independent of $\rho$ in the sense that $\calL_{\rd_{\rho}} A = 0$ (see the proof of Proposition~\ref{p:stat-ext}). We may furthermore check that the pullback of $A$ to $\bbH^{4} = \set{\rho = 1}$, which we still denote by $A$, is a solution to the harmonic Yang--Mills equation in $\bbH^{4}$.

Using the stereographic projection, we represent $\bbH^4$ as the unit disc $\bbD^4$ in $\R^4$ with metric 
\[
\ud s^2 =  \Omega^2 \ud x^2, \qquad \Omega = \frac{2}{1-\abs{x}^2}.
\]
By conformal invariance of the harmonic Yang--Mills equation in dimension $4$, the conformal factor drops out, and we obtain the 
elliptic Yang--Mills system in $\bbD^4$ 
\[
\covD^i F_{ij} =0.
\]
with respect to the Euclidean metric $\ud s^{2} = \ud x^{2}$.

We now move the finite energy condition in hyperbolic polar coordinates, then translate it to $\bbD^{4}$. This computation is equivalent 
to that in \cite[Section~7.2]{OT3}, and yields
\begin{equation} \label{eq:F-vanish}
\int_{\bbD^4} \frac{1+r^2}{1-r^2} |F|^2 \ud x < \infty,
\end{equation}
where $r = \abs{x}$. At this point we know that the connection is smooth inside $\bbD^4$, but nothing about its behavior at the boundary.
Let $\calA_{(\frac{1}{2}, 1)} = \set{x \in \bbD^{4} : \frac{1}{2} < \abs{x} < 1}$. We claim that there exists a gauge such that $A \in H^{1}(\calA_{(\frac{1}{2}, 1)})$ and
\begin{equation} \label{eq:A-vanish}
	\int_{\calA_{(\frac{1}{2}, 1)}} \frac{1}{1-r^{2}} \abs{\nb A}^{2} + \frac{1}{(1-r^{2})^{3}} \abs{A}^{2} \, \ud x < \infty.
\end{equation}
Assuming the claim, the proof of the theorem may be completed as follows.
By \eqref{eq:A-vanish}, it follows that $A \vert_{\partial \bbD^4}$ vanishes. Thus its zero extension $\bar{A}$ outside the ball
is also in $H^1$, and its curvature $\bar{F}$ is the zero extension of $F$. We conclude that the zero extension of $A \in H^{1}(\bbD^{4})$ satisfying \eqref{eq:A-vanish} still solves the harmonic Yang--Mills system
\[
\bar{\covD}^i \bar{F}_{i j} = 0.
\]
By the classical elliptic regularity results of Uhlenbeck (Theorem~\ref{t:harmonic-reg}), the connection $\bar{A}$ is gauge equivalent to a 
smooth connection $\tA$ in $\bbR^{4}$. To continue we write elliptic equations for $\tF$, 
\[
\Delta_{\tA} \tF = [\tF, \tF]
\]
Since $\tF$ has compact support and $\tA$ is smooth (thus bounded in the support of $\tF$), it follows from the classical elliptic unique continuation result\footnote{For further references on unique continuation for second order elliptic PDEs, see, for instance, \cite{KoTa}} due to Aronszajn \cite{Aro} that $\tF= 0$.  Thus the connection $A$ is trivial. 

It remains to prove our claim, and show that 
a representation satisfying \eqref{eq:A-vanish} exists.

\pfstep{Step~2}
For $0 < d < \frac{1}{6}$, $k \geq 0$ and $2 \leq p \leq \infty$, we claim that
\begin{equation} \label{eq:F-cov-w}
	\nrm{\covD^{(k)} F}_{L^{p}(\calA_{(1-2d, 1-d)})} \aleq d^{-\alp + \frac{1}{2}} \nrm{d^{-\frac{1}{2}} F}_{L^{2}(\calA_{(1-4d, 1-\frac{d}{2})}}), \quad \alp = k + 2 - \frac{4}{p}.
\end{equation}
This is proved by applying Uhlenbeck's lemma (Theorem~\ref{t:uhl}) in balls of size proportional to their distance to the boundary, and using the fact that the harmonic Yang--Mills equation becomes strictly elliptic in the Coulomb gauge, which allows us to use interior elliptic regularity.

\pfstep{Step~3}
In the remainder of this proof, we work in the polar coordinates $x = r \Tht$ on $\bbD^{4} \setminus \set{0}$. As stated in  Section~\ref{subsec:notation}, we write $A_{\Tht}(r, \cdot)$ for the pullback of a $1$-form $A$ to each constant $r$-sphere, which \emph{we then view as a $1$-form on the unit sphere $(\bbS^{3}, \smet)$.} Alternatively, one may think of the whole analysis in the rest of this proof as taking place on $(0, 1)_{r} \times \bbS^{3}_{\Tht}$ equipped with the metric $\ud s^{2} = \ud r^{2} + \ud \Tht^{2}$.

Roughly speaking, the idea is to fix the gauge by specifying the conditions
\begin{equation} \label{eq:self-sim-gauge}
	A_{r} = 0 \hbox{ in } \calA_{(\frac{1}{2}, 1)}, \quad
	A (r, \cdot) \to 0 \hbox{ and } \snb^{\Tht} A_{\Tht}(r, \cdot) \to 0 \hbox{ as } r \to 1,
\end{equation}
which is possible, at least heuristically, by the decay of $F$ as $r \to 1$ and Uhlenbeck's lemma.
For technical reasons, however, we proceed slightly differently and work with a sequence of gauges approximately satisfying \eqref{eq:self-sim-gauge}.

We start with the connection $A^{(0)}$ in the exponential gauge at the origin (i.e., $A^{(0)}_{r} = 0$ in $\bbD^{4}$ and $A^{(0)}(0) = 0$). Observe that $A^{(0)}$ is smooth. By \eqref{eq:F-vanish}, we may find a sequence $r_{n} \to 1$ such that $\nrm{(1-r_{n})^{-\frac{1}{2}} F(r_{n}, \cdot)}_{L^{2}_{\Tht}} \to 0$. Viewing $A_{\Tht}^{(0)}(r_{n}, \cdot)$ as a connection 1-form on the unit sphere $\Tht \in \bbS^{3}$ and applying Uhlenbeck's lemma on $\bbS^{3}$ (Proposition~\ref{p:uhl-S3}), we find gauge transformations $O^{(n)} = O^{(n)}(\Tht)$ on $\bbS^{3}$ such that the following property holds: Viewing $O^{(n)}$ as defined on $\bbD^{4} \setminus \set{0}$ by $O^{(n)}(r, \Tht) = O^{(n)}(\Tht)$, the representation $A^{(n)} = \calG(O^{(n)}) A^{(0)}$ obeys
\begin{equation} \label{eq:self-sim-gauge-approx}
	\snb^{\Tht} A^{(n)}_{\Tht}(r_{n}, \cdot) = 0, \quad 
	\nrm{\snb_{\Tht} A^{(n)}_{\Tht}(r_{n}, \cdot)}_{L^{2}_{\Tht}} \aleq \nrm{F(r_{n}, \cdot)}_{L^{2}_{\Tht}},
\end{equation}
as well as
\begin{equation} \label{eq:self-sim-gauge-approx-radial}
	A^{(n)}_{r} (r, \cdot) = 0 \quad \hbox{ for } 0 < r < 1,
\end{equation}
simply due to $A^{(0)}_{r} = 0$ and the $r$-independence of $O^{(n)}$. Thanks to smoothness of $A^{(0)}$ and \eqref{eq:F-cov-w} (elliptic regularity for $F$), note that $O^{n} \in C^{\infty}(\bbS^{3})$.

\pfstep{Step~4}
Let
\begin{equation*}
	D^{2} = \int_{\calA_{(\frac{1}{3}, 1)}} \frac{1}{1-r} \abs{F}^{2} \, \ud x, \quad
	\eps_{n}^{2} = \frac{1}{1-r_{n}}\int_{\bbS^{3}} \abs{F}^{2}(r_{n}, \Tht) \, \ud \Tht,
\end{equation*}
where $\calA_{(\frac{1}{3}, 1)} = \set{x \in \bbD^{4} : \frac{1}{3} < \abs{x} < 1}$. By \eqref{eq:F-vanish}, $D^{2} < \infty$, and by construction, $\eps_{n} \to 0$.

From now on, we work in the gauge constructed in the previous step. Without loss of generality, we may assume that $r_{n} > \frac{1}{2}$. 
Note that on the annulus $\calA_{(\frac{1}{2}, r_{n})} = \set{x \in \bbD^{4} : \frac{1}{2}< r < r_{n}}$, we have the equivalence $\nrm{g}_{L^{p}(\calA_{(\frac{1}{2}, r_{n})})} \aeq \nrm{g}_{L^{p}_{r} L^{p}_{\Tht}((\frac{1}{2}, r_{n}) \times \bbS^{3})}$ for any $1 \leq p \leq \infty$, where the constant is independent of $n$.

As $A^{(n)}_{r} =0$ and $\rd_{r} A^{(n)}_{\Tht} = F^{(n)}_{r \Tht}$, we immediately have
\begin{equation} \label{eq:A-w-rad}
	\int_{1/2}^{r_{n}} \int \frac{1}{1-r} \abs{\rd_{r} A^{(n)}}^{2} \, \ud r \ud \Tht \aleq D^{2}.
\end{equation}

To proceed, recall the following elementary inequality (essentially one-dimensional Hardy's inequality): For $1 \leq p \leq \infty$ and $0 < r \leq r_{n}$, we have
\begin{equation} \label{eq:hardy-1d}
	\int_{r}^{r_{n}} \left( (1-r')^{\bt - 1} \abs{g} \right)^{p} \, \ud r' \aleq_{\bt, p} \int_{r}^{r_{n}} \left( (1-r')^{\bt} \abs{\rd_{r} g} \right)^{p} \, \ud r'
	+ \left( (1-r_{n})^{\bt - \frac{p -1}{p}} \abs{g} \right)^{p}(r_{n})
\end{equation}
provided that $\bt < \frac{p-1}{p}$. 

In our gauge, \eqref{eq:hardy-1d} implies
\begin{equation} \label{eq:A-w-rad-0}
	\int_{1/2}^{r_{0}} \int \frac{1}{(1-r)^{3}} \abs{A^{(n)}}^{2} \, \ud r \ud \Tht \aleq D^{2} + \frac{1}{(1-r_{0})^{2}} \int \abs{A^{(n)}}^{2}(r_{0}, \Tht) \, \ud \Tht,
\end{equation}
for any $1/2 < r_{0} \leq r_{n}$.

\pfstep{Step~5} 
To complete the proof, in view of \eqref{eq:A-w-rad} and \eqref{eq:A-w-rad-0}, it remains to establish
\begin{equation} \label{eq:A-w-ang}
	\limsup_{n \to \infty} \int_{1/2}^{r_{n}} \int \frac{1}{1-r^{2}} \abs{\snb_{\Tht} A^{(n)}}^{2} \, \ud r \ud \Tht \aleq_{D} 1.
\end{equation}
Once \eqref{eq:A-w-ang} is proved, then it is a routine matter to extract a limit $O^{(n)} \weakto O$ in $H^{2}(\bbD^{4})$ such that $A = Ad(O) A^{(0)} - O_{;x}$ obeys the desired vanishing condition \eqref{eq:A-vanish}. 

As a first attempt to prove \eqref{eq:A-w-ang}, note that we have control of $\rd_{r} \rd_{\Tht} A^{(n)}_{\Tht} = \covD^{(n)}_{\Tht} F^{(n)}_{r \Tht} + O(A^{(n)}_{\Tht}, F^{(n)}_{r \Tht}$) by \eqref{eq:F-cov-w}. However, if we naively use the $L^{2}$ bound in \eqref{eq:F-cov-w}, we encounter a logarithmic divergence. To rectify this, we use an additional cancellation from the harmonic Yang--Mills equation.

The idea is to compute the div-curl system on $\bbS^{3}$ satisfied by $A^{n}_{\Tht}(r, \cdot)$. First, note that
\begin{equation} \label{eq:div-FrTht}
	\scovD^{\Tht} F_{r \Tht} = r^{2} \covD^{j} (\iota_{\rd_{r}}F)_{j} = r x^{j} \left( \covD^{i} F_{ij} \right),
\end{equation}
where the last term is zero if the harmonic Yang--Mills equation holds. Therefore, we have
\begin{align*}
	\rd_{r} \snb^{\Tht} A^{(n)}_{\Tht} = \snb^{\Tht} F^{(n)}_{r \Tht} 
	=  (\scovD^{(n)})^{\Tht} F^{(n)}_{r \Tht} - ad(A^{(n) \Tht}) F^{(n)}_{r \Tht}) = - ad(A^{(n) \Tht}) F^{(n)}_{r \Tht}.
\end{align*}
We furthermore note that
\begin{equation} \label{eq:A-w-L4}
	\int_{1/2}^{r_{n}} \int \frac{1}{(1-r)^{2}} \abs{A^{(n)}}^{4} \, \ud r \ud \Tht \aleq D^{4} + \frac{1}{1-r_{n}} \int \abs{A^{(n)}}^{4}(r_{n}, \Tht) \, \ud \Tht \aleq D^{4} + (1-r_{n}) \eps_{n}^{4}
\end{equation}
by \eqref{eq:F-cov-w} with $(k, p) = (0, 4)$, \eqref{eq:hardy-1d} with $p = 4$ and \eqref{eq:self-sim-gauge-approx}. Then thanks to \eqref{eq:F-cov-w} and \eqref{eq:A-w-L4}, we have 
\begin{equation*}
\limsup_{n \to \infty} \nrm{\rd_{r} \snb^{\Tht} A^{(n)}_{\Tht}}_{L^{2}_{r} L^{2}_{\Tht}((1/2, r_{n}) \times \bbS^{3})} \aleq D^{2}.
\end{equation*}
Recall from \eqref{eq:self-sim-gauge-approx} that $\snb^{\Tht} A^{(n)}_{\Tht}(r_{n}, \Tht) = 0$. By \eqref{eq:hardy-1d}, it follows that
\begin{equation*}
	\limsup_{n \to \infty} \nrm{(1-r)^{-1} \snb^{\Tht} A^{(n)}_{\Tht}}_{L^{2}_{r} L^{2}_{\Tht}((1/2, r_{n}) \times \bbS^{3})} \aleq D^{2}.
\end{equation*}
On the other hand, by the schematic relation
\begin{equation*}
	(\ud A)_{\Tht \Tht} = F_{\Tht \Tht} - [A_{\Tht}, A_{\Tht}]
\end{equation*}
and the bound \eqref{eq:A-w-L4}, we have
\begin{equation*}
	\limsup_{n \to \infty} \nrm{(1-r)^{-\frac{1}{2}} (\ud A^{(n)})_{\Tht \Tht}}_{L^{2}_{r} L^{2}_{\Tht} ((1/2, r_{n}) \times \bbS^{3})} \aleq D + D^{2}.
\end{equation*}
By the div-curl estimate on $\bbS^{3}$, we obtain
\begin{equation*}
	\limsup_{n \to \infty} \nrm{(1-r)^{-\frac{1}{2}} \snb_{\Tht} A^{(n)}_{\Tht}}_{L^{2}_{r} L^{2}_{\Tht}((1/2, r_{n}) \times \bbS^{3})} \aleq D + D^{2},
\end{equation*}
which implies the desired bound \eqref{eq:A-w-ang}.
\end{proof}

\section{The bubbling-off result}
\label{sec:bubble-off}
In this section we prove the bubbling off result in
Theorem~\ref{t:bubble-off}. Much of the argument is similar to that in \cite[Sections~8.3--8.6]{OT3} (see also \cite[Sections~6.5--6.8]{ST2}), from which we borrow many results.

Throughout this section, we assume that $A$ is a finite energy Yang--Mills connection satisfying the hypothesis of either Theorem~\ref{t:bubble-off}.a) (finite time blow-up at $(T, X)$) or b) (infinite time blow-up). We write $E = \nE(A)$ and $E_{1} = \limsup_{t \nearrow T} \nE_{C_{\gmm} \cap S_{t}} (A)$, where $T = \infty$ in the infinite time blow-up case. Moreover, in the finite time blow-up case, we translate the point $(T, X)$ to $(0, 0)$, and reverse the time direction so that the blow up occurs inside $C$ backward in time towards $(0, 0)$.

Our first goal is to prove that from the
connection $A$ we can extract a sequence of smooth connections $A^{(n)}$ in
increasing cone sections and with decreasing fluxes:

\begin{lemma}
There exists a sequence of smooth hyperbolic Yang--Mills connections $A^{(n)}$ in cone sections
$C_{[\veps_n,1]}$ with $\veps_n \to 0$ which satisfy the following properties:
 \begin{enumerate}
\item Closeness to $A$. There exists a sequence $\tA^{(n)}$ of rescaled and translated 
copies of $A$ so that 
\begin{equation} \label{eq:ini-seq:approx}
\lim_{n \to \infty} \sup_{t \in [\veps_n,1]} \| A^{(n)} -\tA^{(n)}\|_{\dot H^1 \cap L^4(S_t)} = 0 
\end{equation}
  \item Bounded energy in the cone.
    \begin{equation} \label{eq:ini-seq:energy} \nE_{S_{t}} (A^{(n)}) \leq E + o(1) \quad \hbox{ for every }
 t \in [\veps_{n}, 1],
    \end{equation}
   \item Decaying flux on $\rd C$.
    \begin{equation} \label{eq:ini-seq:flux} \calF_{[\veps_{n}, 1]}
      (A^{(n)})  \leq \veps_{n}^{\frac{1}{2}} E,
    \end{equation}
  \item Time-like energy concentration at $t = 1$.
\begin{equation}
 \nE_{C_{\gmm} \cap S_1} (A^{(n)}) \geq E_1 > 0
\end{equation}
with some  $\gamma < 1$.
  \end{enumerate}
\end{lemma}

\begin{proof}
To clarify the ideas we assume at first that $A$ is smooth in the closure of $C$. We start from the flux-energy relation 
\eqref{energy-flux}, which shows that the flux decays toward the tip of the cone in the blow-up case,
\[
\lim_{t \searrow 0} \calF_{(0, t]}(A) = 0.
\]
respectively toward infinity in the non-scattering case,
\[
\lim_{t \nearrow \infty} \calF_{[t,\infty)}(A) = 0.
\]
Using these properties, it easily follows in both cases that we can find a sequence of connections $A^{(n)}$ 
 which are obtained from $A$ simply by rescaling.

Suppose now that $A$ is a finite energy solution inside the cone. Then its energies on each time slice
$\nE_{S_t}(A)$ are still well defined and nondecreasing in $t$. Thus its fluxes are at least formally defined via the 
energy flux relation \eqref{energy-flux}, therefore we obtain the sequence $\tA^{(n)}$ of rescaled copies
of $A$  which satisfies the properties (2)--(4). We now consider a smooth approximation $\tA^{(n)}_\delta$ of $\tA^{(n)}$
so that 
\[
\sup_{t \in [\veps_n+\delta ,1-\delta]} \| \tA^{(n)}_\delta -\tA^{(n)}\|_{\dot H^1 \cap L^4(S_t^\delta)} \leq \frac{1}{n}.
\] 
Then the desired smooth connections $A^{(n)}$ are obtained by slightly
(by $O(\delta) $ to be precise) translating and rescaling
$\tA^{(n)}_\delta$ provided that $\delta$ is small enough, depending
on $n$. \qedhere
\end{proof}

At this point, we may apply Proposition~\ref{prop:monotonicity} to $A^{(n)}$ and obtain
\begin{equation} \label{eq:monotonicity-n}
\iint_{C_{[\veps_{n}, 1]}} \frac{2}{\rho_{\veps_{n}}} \abs{\iota_{X_{\veps_{n}}} F^{(n)}}^{2} \ud t \ud x \aleq E.
\end{equation}
This property implies a decay of $F^{(n)}$ towards the tip of the cone $C$ for large $n$, which is the key ingredient of the proof (see Lemma~\ref{lem:final-rescale}.(4) below).

Next, we show that the energy concentration at time $t = 1$ persists in time:

\begin{lemma} \label{lem:t-like-e}
Let $A^{(n)}$ be the sequence of smooth Yang--Mills connections in the previous lemma.
 Then there exist $E_{2} > 0$ and $\gmm_{2}  \in (0, 1)$ such that
\begin{equation} \label{eq:t-like-e}
	\int_{C_{\gmm_{2}} \cap S_{t}} \mvC{X_{0}}_{0}(A^{(n)})  \, \ud x \geq E_{2} \quad \hbox{ for every } t \in [\veps_{n}^{\frac{1}{2}}, \veps_{n}^{\frac{1}{4}}].
\end{equation}
\end{lemma}

This is a gauge independent property, which follows from the localized
monotonicity formulas as in the (MKG) case, via
Proposition~\ref{prop:monotonicity:t-like}; see \cite[Proof of Lemma~8.10]{OT3} for details.

At this point, we can freely replace
$\mvC{X_{0}}$ in \eqref{eq:t-like-e} by $\mvC{T}$ (i.e., the energy density) at the expense of adjusting $E_2$ as we are away from the
cone. Now a final rescaling leads us to

\begin{lemma} \label{lem:final-rescale} There exists a sequence of smooth Yang--Mills connections $A^{(n)}$ on $[1, T_{n}] \times
  \bbR^{4}$ with $T_{n} \to \infty$ satisfying the following
  properties:
\begin{enumerate}
\item Closeness to $A$. There exists a sequence $\tA^{(n)}$ of rescaled and translated 
copies of $A$ so that 
\begin{equation} \label{eq:ini-seq:approx+}
\lim_{n \to \infty} \sup_{t \in [1,T_n]} \| A^{(n)} -\tA^{(n)}\|_{\dot H^1 \cap L^4(S_t)} = 0 
\end{equation}
\item Bounded energy in the cone,
\begin{equation} \label{eq:final-rescale:energy}
	\nE_{S_{t}}(A^{(n)})\leq E + o(1), \quad
	 \quad \hbox{ for every } t \in [1, T_{n}],
\end{equation}

\item Nontrivial energy in a time-like region,
\begin{equation} \label{eq:final-rescale:nontrivial}
	\nE_{C_{\gmm_{2}} \cap S_{t}} (A^{(n)})  \geq E_{2} \quad \hbox{ for every } t \in [1, T_{n}],
\end{equation}
\item Asymptotic self-similarity,
\begin{equation} \label{eq:final-rescale:asymp-ss}
	\iint_{K} \abs{\iota_{X_{0}} F^{(n)}}^{2} \ud t \ud x \to 0 \quad \hbox{ as } n \to \infty
\end{equation}
	for every compact subset $K$ of the interior of $C_{[1, \infty)}$.
\end{enumerate}
\end{lemma}

Here, $E_{2}$ is the constant from Lemma~\ref{lem:t-like-e}, after making an adjustment mentioned before Lemma~\ref{lem:final-rescale}. This lemma is essentially rescaling and pigeonhole principle; see \cite[Proof of Lemma~8.11]{OT3}. Note that Property~(4) follows from \eqref{eq:monotonicity-n}, which in turn was a consequence of the monotonicity formula (Proposition~\ref{prop:monotonicity}).

To proceed, we introduce few definitions. For each $j=1, 2, \ldots$, let
\begin{align*}
	C_{j} =& \set{(t, x) \in C^{1}_{[1, \infty)} : 2^{j} \leq t < 2^{j+1}}, \\
	\tilde{C}_{j} = & \set{(t, x) \in C^{1/2}_{[1/2, \infty)} : 2^{j} \leq t < 2^{j+1}}.
\end{align*}
Note that $C_{j}$, respectively $\tilde{C}_{j}$, is simply the set of points in the truncated cone $C_{[2^{j}, 2^{j+1})}$ at distance $\geq 1$, respectively $\geq \frac{1}{2}$, from the lateral boundary $\rd C$.

We have the following lemma, which is basically \cite[Lemma~8.12]{OT3}, for understanding concentration scales:

\begin{lemma} \label{lem:conc-scales} Let $A^{(n)}$ be a sequence of
hyperbolic Yang--Mills connections as in the previous lemma.
 Let $E_{0}$ be sufficiently small.  Then for each $j = 1, 2, \cdots$, after
  passing to a subsequence, one of the following alternatives holds:
\begin{enumerate}
\item {Concentration of energy:} There exist points $(t_{n}, x_{n}) \in \widetilde{C}_{j}$, scales $r_{n} \to 0$ and $0 < r = r(j) < 1/4$ such that the following bounds hold:
\begin{align} 
	\nE_{\set{t_{n}} \times B_{r_{n}}(x_{n})}(A^{(n)}) = & \ \eps_{0}, \label{eq:conc-scales:conc:1} \\
	\sup_{x \in B_{r}(x_{n})} \nE_{\set{t_{n}} \times B_{r_{n}}(x)}(A^{(n)}) \leq & \ \eps_{0},
\label{eq:conc-scales:conc:2} \\
	\frac{1}{4 r_{n}} \int_{t_{n}-2 r_{n}}^{t_{n}+2 r_{n}} \int_{B_{r}(x_{n})} 
				\abs{\iota_{X_{0}} F^{(n)}}^{2}  \, \ud t \ud x 
	\to & \ 0 \quad \hbox{ as } n \to \infty. \label{eq:conc-scales:conc:3}
\end{align}

\item {Uniform non-concentration of energy:} There exists $0 < r =
  r(j) < 1/4$ such that the following bounds hold:
\begin{align} 
\nE_{C_{\gmm} \cap S_{t}} (A^{(n)}) \geq  &\  E_{2} 
\quad \hbox{ for } t \in [2^{j}, 2^{j+1}), \label{eq:conc-scales:non-conc:1}\\
\sup_{(t, x) \in C_{j}} \nE_{\set{t} \times B_{r}(x)} (A^{(n)}) \leq & \ \eps_{0}, \label{eq:conc-scales:non-conc:2}\\
\iint_{\widetilde{C}_{j}} \abs{\iota_{X_{0}} F^{(n)}}^{2} \, \ud t \ud x \to & \ 0 \quad \hbox{ as } n \to \infty. 
\label{eq:conc-scales:non-conc:3}
\end{align}
\end{enumerate}

\end{lemma}

In applying this lemma there are two scenarios we need to consider. Either 

\begin{enumerate}
\item[(i)] Property (1) holds for some $j$, or
\item[(ii)] Property (2) holds for all $j$.
\end{enumerate}

\medskip

{\em (i) Concentration scenario.} 
Now we need to run a compactness argument. 
On a subsequence we can assume that 
\[
\lim_{n\to \infty}  \frac{x_n}{t} = v, \qquad   |v| < 1.
\]
We denote $V = (1,v)$, which is a future pointing time-like vector.

We restrict the connections $A^{(n)}$  to the regions $ [t_n-r_n,t_n+r_n] \times B_{r}(x_n)$, with fixed  $r \ll 1$.
Then we rescale to unit time  and translate to center them at $(0,0)$. 
We obtain a sequence of Yang--Mills connections $A^{(n)}$ in the time interval $[-1,1]$ with the properties that
\begin{align} 
	\nE_{\set{0} \times B_{1}(0)}(A^{(n)}) = & \  E_0, \label{conc:1} \\
	\sup_{x \in B_{R_n}(0)} \nE_{\set{0} \times B_{1}(x)} (A^{(n)}) \leq & \ E_0, \label{conc:2} \\
 \int_{-1}^{1} \int_{B_{R_n}(0)} \abs{\iota_{V} F^{(n)}}^{2}  \, \ud t \ud x 
	\to & \ 0 \quad \hbox{ as } n \to \infty, \label{conc:3}
\end{align}
where $R_n = r_{n}^{-1} r \to \infty$. In \eqref{conc:3}, we replaced $X_{0}$ by $V = \sqrt{1-\abs{v}^{2}} X_{0} \vert_{(1, v)}$ by the convergence $\frac{x_{n}}{t} \to v$, the decay $r_{n} \to 0$ and $0$-homogeneity of $X_{0}$. 

Our next aim is to find a sequence of admissible gauge transformations such that, after passing to a subsequence and for some $A$ on $[-1/2, 1/2] \times \bbR^{4}$ such that
\[
A \in \ell^{1} H^\frac32_{loc}, \qquad 
A_0 \in \ell^1 H^{1,\frac12}_{loc},
\]
we have
\[
\calG(O^{(n)}) A^{(n)} \to A \qquad \text{in}\   H^1_{loc}([-1/2, 1/2] \times \bbR^{4}),
\]
and in addition 
\[
\iota_V F = 0 \qquad \text{in}\ [-1/2, 1/2] \times \bbR^{4}.
\]

So far, we have only used energy considerations, which are gauge independent. For the next step, however,
we need better regularity information, so for each $R = 1, 2, \ldots$, we place the above solutions in a ``good gauge'' $A^{R(n)}$, as provided 
by Theorem~\ref{t:good-gauge} on $[-1, 1] \times B_{2R}$ for $n$ sufficiently large $r/r_{n} \gg R$ (the theorem is stated on a truncated cone $C_{[t_{1}, t_{2}]}$, but the domain $[-1, 1] \times B_{2R}$ is essentially the same). We apply the compactness result in Theorem~\ref{t:compact} to conclude that
on a subsequence we have local convergence to a finite energy Yang--Mills connection $A^{R}$ in $[-1/2,1/2] \times B_{R}$. Moreover, for $R' < R$, note that $A^{R'(n)}$ and $A^{R(n)}$ are connected by a gauge transformation $O^{RR'(n)}$ in $B_{2R'}$ such that 
\begin{equation*}
\nb O^{RR'(n)}_{;t, x} \in \ell^{1} L^{2} H^{\frac{1}{2}}_{loc}([-1, 1] \times B_{2R'}),
\end{equation*}
with uniform bounds on compact subsets (indeed, for $O_{;t}$, we use the regularity $\nb A_{0} \in \ell^{1} L^{2} H^{\frac{1}{2}}$, whereas for $\rd O_{;x}$ we use $\rd^{k} A_{k} \in \ell^{1} L^{2} \dot{H}^{\frac{1}{2}}$ and Lemma~\ref{lem:div-curl-O}. Finally, for $\rd_{t} O_{;x}$ we use $\rd_{t} O_{;j} = \rd_{j} O_{;t} + [O_{;j}, O_{;t}]$). Then, passing to a subsequence, we obtain a gauge transformation $O_{;t,x}^{RR'} \in \ell^{1} L^{2} H^{\frac{1}{2}}_{loc}([-1/2, 1/2] \times B_{R'})$, which is admissible, such that $A^{R} =  \calG(O^{RR'}) A^{R'}$. By patching together $A^{R}$ for $R=1, 2, \ldots$ (see, for instance, \cite[Section~3.5, Scenario~(3)]{OTYM2.5}), we obtain a global solution $A$ on $[-1/2, 1/2] \times \bbR^{4}$, as desired.

By the regularity result in Proposition~\ref{p:stat}, the
connection $A$ is gauge equivalent to a smooth connection in the
domain $(-1/4, 1/4) \times \bbR^4$, which we still denote be $A$, that moreover satisfies

(i) Nontriviality, $\nE(A) > 0$;

(ii) Finite energy, $\nE(A) \leq E$;

(iii) Stationarity, $\iota_V F = 0$.

Applying Proposition~\ref{p:stat-ext}, we may place $A$ in the gauge $\iota_{V} A = 0$, and extend it to the whole spacetime $\bbR^{1+4}$. 
Then $A$ is a Lorentz transform of a nontrivial harmonic Yang--Mills connection $Q$, namely $A = L_v Q$.  The theorem is proved in this case.

\medskip
{\em (ii) Non-concentration scenario.}  The argument is similar
here. By Theorem~\ref{t:good-gauge} we can use gauge-equivalent representations of the connections
$A^{(n)}$ which are in a ``good gauge'' as provided by Theorem~\ref{t:good-gauge}, and thus are bounded in the sense of \eqref{good-gauge}, 
uniformly on compact subsets of $C_{[1,T)}^1$ for each $T = 2, 3, \ldots$. There applying the compactness result in
Theorem~\ref{t:compact} for each $T$, and patching together the resulting limits as in Case~(i), we obtain a global nontrivial self-similar, finite energy connection $A$ in $C_{[\frac{3}{2},\infty)}^\frac{3}{2}$. Applying the regularity result in
Proposition~\ref{p:stat}, and then Proposition~\ref{p:stat-ext} with $\calO' = C_{[2, \infty)}^{2}$, we obtain a smooth global self-similar
solution with finite energy inside the light cone $C$. The nontrivial
energy of $A^{(n)}$ inside the cone \eqref{eq:conc-scales:non-conc:1} insures that this limiting connection is
nontrivial. But such a connection does not exist by
Theorem~\ref{t:no-self}.

\section{No null concentration}\label{sec:no-null}

A key step in the transition from Theorem~\ref{t:bubble-off} to Theorems~\ref{t:threshold} and \ref{t:no-bubble}
is to deal with the possibility that the  energy stays concentrated near the boundary of the 
light cone. Whereas it is not implausible that the energy near the cone must necessarily decay to zero (in particular, see \cite{DJKM1} for the small data wave maps problem),
at this point we are not able to prove this. Instead, here we prove a weaker statement which asserts
that if almost all  energy stays near the cone, then our connection admits an energy dispersed 
caloric representation:

\begin{theorem} \label{thm:no-null}
Let $A$ be a  finite energy Yang--Mills connection on $\set{1} \times \bbR^{4}$. Suppose that
\begin{equation} \label{eq:no-null:hyp-small}
	\nE_{S_{1}^{\gmm}}(A) + \nE_{\set{1} \times \bbR^{4} \setminus S_{1}}(A) \leq \eps_{1},
\end{equation}
and
\begin{equation} \label{eq:no-null:hyp-bdd}
	\nE_{S_{1}}(A) + \int_{S_{1}} \mvC{X_{\veps}}_{0} (A) \, \ud x \leq  \En.
\end{equation}
Given any $\eps, \En > 0$, for sufficiently small $\eps_{1}, \veps$ and $\gmm$ close enough to $1$ (depending only on $\eps, \En$), there exists a caloric gauge representation of the connection $A_{x}(1)$
so that 
\[
\| A_{x}(1)\|_{\dot H^1} \lesssim_\En 1, \qquad \hM(A_{x}(1)) \lesssim_\En 1,
\]
whereas 
\[
\|A_{x}(1)\|_{L^{4}} + \|F(1)\|_{\dot W^{-1,4}}  \lesssim_\En \eps .
\]
\end{theorem}
We emphasize that the term $\nrm{F(1)}_{\dot{W}^{-1, 4}}$ contains both spatial and temporal components $F_{jk}$ and $F_{0j}$, respectively, of $F$.
In our application, control of the second term on the LHS of \eqref{eq:no-null:hyp-bdd} will come from the monotonicity formula (Proposition~\ref{prop:monotonicity}).

As an immediate consequence of the last bound, we obtain the smallness of the fixed-time energy dispersion in the caloric gauge:
\begin{corollary}
 The caloric connection $A(1)$ provided by the above theorem satisfies
\begin{equation} \label{eq:no-null}
	\sup_{k} 2^{-2k} \nrm{P_{k} F (1)}_{L^{\infty}} \lesssim \eps.
\end{equation}
\end{corollary}

The hypothesis of the theorem involves the full connection $A$ at time
$t=1$, which includes both information about $A_x$, $A_0$ and the
corresponding curvature components $F_{ij}$ and $F_{0j}$. Our first task is
to peel off the nonessential parts $A_0$ and $F_{0j}$ and to reduce the problem 
to a statement about only the spatial part of the connection. To state the result, we introduce an orthonormal frame $(e_{r}, e_{1}, e_{2}, e_{3})$ at every point of $\bbR^{4} \setminus \set{0}$, where $e_{r} = \rd_{r}$ in the polar coordinates $(r, \Tht)$ and $\set{e_{\frka}}_{\frka=1, 2, 3}$ is an orthonormal frame at $x$ tangent to the sphere $\rd B_{r}(0)$ (with $r = \abs{x}$). We also fix a small constant $0 < \dltnc \ll \frac{1}{100}$.

 \begin{proposition}\label{p:parabolic}
Let $A$ be an $\dot H^1$ connection in $\bbR^4$ with energy at most $\En$, which satisfy the following properties:

i) $F_{\frka \frkb} = F(e_{\frka}, e_{\frkb})$ $(\frka, \frkb = 1, 2, 3)$ is small in $L^2$,
\begin{equation}
\| F_{\frka \frkb} \|_{L^2} \leq \epsilon.
\end{equation}

ii)  $F_{r \frka} = F(e_{r}, e_{\frka})$ $(\frka=1,2,3)$ is small outside an annulus,
\[
\|  F_{r \frka}  \|_{L^2(\{ \frac78 \leq |x| \leq 1\}^{c})} \leq \epsilon.
\]

iii) $\covD^{\frka} F_{r \frka}$ (i.e., the covariant angular divergence) is small in $\dot H_A^{-1}$,
\begin{equation} \label{eq:small-div-FrTht}
\| \covD^{\frka} F_{r \frka} \|_{\dot H_A^{-1}} \leq \epsilon.
\end{equation}

Assume that $\epsilon$ is  sufficiently small,
\[
\epsilon \ll_{\En} 1.
\]
Then there exists a caloric gauge representation of the connection $A$
so that 
\begin{equation} 
\| A \|_{\dot H^1} \lesssim_\En 1,  \qquad \hM(A) \lesssim_\En 1,
\end{equation}
whereas\footnote{The factor $\frac{3}{8}$ can be improved to $\frac{3}{4}$ by using further techniques in \cite{OTYM1}, but for our purposes it is unnecessary.}
\begin{equation}	 \label{eq:parabolic-L4}
\|A\|_{L^4}  \lesssim_\En \eps^{\frac{3}{8} (1-\dltnc)}.
\end{equation}
\end{proposition} 

We remark that the assumptions in the proposition are all formulated in a gauge invariant fashion.
Most notably, assumption (iii) involves the space $\dot H^{-1}_A$, which is the dual of the space
$\dot H^1_A$ with norm
\[
\|B\|_{\dot H^1_A}^2 = \| \covD_A B\|_{L^2}^2 .
\]
In particular, nothing is assumed about the $\dot H^1$ size of $A$ and its various components.
This turns out to be a problem in the proof, where it would be very convenient to have as a starting point
a connection $A$ with some good bounds. To address this difficulty, 
 the key ingredient of the proof of the proposition is the following lemma, which we now state in the polar coordinates $x = r \Tht$:

\begin{lemma}\label{l:elliptic}
Let $A$ be a connection which satisfies the hypotheses of Proposition~\ref{p:parabolic}. Then there 
exists a gauge-equivalent connection $B$ which has the following properties:

\begin{enumerate}
\item $B$ is bounded in $\dot H^1$,
\[
\|B \|_{\dot H^1} \lesssim_\En 1.
\]

\item $B$ is small away from the unit sphere, 
\[
\|B \|_{\dot H^1(\{ \frac34 \leq |x| \leq 1\}^{c})} \lesssim_\En \epsilon .
\]

\item $B_r$ is small in $\calA_{(\frac{2}{3}, \frac{4}{3})}$,
\[
\|B_r \|_{\dot H^1(\calA_{(\frac{2}{3}, \frac{4}{3})})} \lesssim_\En \epsilon^{1-\dltnc}.
\]

\item $B_\Theta$ has small angular derivatives in $\calA_{(\frac{2}{3}, \frac{4}{3})}$,
\[
\| \snb_\Theta B_{\Theta} \|_{L^2(\calA_{(\frac{2}{3}, \frac{4}{3})})} + \nrm{r^{-1} B_{\Tht}}_{L^{2}(\calA_{(\frac{2}{3}, \frac{4}{3})})}  \lesssim_\En \epsilon^{1-\dltnc} .
\]
\end{enumerate}
\end{lemma}

We remark that the connection $B$ provided by the above lemma has the property that it is 
small in $L^4$.  This is explained in what follows.

From properties (2)--(4) and Hardy's inequality, it follows that each component $B_{j}$ in the rectangular coordinates obeys
\begin{equation*}
	\nrm{r^{-1} B_{j}}_{L^{2}} \aleq_{\En} \eps^{1-\delta_0}.
\end{equation*}
Thus we may localize $B_{j}$ via a smooth cutoff outside the annulus $\set{\frac{3}{4} \leq \abs{x} \leq 1}$, and show that the $L^{4}$ norm of this portion is small using property (2). To bound the 
$L^4$ norm of  the localized remainder, the following variant of the Sobolev (or Bernstein) inequality applies:
\begin{lemma} \label{l:bern}
	Let $u$ be supported in an annulus $\calA_{(r_{0}, r_{1})}$. Then
	\begin{equation*}
	\nrm{u}_{L^{4}} \aleq_{r_{0}, r_{1}} \nrm{u}_{\dot{H}^{1}}^{\frac{1}{4}} \left( \nrm{\snb_{\Tht} u}_{L^{2}} + \nrm{r^{-1} u}_{L^{2}} \right)^{\frac{3}{4}}.
\end{equation*}
\end{lemma}
\begin{proof}
In what follows, we suppress the dependence of constants on $r_{0}, r_{1}$. Using a smooth partition of unity in the angular variables, we may assume that $u$ is supported in an angular sector $\Gmm = \set{r \Tht \in \calA_{(r_{0}, r_{1})} : \Tht \in \kpp}$, where $\kpp$ is a spherical cap in $\bbS^{3}$. Then we may use a diffeomorphism from $\kpp$ to a ball $B \subset \bbR^{3}$ to map $\Gmm$ to $[-1, 1] \times B \subset \bbR^{4}$. 

We are left to prove
\begin{equation} \label{eq:bern-R4}
	\nrm{u}_{L^{4}} \aleq \nrm{u}_{H^{1}}^{\frac{1}{4}} \nrm{u}_{L^{2} H^{1}}^{\frac{3}{4}}
\end{equation}
for a function $u$ supported on the cylinder $[-1, 1] \times B \subset \bbR^{4}$; here, the mixed norms are defined with respect to $x^{1}$ and $x' = x^{2}, x^{3}, x^{4}$. By the Littlewood--Paley inequality, it suffices to verify this inequality for a single piece $P_{k} u$. We also introduce the Littlewood--Paley projections $P'_{j}$ associated with $x'$. Then by Bernstein's inequality,
\begin{equation*}
	\nrm{P_{k} P'_{j} u}_{L^{4}} \aleq 2^{\frac{1}{4} k} 2^{\frac{3}{4} j} \min \set{2^{-j} \nrm{u}_{L^{2} H^{1}}, 2^{-k} \nrm{u}_{H^{1}}},
\end{equation*}
and the LHS vanishes for $j \geq k+ O(1)$. Now summing up in $j$, the desired bound follows. \qedhere
\end{proof}
It follows that 
\[
\|B\|_{L^4}  \lesssim_\En \epsilon^{\frac{3}{4}(1-\dltnc)} ,
\]
which in turn shows that in this gauge,
$B$ is energy dispersed (i.e., $2^{-k} \nrm{P_{k} B}_{L^{\infty}} \aleq_{\En} \eps^{\frac{3}{4}(1-\dltnc)}$). 

One minor downside of Lemma~\ref{l:elliptic} is that the polar coordinates are not 
so convenient to use near zero and near infinity.  However, both near zero and near infinity
we have small $L^2$ curvature, so we may directly apply Uhlenbeck's lemmas; see Theorems~\ref{t:uhl} and \ref{t:uhl-ext} in the appendix.
Thus, after standard partitioning and regluing operations, the problem reduces to the simpler case when we work in an annulus:

\begin{lemma} \label{l:annulus}
Let $A \in \dot H^1$ be a connection in the annulus $\calA_{(1, 2)}$ with energy at most $\En$, which has the following properties (all norms are implicitly defined on $\calA_{(1, 2)}$ by restriction):

i) Small tangential curvature,
\[
\| F_{\Theta\Theta} \|_{L^2} \leq \epsilon.
\]

ii) Small angular covariant divergence of the transversal curvature,
\begin{equation} \label{eq:small-sdiv-FrTht}
\| \scovD^{\Tht} F_{\Theta r}\|_{\dot H^{-1}_{A}} \leq \epsilon.
\end{equation}
If $\epsilon \ll_\En 1$ then there is a gauge  equivalent connection $B$ in the 
Coulomb gauge\footnote{Here we reserve the right to choose the metric favorably.}
with the following properties:

a) Bounded size,
\[
\| B \|_{\dot H^1} \lesssim_\En 1.
\]

b) Small components:
\[
\| B_r\|_{\dot H^1} + \| \snb_\Theta B_\Theta\|_{L^2} + \nrm{r^{-1} B_{\Tht}}_{L^{2}} \lesssim_\En \epsilon^{1-\dltnc}.
\]
\end{lemma}

\begin{remark} \label{rem:small-L4-ann}
As a corollary of Lemma~\ref{l:bern} and Properties (a)--(b) in Lemma~\ref{l:annulus}, we have
\begin{equation}
\|B_r \|_{L^4}  \lesssim_{\En} \epsilon^{1-\dltnc}, \qquad \|B_\Theta\|_{L^4} \lesssim_{\En} \epsilon^{\frac{3}{4}(1-\dltnc)}.
\end{equation}
Indeed, the bound for $B_{r}$ is simply the Sobolev embedding. To apply Lemma~\ref{l:bern} to $B_{\Tht}$, we need to find an $\dot{H}^{1}$-extension $\bar{B}_{\Tht}$ of $B_{\Tht}$ outside $\calA_{(1, 2)}$ that is supported in (say) $\calA_{(\frac{1}{2}, \frac{5}{2})}$ and
\begin{equation*}
\nrm{\rd_{r} \bar{B}_{\Tht}}_{L^{2}} \aleq_{\En} 1, \quad
\nrm{\covnb_{\Tht} \bar{B}_{\Tht}}_{L^{2}} + \nrm{r^{-1} \bar{B}_{\Tht}}_{L^{2}} \aleq_{\En} \eps^{1-\dltnc}.
\end{equation*}
For this purpose, we take an even reflection of $B_{\Tht}$ across the boundaries of $\rd \calA_{(1, 2)}$ (see Lemma~\ref{lem:ext-simple}) and apply a radial cutoff that equals one on $\calA_{(1, 2)}$ and supported in $\calA_{(\frac{1}{2}, \frac{5}{2})}$.

Conversely, if such a bound holds  then the $\dot H^1$ type bounds follow 
by solving linear elliptic systems; see the proof of Lemma~\ref{l:annulus} below.
\end{remark}

We now successively prove the above results in reverse order:
\begin{proof}[Proof of Lemma~\ref{l:annulus}]
We seek the connection $B$, which is gauge-equivalent to $A$, so that it satisfies the Coulomb gauge condition $\covnb^{k} B_{k} = 0$ with respect to the metric
\begin{equation*}
	\bfe = \ud r^{2} + \ud \Tht^{2},
\end{equation*}
and with the boundary condition 
\[
B_r = 0 \qquad \text{on } \partial \calA_{(1, 2)}. 
\]
We claim that such a connection exists, and satisfies the conclusion of the lemma.

We use an elliptic bootstrap (i.e., a continuity) argument. Suppose that 
we have a continuous one parameter family of connections 
\[
A^{(h)} \in (\dot H^1\cap L^4)(\calA_{(1, 2)}), \qquad h \in [0,1], \qquad  A^{(0)}= 0, \quad A^{(1)} = A 
\]
so that the hypotheses of the lemma hold uniformly in $h \in [0,1]$. Then the Coulomb connection $B^{(0)}=0$
is the obvious solution when $h = 0$, and we seek to extend this property by continuity up to $h = 1$.
For this we consider the following bootstrap assumption:

\bigskip
\emph{ The connection $A^{(h)}$ admits a Coulomb gauge representation $B^{(h)}$ as above,
and which satisfies the additional property
\begin{equation}\label{boot}
\| B^{(h)} \|_{H^{1}} \leq C_{0}, \quad \| B^{(h)}\|_{L^4} \leq \epsilon^{\frac34(1-\dltnc)}.
\end{equation}
}

We will establish that, if $C_{0}$ is large enough and $\eps$ is sufficiently small (depending on $\En$ and chosen in this order), the set 
\[
H = \{ h \in [0,1] : \ \text{ \eqref{boot} holds} \},
\]
which trivially contains $0$, is both open and closed, and thus contain $h = 1$.

\pfstep{Fractional Sobolev spaces and elliptic operators on $\bbS^{3}$ and $\calA_{(1, 2)}$}
In what follows, we will employ fractional Sobolev spaces on $\bbS^{3}$ and $\calA_{(1, 2)}$, where we will distinguish between tangential and tranversal regularities in the latter case (for more details, see below). We start with the case of the unit sphere $\bbS^{3}$. Denote by $\calX$ a finite set of smooth vector fields $X$ on $\bbS^{3}$ that spans the tangent space at each point (e.g., the set of rotations $\Omg_{jk}$ in the $x^{j}x^{k}$-plane for $j, k=1, \ldots, 4$ would do). The $L^{2}$-Sobolev space of $k$-forms on $\bbS^{3}$ of order $m \in \bbN$ is defined by the norm
\begin{equation*}
	\nrm{\omg}_{H^{m}(\bbS^{3}; \Lmb^{k})}^{2} = \sum_{m'=0}^{m} \sup_{\substack{X_{1}, \ldots X_{m'} \in \calX \\ X'_{1}, \ldots, X'_{k} \in \calX}} \nrm{X_{1} \cdots X_{m'} (\omg(X'_{1}, \ldots, X'_{k}))}_{L^{2}(\bbS^{3})}^{2}.
\end{equation*}
Clearly, any different choice of $\calX$ gives rise to an equivalent norm. As usual, these spaces are extended to negative orders by duality, and to fractional orders by complex interpolation. 

A basic operator in this setting is the Hodge Laplacian on $k$-forms, which we denote by $\slap_{k}$. It is a second order elliptic operator that is nonpositive on $L^{2}$. Thus, for any $\alp > 0$ and $\gmm \in \bbR$, $(-\slap_{k} + 1)^{\alp} : H^{\gmm+\alp}(\bbS^{3}; \Lmb^{k}) \to H^{\gmm}(\bbS^{3}; \Lmb^{k})$ has a well-defined inverse, which we denote by $(-\slap_{k} + 1)^{-\alp}$. Moreover, the first and second de Rham cohomology groups of $\bbS^{3}$ are trivial, so $\slap_{1}$ and $\slap_{2}$ have trivial kernel by the Hodge theorem. Thus, for $k = 1, 2$, the preceding discussion holds with $-\slap_{k} + 1$ replaced by $-\slap_{k}$.

Next, we consider the domain $\calA_{(1, 2)}$. For the moment, we view it as an open submanifold $(1, 2) \times \bbS^{3}$ of the compact manifold $\calM = (\bbR / 4 \bbZ)_{r} \times \bbS_{\Tht}^{3}$, equipped with the product metric $\bfe = \ud r^{2} + \ud \Tht^{2}$. We say that a vector field $X$ (resp. a $k$-form) on $\calM$ is \emph{tangential} (to the constant-$r$ spheres) if $\ud r(X) = 0$ (resp. $\iota_{\rd_{r}} \omg = 0$). A general $k$-form $\omg$ on $\calM$ may be decomposed into its tangential part, which may be identified with the pullback $\omg_{\Tht \ldots \Tht}$, and its transversal part $\iota_{\rd_{r}} \omg$, which is a tangential $(k-1)$-form. 

Let $\calX_{tan}$ be a finite set of smooth tangential vector fields on $\calA_{(1, 2)}$ that spans the tangent space of $\set{r = const}$ at every point. For $\sgm \in \bbR$ and $m \in \bbN$, we define the anisotropic $L^{2}$-Sobolev norm of order $(\sgm, m)$ for a tangential $k$-form $\omg$ by
\begin{equation*}
	\nrm{\omg}_{H^{\sgm, m}(\calM; \Lmb_{tan}^{k})}^{2} =  \sum_{m'=0}^{m} \sup_{\substack{X_{1}, \ldots X_{m'} \in \calX_{tan} \\ X'_{1}, \ldots, X'_{k} \in \calX_{tan}}} \nrm{X_{1} \cdots X_{m'} (\omg(X'_{1}, \ldots, X'_{k}))}_{H^{\sgm}(\calM)}^{2}.
\end{equation*}
Again, any other choice of $\calX_{tan}$ gives rise to an equivalent norm.
This definition is extended to negative $m$ by duality, and to fractional orders by complex interpolation (while keeping $\sgm$ fixed). Then the norm $H^{\sgm, \gmm}((1, 2) \times \bbS^{3}, \Lmb^{k}_{tan})$ is defined by restriction (cf. Section~\ref{subsec:notation}).

The Hodge Laplacian $\slap_{k}$ acts on a tangential $k$-form $\omg$ by viewing each $\omg(r, \Tht)$ as a $k$-form on the unit sphere $\bbS^{3}_{\Tht}$. 
It is not difficult to verify (via induction on $\gmm \in \bbN$, duality for $\gmm < 0$ and interpolation for $\gmm \in \bbR \setminus \bbZ$) that $(-\slap_{k} + 1)^{\alp} : H^{\sgm, \gmm+\alp}(\calM, \Lmb^{k}_{tan}) \to H^{\sgm, \gmm}(\calM, \Lmb^{k}_{tan})$ is invertible for any $\alp > 0$, $\sgm, \gmm \in \bbR$; the same property holds on $(1, 2) \times \bbS^{3}$, too. Moreover, it is clear that $H^{\sgm, \gmm}(\calM; \Lmb^{k}_{tan})$ admits the following useful spectral characterization in terms of the commuting self-adjoint operators $-\calL_{\rd_{r}}^{2} + 1$ and $-\slap_{k} + 1$ on $L^{2}(\calM; \Lmb^{k}_{tan})$:
\begin{equation} \label{eq:sob-spec}
	\nrm{\omg}_{H^{\sgm, \gmm}(\calM; \Lmb^{k}_{tan})} \aeq \nrm{(-\calL_{\rd_{r}}^{2} - \slap_{k} + 2)^{\sgm} (-\slap_{k} + 1)^{\gmm} \omg}_{L^{2}(\calM; \Lmb^{k}_{tan})}.
\end{equation}
The same conclusions hold with $-\slap_{k} + 1$ replaced by $-\slap_{k}$ if $k = 1, 2$. 

Finally, to connect back to the original setting, we note that for each $\sgm \in \bbR$ and $m \in \bbN$, we have the equivalence
\begin{equation*}
	\nrm{\omg}_{H^{\sgm, m}((1, 2) \times \bbS^{3}, \Lmb^{k}_{tan})}
	\aeq  \sum_{m'=0}^{m} \sup_{\substack{X_{1}, \ldots X_{m'} \in \calX_{tan} \\ X'_{1}, \ldots, X'_{k} \in \calX_{tan}}} \nrm{X_{1} \cdots X_{m'} (\omg(X'_{1}, \ldots, X'_{k}))}_{H^{\sgm}(\calA_{(1, 2)})}^{2}
\end{equation*}
where the norm $H^{\sgm}(\calA_{(1, 2)})$ on the RHS is defined with respect to the Euclidean metric on $\bbR^{4}$. Note also the equivalences $\nrm{\cdot}_{H^{\sgm}((1, 2) \times \bbS^{3})} = \nrm{\cdot}_{H^{\sgm, 0}((1, 2) \times \bbS^{3})} \aeq \nrm{\cdot}_{H^{\sgm}(\calA_{(1, 2)})}$ and $\nrm{\cdot}_{L^{p}_{r} L^{p}_{\Tht}((1, 2) \times \bbS^{3})} \aeq \nrm{\cdot}_{L^{p}(\calA_{(1, 2)})}$. Accordingly, in what follows we will refer to these norms simply by $H^{\sgm}$ and $L^{p}$, respectively, without any possibility of confusion.

\pfstep{A-priori estimates for $B^{(h)}$}
We now begin our proof in earnest. First, we establish some a-priori bounds for $B^{(h)}$, which improve the bootstrap assumption \eqref{boot} and also imply the desired bounds stated in the lemma.

In what follows, we suppress $h$ and just write $B = B^{(h)}$. We also abuse the notation a bit and write $F$ for the curvature $2$-form associated to $B$. We omit the dependence of constants on $\En$ and write $c$ for a small positive constant that may vary from line to line. We use the roman indices $a, b, \ldots$ for coordinates on $\bbS^{3}$ and use the metric $\ud \Tht^{2}$ to raise and lower these indices. We also suppress $\Lmb_{tan}^{k}$ in the norms when the degree of the differential form is clear from the context.

Note that $(\iota_{\rd_{r}} F) = F_{r \Tht}$, viewed as a tangential $1$-form, satisfies the following div-curl system on each sphere:
\begin{equation*}
\left\{
\begin{aligned}
	\snb^{a} F_{a r} &= \scovD^{a} F_{a r} - [B^{a}, F_{a r}], \\
	\rd_{a} F_{b r} - \rd_{b} F_{a r}  &= - \rd_{r} F_{ab} - [B_{r}, F_{ab}] - [B_{a}, F_{rb}] + [B_{b}, F_{ra}].
\end{aligned}
\right.
\end{equation*}
Thus
\begin{equation} \label{eq:boot-G-1}
	\slap_{1} F_{a r} + \rd_{a} [B^{b}, F_{b r}] + \snb^{b} ([B_{b}, F_{ra}] - [B_{a}, F_{rb}]) = G,
\end{equation}
where
\begin{equation} \label{eq:boot-G-2}
G = \rd_{a} (\scovD^{b} F_{b r} - [B^{b}, F_{b r}]) - \snb^{b} (\rd_{r} F_{ba} + [B_{r}, F_{ba}]).
\end{equation}

We claim that
\begin{equation} \label{eq:boot-key}
\nrm{[B^{a}, F_{a r}]}_{H^{-1+\dltnc, -\dltnc}((1, 2) \times \bbS^{3})} 
 \aleq_{C_{0}} \eps^{c} \nrm{F_{\Tht r}}_{H^{-1+\dltnc, 1-\dltnc}((1, 2) \times \bbS^{3})}.
\end{equation}
We defer the proof until later, but remark that this estimate (barely) fails when $\dltnc = 0$; this failure is the reason why we introduce $\dltnc > 0$ and the fractional anisotropic Sobolev spaces. Combined with the bound $\nrm{F_{\Tht r}}_{H^{-1+\dltnc, 1-\dltnc}((1, 2) \times \bbS^{3})} \aleq \nrm{F}_{L^{2}(\calA)} \aleq 1$, which is obvious from \eqref{eq:sob-spec}, it follows that the terms involving $B_{a}$ on the LHS may be absorbed into the main term. Therefore,
\begin{equation*}
\nrm{F_{\Tht r}}_{H^{-1+\dltnc, 1-\dltnc}((1, 2) \times \bbS^{3})} \aleq \nrm{G}_{H^{-1+\dltnc, -1-\dltnc}((1, 2) \times \bbS^{3})}.
\end{equation*}
On the one hand, by \eqref{eq:boot-G-2} and the assumptions, $G$ is $O(\eps)$ in $H^{-1, -1}((1, 2) \times \bbS^{3})$. On the other hand, by \eqref{eq:boot-G-1} and the estimate
\begin{equation} \label{eq:boot-key'}
\nrm{\rd_{a} [B^{b}, F_{b r}] + \snb^{b} ([B_{b}, F_{ra}] - [B_{a}, F_{rb}]))}_{H^{0, -2}((1, 2) \times \bbS^{3})}
 \aleq_{C_{0}} \eps^{c} \nrm{F}_{L^{2}((1, 2) \times \bbS^{3})},
\end{equation}
whose proof we also defer until later, it follows that $G$ is $O(1) + O_{C_{0}}(\eps^{c})$ in $H^{0, 2}((1, 2) \times \bbS^{3})$.
By interpolation, which is obvious from \eqref{eq:sob-spec}, we then have
\begin{equation*}
\nrm{F_{\Tht r}}_{H^{-1+\dltnc, 1-\dltnc}((1, 2) \times \bbS^{3})} \aleq \eps^{1-\dltnc},
\end{equation*} 
provided that $\eps$ is small enough depending on $C_{0}$.

Next, we can write an elliptic equation for $B_r$ in $(1, 2) \times \bbS^{3}$,
\begin{equation*}
(\rd_{r}^{2} + \slap_{0}) B_r + \snb^{a} [B_{a}, B_{r}] = \scovD^{a} F_{a r} - [B^{a}, F_{a r}],
\end{equation*}
with the Dirichlet boundary condition $B_{r} = 0$ on $\rd ((1, 2) \times \bbS^{3})$. Here the RHS has size $O(\eps) + O_{C_{0}}(\eps^{c} \epsilon^{1-\dltnc})$ in $H^{-1}((1, 2) \times \bbS^{3})$ because
of the hypothesis, \eqref{eq:boot-key} and the simple embedding (see \eqref{eq:sob-spec})
\begin{equation*}
H^{-1+\dltnc, -\dltnc}((1, 2) \times \bbS^{3}) \hookrightarrow H^{-1}((1, 2) \times \bbS^{3}).  
\end{equation*}
Also the coefficient $B_{a}$ on the 
left is small in $L^4$. Hence the elliptic problem is uniquely solvable, and the solution $B_r$ satisfies
\[
\nrm{B_{r}}_{H^{1}} \aleq \eps + C_{C_{0}} \eps^{c} \eps^{1-\dltnc}.
\]

Finally, for $B_\Theta$ we have the following div-curl system on each sphere:
\begin{equation*}
\left\{
\begin{aligned}
	\snb^{a} B_{a} &= - \rd_{r} B_{r}, \\
	\rd_{a} B_{b} - \rd_{b} B_{a} &= F_{ab} - [B_{a}, B_{b}].
\end{aligned}
\right.
\end{equation*} 
The first RHS is $O(\eps) + O_{C_{0}}(\eps^{c} \epsilon^{1-\dltnc})$ in $L^{2}$, whereas the second RHS is $O(\eps)$.
It follows that
\[
\| \covnb_{\Tht} B_\Theta\|_{L^2} + \|B_\Theta\|_{L^2} \lesssim \eps + C_{C_{0}} \eps^{c} \epsilon^{1-\dltnc}.
\]
On the other hand we can use $F_{r\Theta}$ to bound
\[
\| \covnb_{r} B_\Theta\|_{L^2} \leq \nrm{F_{r \Theta}}_{L^{2}} + \eps + C_{C_{0}} \eps^{c} \epsilon^{1-\dltnc} \aleq 1,
\]
if $\eps$ is sufficiently small depending on $C_{0}$. Thus, we have proved that the conclusion of the lemma holds for $h \in H$. Furthermore, we have
\begin{equation}\label{boot-out}
\| B\|_{H^1}  \aleq 1, \quad
\| B\|_{L^4} \aleq (\eps + C_{C_{0}} \epsilon^{c} \eps^{1-\dltnc})^{\frac{3}{4}},
\end{equation}
where we used Remark~\ref{rem:small-L4-ann} for the second estimate. Once we choose $C_{0}$ large enough and $\eps$ sufficiently small, the bootstrap assumption for $B$ is improved.

\pfstep{Proof of \eqref{eq:boot-key} and \eqref{eq:boot-key'}}
To conclude the proof of the a-priori estimates, it remains to establish \eqref{eq:boot-key} and \eqref{eq:boot-key'}. In both cases, the idea is to reduce the problem to global-in-spacetime estimates via localization and change of variables (cf.~proof of Lemma~\ref{l:bern}). The key ingredient in the reduction are the invariance of the anisotropic Sobolev spaces on $(1, 2) \times \bbS^{3}$ under multiplication by smooth functions and pullback by diffeomorphisms, both of which are straightforward to verify. 

First, we extend $B_{\Tht}$ and $F_{\Tht r}$ by an even reflection across the boundaries of $(1, 2) \times \bbS^{3}$ and apply a smooth cutoff that equals $1$ on $(1, 2) \times \bbS^{3}$ and is supported in $(\frac{1}{2}, \frac{5}{2}) \times \bbS^{3}$. Using a partition of unity on $\bbS^{3}$, it suffices to consider $B_{\Tht}$ and $F_{\Tht r}$ that are supported in $(\frac{1}{2}, \frac{5}{2}) \times \kpp$, where $\kpp$ is a spherical cap in $\bbS^{3}$. Finally, we use the invariance under pullback by diffeomorphisms to straighten $\kpp$ to a unit ball $B'$ in $\bbR^{3}$, and also the invariance under multiplication by smooth functions to strip away the (variable coefficient) metric $\bfe$ and the volume form. 

As a result, estimates \eqref{eq:boot-key} and \eqref{eq:boot-key'} are reduced, respectively, to the following estimates for functions $u, v$ on $\bbR^{4}$:
\begin{align}
	\nrm{(-\lap + 1)^{-\frac{1-\dltnc}{2}} (u v)}_{L^{2} H^{-\dlt_{0}}} & \aleq \eps^{c} \nrm{(-\lap + 1)^{-\frac{1 - \dltnc}{2}} v}_{L^{2} H^{1-\dltnc}}, \label{eq:boot-key-core} \\
	\nrm{u v}_{L^{2} H^{-1}} & \aleq \eps^{c} \nrm{v}_{L^{2}}, \label{eq:boot-key'-core}
\end{align}
where the mixed norms are defined with respect to $x^{1}$ and $x' = (x^{2}, x^{3}, x^{4})$, and $u$ obeys 
\begin{equation} \label{eq:boot-core}
	\nrm{u}_{H^{1}} \aleq 1, \quad \nrm{u}_{L^{4}} \aleq \eps^{\frac{3}{4}(1-\dltnc)}.
\end{equation}

Before we turn to the proof of \eqref{eq:boot-key-core} and \eqref{eq:boot-key'-core}, we first deduce from \eqref{eq:boot-core}
\begin{equation} \label{eq:boot-core-L3}
	\nrm{u}_{L^{\infty} L^{3}} \aleq \eps^{c}.
\end{equation} 
We introduce the inhomogeneous Littlewood--Paley projections $\set{\tilde{P}_{j}}_{j \geq 0}$ on $\bbR^{4}$ (i.e., $\tilde{P}_{0} = P_{\leq 0}$ and $\tilde{P}_{j} = P_{j}$ for $j \geq 1$), as well as their analogues $\set{\tilde{P}'_{j}}_{j \geq 0}$ defined with respect to $x' = (x^{2}, x^{3}, x^{4})$. In view of the refined Sobolev embedding \cite[Theorem~1.43]{BCD}\footnote{To be precise, \cite[Theorem~1.43]{BCD} is formulated in terms of homogeneous spaces, but the inhomogeneous version stated here follows immediately.} on $\bbR^{3}$,
\begin{equation*}
	\nrm{u}_{L^{\infty} L^{3}}^{3} \aleq \nrm{u}_{L^{\infty} H^{\frac{1}{2}}}^{2} \sup_{j} \nrm{\tilde{P}'_{j} u}_{L^{\infty}},
\end{equation*}
it suffices to show that $\nrm{u}_{L^{\infty} H^{\frac{1}{2}}} \aleq 1$ and $\sup_{j} \nrm{\tilde{P}'_{j} u}_{L^{\infty}} \aleq \eps^{c}$. The former assertion follows from \eqref{eq:boot-core} and the trace theorem. For the latter assertion, we introduce a parameter $m > 1$ and estimate
\begin{align*}
2^{-j} \nrm{\tilde{P}'_{j} u}_{L^{\infty}}
& \aleq 2^{-j} \nrm{\tilde{P}'_{j} \tilde{P}_{\leq j+m}u}_{L^{\infty}} + 2^{-j} \nrm{\tilde{P}'_{j} \tilde{P}_{>j+m}u}_{L^{\infty}} \\
& \aleq \sum_{k \leq j+m} 2^{k} \nrm{\tilde{P}_{k} u}_{L^{4}} + \sum_{k > j-m} 2^{\frac{1}{2} j} \nrm{\tilde{P}_{k} u}_{L^{\infty} L^{2}} \\
& \aleq 2^{m} \nrm{u}_{L^{4}} + 2^{-\frac{1}{2} m} \nrm{u}_{H^{1}}.
\end{align*}
Then using \eqref{eq:boot-core} and optimizing the choice of $m$, the desired estimate follows.

Next, we establish \eqref{eq:boot-key-core}. We normalize $v$ so that $\nrm{(-\lap + 1)^{-\frac{1 - \dltnc}{2}} v}_{L^{2} H^{1-\dltnc}} \leq 1$. We decompose $u v = \sum_{j, k, \ell \geq 0} \tilde{P}_{j} (\tilde{P}_{k} u \tilde{P}_{\ell} v)$ and divide the proof into the following (overlapping) cases:
\begin{enumerate}
\item {\it Low-High interaction, $\abs{j - \ell} < 3$, $k < j + 5$.} We introduce the exponents $2-$ and $6-$ defined by the relations $\frac{1}{2-} = \frac{1}{2} + \frac{\dltnc}{3}$ and $\frac{1}{6-} = \frac{1}{6} + \frac{\dltnc}{3}$. We estimate
\begin{align*}
	2^{(-1+\dltnc)j} \nrm{\tilde{P}_{j}(\tilde{P}_{<j+5} u \tilde{P}_{\ell} v)}_{L^{2} H^{-\dlt_{0}}}
	& \aleq 2^{(-1+\dltnc)j} \nrm{\tilde{P}_{j} (\tilde{P}_{<j+5} u \tilde{P}_{\ell} v)}_{L^{2} L^{2-}} \\
	& \aleq 2^{(-1+\dltnc)j} \nrm{\tilde{P}_{< j+5} u}_{L^{\infty} L^{3}} \nrm{\tilde{P}_{\ell} v}_{L^{2} L^{6-}}  \\
	& \aleq \nrm{u}_{L^{\infty} L^{3}} 2^{(-1+\dltnc)\ell} \nrm{\tilde{P}_{\ell} v}_{L^{2} H^{1-\dltnc}},
\end{align*}
which is acceptable thanks to \eqref{eq:boot-core-L3}.

\item {\it High-Low interaction, $\abs{j - k} < 3$, $\ell < j + 5$.} This case can be handled similarly as in the Low-High interaction case; we even get an additional gain of $2^{(-1+\dltnc) (j-\ell)}$.
\item {\it High-High interaction, $\abs{k - \ell} < 3$, $j < \min \set{k, \ell} - 3$.} Let $m > 1$ be a parameter to be fixed later, and let $(\frac{3}{2}-)^{-1} = \frac{2}{3} + \frac{\dltnc}{3}$. For $\ell > j + m$, we estimate
\begin{align*}
	2^{(-1+\dltnc)j} \nrm{\tilde{P}_{j}(\tilde{P}_{k} u \tilde{P}_{\ell} v)}_{L^{2} H^{-\dlt_{0}}}
& \aleq 2^{(-1+\dltnc)j} \nrm{\tilde{P}_{j}(\tilde{P}_{k} u \tilde{P}_{\ell} v)}_{L^{2} L^{2-}} \\
& \aleq 2^{\dltnc j} \nrm{\tilde{P}_{j}(\tilde{P}_{k} u \tilde{P}_{\ell} v)}_{L^{1} L^{\frac{3}{2}-}} \\
& \aleq 2^{-\dltnc(\ell-j)} 2^{k} \nrm{\tilde{P}_{k} u}_{L^{2}} \, 2^{(-1+\dltnc) \ell}\nrm{\tilde{P}_{\ell} v}_{L^{2} L^{6-}}.
\end{align*}
After summation, the contribution of these terms is $O(2^{- \dlt_{0} m} )$.
On the other hand, for $\ell \leq j + m$, we estimate
\begin{align*}
	2^{(-1+\dltnc)j} \nrm{\tilde{P}_{j}(\tilde{P}_{k} u \tilde{P}_{\ell} v)}_{L^{2} H^{-\dlt_{0}}}
	& \aleq 2^{(1-\dltnc)(\ell - j)} \nrm{\tilde{P}_{k} u}_{L^{\infty }L^{3}} 2^{(-1+\dltnc) \ell} \nrm{\tilde{P}_{\ell} v}_{L^{2} L^{6-}},
\end{align*}
which contributes $O(2^{(1-\dlt_{0}) m} \nrm{u}_{L^{\infty} L^{3}})$ after summation. Using \eqref{eq:boot-core-L3} and optimizing the choice of $m$, the desired estimate follows.
\end{enumerate}

Finally, we prove \eqref{eq:boot-key'-core}. By the Sobolev embeddings on $\bbR^{3}$, we have
\begin{equation*}
	\nrm{uv}_{L^{2} H^{-1}} \aleq \nrm{uv}_{L^{2} L^{\frac{6}{5}}} \aleq \nrm{u}_{L^{\infty} L^{3}} \nrm{v}_{L^{2}},
\end{equation*}
which implies \eqref{eq:boot-key'-core} in view of \eqref{eq:boot-core-L3}.

\pfstep{Completion of the continuity argument}
Next, we consider a perturbative problem,  and prove  that  if 
$B$ is Coulomb, $\dot H^1$ and small in $L^4$
 as above,  then all connections $\tA$ which are sufficiently close to $A$ in $\dot H^1$  admit 
a similar Coulomb representation. 

Abusing the notation a bit, we write $A$ instead of $B$ and redefine $\tA$ by applying the same gauge transformation that takes $A$ to $B$. Hence, $\rd^{k} \tA_{k}$ is small. Applying a further gauge transformation (see Lemma~\ref{lem:Ar=0}), we may assume that $\tA_{r} = 0$ on $\rd \calA_{(1, 2)}$ as well. Then to find a gauge transformation $O$ which takes $\tA$ into the Coulomb gauge, we end up having to solve for $\Omg_{k} = O_{;k}$ the system\footnote{An alternative idea would have been to work with $O^{-1} \rd_{k} O$ as in Section~\ref{sec:reg}, which has the advantage that no $O$ appears in the div-curl system; see \eqref{eq:div-curl-O-5d}. However, for the boundary value problem on the annulus, the cokernel (and also the kernel) of the associated Neumann problem is nontrivial. The system \eqref{eq:div-curl-O-ann} has the virtue of having a cokernel independent of $O$, while it depends on $O$ for \eqref{eq:div-curl-O-5d}.}
\begin{equation} \label{eq:div-curl-O-ann}
\left\{
\begin{aligned}
	\rd^{k} \Omg_{k} &= \rd^{k}(Ad(O) \tA_{k}) = Ad(O) \rd^{k} \tA_{k} + [O_{;k}, Ad(O) \tA^{k}], \\
	\rd_{j} \Omg_{k} - \rd_{k} \Omg_{j} &= - [\Omg_{j}, \Omg_{k}],
\end{aligned}
\right.
\end{equation}
with the boundary condition $\Omg_{r} = 0$ on $\rd \calA_{(1, 2)}$. To solve this system, we start with $O^{(0)} = Id$ and construct $\Omg^{(n)}$ by applying Proposition~\ref{prop:div-curl-nonlin-ann} with $B = Ad(O^{(n-1)}) A$. Then $O^{(n)}$ is constructed by integrating the system of ODEs $O^{(n)}_{;j} = \rd_{j} O^{(n)} (O^{(n)})^{-1} = \Omg^{(n)}$, which is possible thanks to the curl condition for $\Omg^{(n)}$. By smallness of $\nrm{\rd^{k} \tA_{k}}_{L^{2}}$ and $\nrm{\tA}_{L^{4}}$, this iteration procedures goes through and we obtain a uniform bound $\nrm{O_{;x}^{(n)}}_{H^{1}} \aleq \nrm{\rd^{k} \tA_{k}}_{L^{2}}$. Taking the limit (along a suitable subsequence), we obtain a desired gauge transformation $O$ that also satisfies $\nrm{O_{;x}}_{H^{1}} \ll 1$.

The a-priori bound shows that if $h \in H$ then the stronger bound
\eqref{boot-out} holds. Then the perturbative argument shows
that for $h \in H$ there exists a fixed size neighborhood $[h-c,h+c]$
which is in $H$. We conclude that $H= [0,1]$, which completes the continuity
argument.  

\pfstep{Existence of a continuous path $A^{(h)}$}
The remaining issue is that of constructing a continuous path from $A$
to $0$.  In effect it suffices to show that there exists an extension
of $A$ inside the full unit ball which still satisfies the assumptions
of the lemma and so that $A$ vanishes near $x = 0$. Then we can obtain
the desired family by scaling\footnote{As in the original proof of 
  Uhlenbeck's lemma in \cite{MR648356}.},
\[
A^{(h)}(x) = hA(hx), \qquad h \in [0,1].
\] 

This can be done as follows:
\begin{enumerate}
\item In a suitable gauge set $A_r = 0$ on the boundary; see Lemma~\ref{lem:Ar=0}.
\item Double the annulus inside, and extend the connection as odd for $A_r$ and even for $A_\theta$.
This extension is still $H^1$, and the smallness hypothesis still holds in the double annulus; see Lemma~\ref{lem:ext-simple}.

\item Choose a sphere $S$ within the extended part on which
  $F_{\Theta \Theta}$ is small in $L^2$. Using Uhlenbeck's lemma on the $3$-sphere (Proposition~\ref{p:uhl-S3}), we may set $\nrm{A_\Theta}_{H^{1}(S)}$ small in a suitable gauge. In addition, again using Lemma~\ref{lem:Ar=0}, we may set $A_{r} = 0$ on $S$.

  \item Choose an extension of $A$ inside $S$ which is small in $H^1$. 
  More precisely, since the trace of $A_{r}$ vanishes on $S$, it follows that the extension of $A_{r}$ by zero inside of $S$ is in $H^{1}$. Similar considerations apply to $A_{\Tht}$ after subtracting an extension of the boundary values, which can be made to have a small $H^{1}$ norm. Overall, the $H^{1}$ norm of the extension is small inside $S$, so that the assumptions of the lemma are kept. Finally, by smallness we may harmlessly cutoff $A_{\Tht}$ near $0$ as well, as desired.
  \qedhere
\end{enumerate}

\end{proof}

\bigskip

\begin{proof}[Lemma~\ref{l:annulus} $\implies$ Lemma~\ref{l:elliptic}.]
In accordance with the choice of metric in Lemma~\ref{l:annulus}, we endow $\bbR^{4}$ with a smooth Riemannian metric that coincides with $\ud r^{2} + \ud \Tht^{2}$ in $\calA_{(1, 2)}$ and with the Euclidean metric $\ud r^{2} + r^{2} \ud \Tht^{2}$ outside $\calA_{(\frac{1}{2}, \frac{5}{2})}$. We formulate the Coulomb gauge conditions in the proof with respect to this metric. As already noted in \cite{MR648356}, Uhlenbeck's lemmas work just as well on Riemannian manifolds if we take $\epsU$ small enough; so does Lemma~\ref{lem:div-curl-O}, which is an interior elliptic regularity result .

By Uhlenbeck's lemma (Theorem~\ref{t:uhl}) we obtain a gauge-equivalent connection $A_{in}$ in $B_{\frac{3}{4}}$ which is $\epsilon$- small in $H^1$.
Next, by Uhlenbeck's lemma in the exterior of a ball (Theorem~\ref{t:uhl-ext}, we obtain a gauge-equivalent connection $A_{out}$ in $\bbR^{4} \setminus B_{1}$, which is $\epsilon$- small in $\dot H^1 \cap L^{4} (\bbR^{4} \setminus B_{1})$.
Now note that \eqref{eq:small-div-FrTht} is equivalent to \eqref{eq:small-sdiv-FrTht}, since
\begin{equation*}
\covD^{\frka} F_{\frka r} = r^{-2} \scovD^{\Tht} F_{\Theta r}.
\end{equation*}
Thus by Lemma~\ref{l:annulus} we obtain a gauge-equivalent connection $A_{mid}$ in the annulus $\calA_{(\frac23,\frac43)}$.
The $L^4$ smallness allows us to patch the three connections cleanly (without any topological obstructions). More precisely, the Coulomb gauge conditions imply that the transition maps $O$ in the intersections obey a favorable div-curl system. The $L^{4}$ bounds on $A_{in}$, $A_{mid}$ and $A_{out}$ imply that $O_{;x}$ (defined in each intersection) is small in $L^{4}$. Then, by the div-curl system we may upgrade this bound to smallness in $\dot{H}^{1}$, and then via Lemma~\ref{lem:div-curl-O} to smallness in $\ell^{1} \dot{H}^{1}$ (where we shrink the domain at each step). Thus each $O$ is uniformly closed to a constant (Lemma~\ref{lem:ptwise-O}), and a standard patching argument (see, e.g., \cite[Proposition~3.2]{MR648356}) now works.
\end{proof}

\bigskip

\begin{proof}[Lemma~\ref{l:elliptic} $\implies$ Proposition~\ref{p:parabolic}]
We start with a continuity argument.
Using the equivalent connection $A$ given by the lemma, we produce
  a continuous family of connections $A^h= hA$ with $h \in [0,1]$ so
  that $A^0 = 0$ and $A^1 = A$, and which satisfies uniformly the
  hypotheses of the proposition.

We consider the subset $H$ of $h \in [0,1]$ for which the following property holds:

\bigskip

\emph{ The Yang--Mills heat flow of  $A$ is global and satisfies the bound 
\[
\|F\|_{L^3_{x,s}} \leq 1.
\]}

Clearly $0 \in H$. Also by the continuity properties of the Yang--Mills
heat flow, $H$ is closed.  It remains to show that $H$ is open, which
would imply that $H = [0,1]$. For this it suffices to take the above
bound as a bootstrap assumption, and show that we can improve it.

Under this assumption, it follows immediately from Proposition~\ref{p:cal-a} that we have a gauge 
transformation $O$ with 
\[
\| O_{;x} \|_{\dot H^1} \lesssim_\En 1,
\]
which transforms $A$ into its caloric representation $\tA$. In turn $\tA$ 
must also satisfy
\[
\| \tA\|_{\dot H^1} \lesssim_{\En} 1.
\]
Further, since $A$ was small in $L^4$, it curvature is small in $\dot W^{-1,4}$,
\[
\|F \|_{\dot W^{-1,4}} \lesssim_\En \epsilon^{\frac{3}{4}(1-\dltnc)}.
\]
By this and the bound for $O_{;x}$, the curvature of $\tA$, namely $\tF = O F O^{-1}$, must also be small,
\[
\|\tF \|_{\dot W^{-1,4}} \lesssim_\En \epsilon^{\frac{3}{4}(1-\dltnc)}.
\]
Propagating this bound along the caloric flow \cite[Proposition~8.9; Eq.~(8.44)]{OTYM1}, we obtain
\begin{align*}
	\nrm{P_{k} \tF(s)}_{\dot{W}^{-1, 4}} \aleq_{\En}  \eps^{\frac{3}{4}(1-\dltnc)} (1+2^{2k} s)^{-100},
\end{align*}
On the other hand, by the bootstrap assumption and \cite[{Proposition~7.13; Eq.~(7.20)}]{OTYM1}, we also have
\begin{equation*}
	\nrm{P_{k} \tF(s)}_{L^{2}} \aleq_{\En} c_{k} (1+2^{2k} s)^{-100},
\end{equation*}
where $\nrm{c_{k}}_{\ell^{2}} \aleq_{\En} 1$. By Bernstein (for the second bound) and interpolation, we have
\begin{equation*}
	\nrm{P_{k} \tF(s)}_{L^{3}} \aleq_{\En} c_{k}^{\frac{2}{3}} \eps^{\frac{1}{4}(1-\dltnc)} 2^{\frac{2}{3}k} (1+2^{2k} s)^{-100}.
\end{equation*}
Then by Schur's test, we obtain
\begin{equation*}
	\nrm{\tF}_{L^{3}_{s,x}} \aleq_{\En} \eps^{\frac{1}{4}(1-\dltnc)} \nrm{c_{k}^{\frac{2}{3}}}_{\ell^{3}} \aleq_{\En} \eps^{\frac{1}{4}(1-\dltnc)},
\end{equation*}
which improves the bootstrap assumption. Moreover, by Bernstein and \cite[Proposition~8.9; Eq.~(8.45)]{OTYM1}, it follows that $\nrm{\tA}_{\ell^{\infty} \dot{W}^{-1, \infty}} \aleq \nrm{\tA}_{\ell^{\infty} L^{4}}  \aleq \eps^{\frac{3}{4}(1-\dltnc)}$. Then by the bound $\nrm{\tA}_{\dot{H}^{1}} \aleq_{\En} 1$ and the improved Sobolev inequality \cite[Theorem~1.43]{BCD}, \eqref{eq:parabolic-L4} follows.
This completes the proof of the proposition. \qedhere
\end{proof}

\bigskip

\begin{proof}[Proposition~\ref{p:parabolic} $\implies$   Theorem~\ref{thm:no-null}]
 
We express the curvature components in the null frame.
By \eqref{eq:no-null:hyp-small} and \eqref{eq:no-null:hyp-bdd}, as well as \eqref{eq:mono-L} and \eqref{eq:mono-uL} for the expression of $\mvC{X_{\veps}}_{0}(A) = \frac{1}{2} (\mvC{X_{\veps}}_{L}(A) + \mvC{X_{\veps}}_{\uL}(A))$, the null
components $\alp$, $\varrho$ and $\sgm$ are already small in
$L^{2}$ provided that $\eps_{1}$, $\veps$ and $1-\gmm$ are sufficiently small. We now use the constraint equation to express 
\begin{equation} \label{eq:div-F-small}
\covD^{\frka} F_{\frka r} 
= \covD^{\frka} (F_{\frka 0} +  F_{\frka L}) 
= - r^{-3} \covD_{r}(r^{3} F_{r 0}) +   \covD^{\frka} \alpha_{\frka}
= -r^{-3} \covD_{r}(r^{3} \varrho) +   \covD^{\frka} \alpha_{\frka},
\end{equation}
which implies the desired smallness of $\covD^{\frka} F_{\frka r}$ in the gauge-invariant space $H^{-1}_{A}$. 
Thus, we have established that the spatial part of the connection $A_x$ 
satisfies the hypotheses of Proposition~\ref{p:parabolic}. 

Suppose  now that $A_x$ is in the caloric gauge and satisfies the 
bounds in   Proposition~\ref{p:parabolic}. It remains to consider the temporal components 
of $F$. For $F_{0\frka}$ we write
\[
F_{0\frka} = F_{L\frka} + F_{r \frka} = \alp_{\frka} + F_{r \frka},
\]
and the $\dot W^{-1,4}$ smallness follows. For $F_{0r}$
we simply have
\[
F_{0r} = \frac{1}{2} F_{L \uL} = \varrho,
\]
which is small even in $L^2$. \qedhere
\end{proof}

\section{Proof of the Threshold Theorem and the Dichotomy Theorem} 
\label{sec:proof}

In this section, we finally prove the Threshold and Dichotomy Theorems (i.e., Theorems~\ref{t:threshold} and \ref{t:no-bubble}, respectively).

For both theorems, we argue by contradiction. Suppose that the conclusion of the Dichotomy Theorem (Theorem~\ref{t:no-bubble}) is false, i.e., there exists a solution $A$ for which both alternatives a) and b) are false.  Then we are in one of the following two scenarios:

\begin{enumerate}[label=(\roman*)]

\item The solution blows up in finite time, and the hypothesis of Theorem~\ref{t:bubble-off}
is false near the tip of the cone $C$.

\item The solution is global but the hypothesis of Theorem~\ref{t:bubble-off}
is false near the infinite end  of the cone $C$.
\end{enumerate}

On the other hand, assume now that the conclusion of the Threshold
Theorem (Theorem~\ref{t:threshold}) is false.  We seek to show that the conclusion of 
Theorem~\ref{t:bubble-off} is false, and therefore we are again in one of the two scenarios above.
To achieve this, we need to use the energy assumption $\nE(A) < 2 \Egs$ along with vanishing of the characteristic number $\ch = 0$ (as a consequence of topological triviality $A \in \dot{H}^{1}$). Our argument is similar to \cite{LO} (see also \cite[Section~6.2]{OTYM1}).

If the conclusion of
Theorem~\ref{t:bubble-off} were true, this would imply that a sequence
of translated,  rescaled and gauge transformed  copies $A^{(n)}$ of $A$ converges (modulo gauge
transformations) in $H^{1}_{loc}$
to a Lorentz transform of a nontrivial harmonic Yang--Mills
connection $L_v Q$. This implies (spacetime) $L^{2}_{loc}$ convergence of curvature tensors $F^{(n)}$, and thus for almost every $t \in (-1/2, 1/2)$ (and possibly passing to a subsequence) 
\begin{equation*}
	\nE_{\set{t} \times B_{R}}(A^{(n)}) = \frac{1}{2} \int_{B_{R}} \brk{F^{(n)}, F^{(n)}}(t) \to \calE_{\set{t} \times B_{R}}(L_{v} Q) \qquad \hbox{ for any } R > 0,
\end{equation*}
which in turn implies
\begin{equation*}
	\nE(Q) \leq \nE(A) < 2 \Egs.
\end{equation*}
By Theorem~\ref{t:gs}, the only possibility for $Q$ is that $\abs{\ch(Q)} = \spE(Q)$. Moreover, since Lorentz transform preserves the topological class, we have $\ch(L_{v}(Q)) = \ch(Q)$. From here on, we assume that $\ch(Q) > 0$; the alternative case is similar. 

Fix a large number $R \gg 1$ and $t \in (-1/2, 1/2)$. By topological triviality of $A^{(n)}(t)$, we have
\begin{equation*}
	0 = \ch(A^{(n)}(t)) = \int_{B_{R}} - \brk{F^{(n)} \wedge F^{(n)}}(t) + \int_{\bbR^{4} \setminus B_{R}} - \brk{F^{(n)} \wedge F^{(n)}}(t).
\end{equation*}
Again by the (spacetime) $L^{2}_{loc}$ convergence of $F^{(n)}$, there exists a subsequence such that
\begin{equation*}
	\int_{\bbR^{4} \setminus B_{R}} \brk{F^{(n)} \wedge F^{(n)}}(t) = \int_{B_{R}} - \brk{F^{(n)} \wedge F^{(n)}}(t) \to \int_{B_{R}} -\brk{F[L_{v} Q] \wedge F[L_{v} Q]}.
\end{equation*}
By \eqref{ch-vs-E}, we have
\begin{align*}
	\nE(A) 
	\geq & \limsup_{n \to \infty} \Big( \frac{1}{2} \int_{B_{R}} \brk{F^{(n)}, F^{(n)}}(t) + \abs{\int_{\bbR^{4} \setminus B_{R}} \brk{F^{(n)} \wedge F^{(n)}}(t)} \Big) \\
	\geq & \nE_{\set{t} \times B_{R}}(L_{v} Q) + \abs{\int_{B_{R}} -\brk{F[L_{v} Q] \wedge F[L_{v} Q]}}.
\end{align*}
Sending $R \to \infty$, the RHS tends to $\nE(L_{v} Q) + \ch(L_{v} Q) \geq 2 \nE(Q) \geq 2 E_{GS}$, which is a contradiction. 

It follows that the conclusion of Theorem~\ref{t:bubble-off}
is false, and thus its hypothesis is false. Hence we have reduced the problem again to the above alternative
(i)--(ii).  From here on, the proofs of the two theorems
are identical.  The analysis is largely similar in the two cases (i) and (ii), but there are still
some differences so we consider them separately.

\begin{remark} 
A difference in the properties of the solutions is that in the
subthreshold case we can work globally in the caloric gauge, whereas
otherwise we need the local solutions given by Theorem~\ref{t:local}.
However this makes no essential differences in the proofs below.
\end{remark}

\medskip

{\em (i) The blow-up scenario.}
Let $E = \nE(A)$.
 If $[0,T)$ is a maximal existence time, then the temporal gauge local well-posedness result 
 (Theorem~\ref{t:global-temp}) implies that there exists a point $X \in \bbR^4$ so
  that the energy does not decay to zero in the backward cone of
  $(X,T)$. 
  By translation invariance we will set $ (X,T)= (0,0)$ and, reversing time, 
  denote its forward cone by $C$. Thus we now have a 
 Yang--Mills connection $A$ with the property that
\begin{equation}\label{ft1}
\lim_{t \searrow 0} \nE_{S_t}(A) > \eps_0,
\end{equation}
where $\eps_0$ is a universal positive constant corresponding to the small data result.

We also know that the hypothesis of Theorem~\ref{t:bubble-off} is false, which 
gives
\begin{equation}\label{ft2}
\lim_{t \searrow 0} \nE_{S^\gamma_t}(A) = 0, \qquad 0 < \gamma < 1.
\end{equation}

We would like to use these two properties in order to show
that the connection $A$ admits a caloric representation near the tip of the cone,
which is also energy dispersed. Then we  could directly  apply the energy dispersed result in  Theorem~\ref{t:ED} 
to conclude that  the solution can be extended beyond the blow-up time $T = 0$, which is a contradiction. 
 However, this strategy cannot work unless the energy of $A$ is very small also outside the cone, which 
is not at all guaranteed a-priori.  To resolve this difficulty, we first truncate the solution outside 
the cone in order to insure that the outer energy stays small:

\begin{lemma}\label{l-final}
For each $\veps > 0$ there exists a $t_\veps > 0$ and a finite energy Yang--Mills solution $\tA$ in $(0,t_\veps]$
with the following properties:
\begin{enumerate}
\item Gauge equivalence: $\tA$ is gauge equivalent to $A$ in $C_{(0,t_\veps]}$.
\item Small energy outside the cone
\begin{equation} \label{ft3} \nE_{(\set{t} \times
        \bbR^{4}) \setminus S_{t}} (\tA) \leq
      \veps^{8} E \quad \hbox{ for every } t \in (0,t_\veps],
 \end{equation}
  \item Small flux on $\rd C$
    \begin{equation} \label{eq:ini-seq:flux1} \calF_{(0,t_\veps]}
      (\tA) \leq \veps^9 E.
    \end{equation}
  \end{enumerate}
 \end{lemma}

\begin{proof}
The flux energy relation \eqref{energy-flux} implies that the flux decays to zero near the tip of the cone,
\[
\lim_{t \to 0}  \calF_{(0,t]} (A) = 0.
\]
so we  first choose $t_\veps$ small enough so that the last condition is satisfied for $A$. We then choose
$\delta > 0$ so that the energy of $A$ in a $\delta$-annulus around the  cone is small,
\[
\nE_{\{ t_{\veps}\} \times \{ t_\veps < |x| < t_\veps+ 3\delta\}}(A) \leq \veps^9 E.
\]
Again by the flux energy relation, this propagates to all smaller times,
\[
\nE_{\{ t \} \times \{ t < |x| < t + 3\delta\}}(A) \leq 2 \veps^9 E, \qquad 0 < t < t_\veps.
\]
We now reset $t_\veps$ to a smaller value, 
\[
t_\veps \to \min\{ t_\veps,\delta\}.
\]
For this new choice we have 
\[
\nE_{\{ t_{\veps}\} \times \{ t_\veps < |x| < 4 t_\veps\}}(A) \leq 2 \veps^9 E.
\]

By (a rescaled form of) Proposition~\ref{t:chop-small}, we can truncate the data 
$(a,e)(t_\veps)$ for $A$ at time $t_\veps$. We obtain a new data set $(\ta,\te)(t_\veps)$
which  agrees with $(a,e)(t_\veps)$ inside the cone, but is small outside,
\[
\nE_{\{ t_{\veps}\} \times \{ t_\veps < |x| < 4 t_\veps\}}(\tA) \lesssim \veps^9 E.
\]

Now we consider the solution $\tA$ generated by the truncated data
$(\ta,\te)(t_\veps)$ below time $t_\veps$.  For as long as it exists it is gauge
equivalent with $A$ inside the cone, since they are gauge equivalent
initially, see Theorem~\ref{t:global-temp}. This shows that it
cannot blow-up inside the cone.  On the other hand outside the cone it
satisfies the second condition in the lemma by the flux-energy
relation, so it does not have enough energy blow up there either. It follows that $\tA$
persists as a finite energy solution and satisfies the conditions in the lemma up to time $t=0$. \qedhere
\end{proof}

From here on, we work with the connection $\tA$ which satisfies 
the properties \eqref{ft1}, \eqref{ft2} and \eqref{ft3} where for the latter 
we choose 
\[
\veps \ll_E 1.
\]
Since the flux decays to zero at the tip of the cone, Proposition~\ref{prop:monotonicity}
applied in the interval $[\veps t,t]$ also implies that for small enough $t$ we have
\begin{equation}\label{X-en}
\int_{S_{t}}  \mvC{X_{\veps t}}_{0}(\tA) \, \ud x \lesssim E.
\end{equation}

The bounds  \eqref{ft1}, \eqref{ft2} and \eqref{ft3} together with \eqref{X-en} with small enough $\veps$
guarantee  that the hypothesis of Theorem~\ref{thm:no-null} is satisfied for $\tA$ all small enough $t$.
This shows that the connection $\tA$ admits a caloric representation in $[0,t]$,
which is also energy dispersed. Thus we  can apply the energy dispersed result in  Theorem~\ref{t:ED} 
to conclude that  the solution can be extended beyond the blow-up time $T = 0$, which is a contradiction. 

\medskip

{\em (ii) The non-scattering scenario.}
This is similar but simpler, as we no longer need to truncate the data for $A$ and instead we may work
directly with $A$. For this we choose $R$ large enough so that the outer energy of $A$ is small,
\[
\nE_{\{0\} \times \{|x| > R\}}  (A) \leq \veps^9 E,
\]
and then work with the translated connection $A(t-R,x)$.  This
satisfies the conditions (2),(3) in Lemma~\ref{l-final} for $t \in
[R,\infty)$. From this point on, the bound \eqref{X-en} must hold for
all large enough $t$.  Hence the hypothesis of Theorem~\ref{thm:no-null}
is satisfied for $\tA$ all large enough $t$. This shows that the
connection $A$ admits a caloric representation in $[t,\infty]$, which
is also energy dispersed (in the case of Theorem~\ref{t:threshold}, we have smallness of the energy dispersion in the original gauge, thanks to the uniqueness of the caloric gauge). Thus we can apply the energy dispersed
result in Theorem~\ref{t:ED} in $[t,\infty)$ and conclude that $A \in
S^1([t,\infty))$, i.e., the desired scattering result.

\appendix
\section{Tools for analysis of gauge transformations} \label{sec:gt}
In this appendix, we collect various technical results, mostly concerning gauge transformations, that are used in the main text.
\subsection{Results from \cite{OTYM2.5}}
We recall some useful results that were proved in \cite{OTYM2.5}. The first result is essentially an extension to the critical regularity of a well-known result (see, for instance, \cite[Lemma~2.6]{MR648356}). 
\begin{lemma} [{\cite[Lemma~3.5]{OTYM2.5}}]\label{lem:Ar=0} 
Let $B$ be a $\g$-valued function in
 $H^{\frac{1}{2}}(\bbS^{3})$. There exists $O \in
  L^{\infty} \cap H^{2}(B_{1})$, which depends continuously on $B$,
  such that
  \begin{equation*}
    (O, O_{;r}) \restriction_{\set{r = 1}} = (Id, B),
  \end{equation*}
where $O_{;r} = \frac{x^{j}}{\abs{x}} O_{;j}$. A similar construction can be done in the exterior region $\bbR^{4}
  \setminus B_{1}$.
\end{lemma}
The next result works in tandem with Lemma~\ref{lem:Ar=0}, and provides an simple way to extend a $H^{1}$ connection $1$-form through a sphere.
\begin{lemma} [{\cite[Lemma~3.18]{OTYM2.5}}]\label{lem:ext-simple} Let $A \in H^{1}(B_{R})$ with $A_{r} = 0$ on $\rd B_{R}$. Extend $A$
  outside $B_{R}$ by
  \begin{equation*}
    \bar{A}_{r} \left( \frac{R^{2}}{r}, \Tht \right)= - A_{r}(r, \Tht), \quad
    \bar{A}_{\Tht} \left( \frac{R^{2}}{r}, \Tht \right)= A_{\Tht} (r, \Tht).
  \end{equation*}
  Then the curvature $2$-form $\bar{F}$ of the extension obeys
  \begin{equation*}  \bar{F}
    \left(\frac{R^{2}}{r}, \Tht\right) = F(r, \Tht) \quad \hbox{
      for } r < R.
  \end{equation*}
  A similar construction can be done starting from the exterior region $\bbR^{4}
  \setminus B_{R}$.
\end{lemma}

The following results are useful tools for proving continuity of a gauge transformation in the critical regularity setting.
\begin{lemma} [{\cite[Lemma~3.16]{OTYM2.5}}]\label{lem:div-curl-O} Let $O_{;j} \in \dot{W}^{1, \frac{d}{2}}(B_{2R})$
  be a solution to the div-curl system
  \begin{equation} \label{eq:div-curl-O}
    \begin{aligned}
      \rd_{j} O_{;k} - \rd_{k} O_{;j} =& [O_{;j}, O_{;k}] \\
      \rd^{\ell} O_{; \ell} = & H.
    \end{aligned}
  \end{equation}
  If $H \in \ell^{1} L^{\frac{d}{2}}(B_{2R})$, then $O_{;x} \in \ell^{1}
  \dot{W}^{1, \frac{d}{2}}(B_{R})$ with the
  bound
  \begin{equation*}
    \nrm{O_{;x}}_{\ell^{1} \dot{W}^{1, \frac{d}{2}}(B_{R})} \aleq \nrm{H}_{\ell^{1} L^{\frac{d}{2}}(B_{2R})} + \nrm{O_{;x}}_{\dot{W}^{1, \frac{d}{2}}(B_{2R})}^{2}.
  \end{equation*}
\end{lemma}
We note that, even though the divergence $\rd^{\ell} O_{;\ell}$ is formulated in terms of the Euclidean metric, the lemma works with $\bfe^{k \ell} \covnb_{k} O_{;\ell}$ defined with respect to any smooth metric on $B_{R}$ with suitable adjustment of the constants.

\begin{lemma} [{\cite[Lemma~3.17]{OTYM2.5}}]\label{lem:ptwise-O} 
If $O_{;x} \in \ell^{1} W^{1,
    \frac{d}{2}}(B)$, then $O$ is continuous on $B$.
\end{lemma}

\subsection{Solvability of div-curl systems}
Our aim in this subsection is to provide solvability results for the nonlinear div-curl system
\begin{equation} \label{eq:div-curl-nonlin}
\left\{
\begin{aligned}
\covnb^{\ell} \Omg_{\ell} &= \covnb^{\ell} B_{\ell} \\
\covnb_{j} \Omg_{k} - \covnb_{k} \Omg_{j} &= -[\Omg_{j}, \Omg_{k}]
\end{aligned}
\right. \qquad \hbox{ in } \calO,
\end{equation}
either with $\calO = \bbR^{5}$ but with respect to a variable metric $\bfe$, or with $\calO = \calA_{(R', R)} \subset \bbR^{4}$ (with respect to the Euclidean metric) and with suitable boundary conditions on $\rd \calA_{(R', R)}$; these problems arise in Sections~\ref{sec:reg} and \ref{sec:no-null}, respectively. To begin with, we solve in each case the easier linear system (see Section~\ref{subsec:notation} for the notation)
\begin{equation} \label{eq:div-curl}
\left\{
\begin{aligned}
	\dlt \omg =& f \\
	\ud \omg =& g
\end{aligned}
\right. \qquad \hbox{ in } \calO,
\end{equation}
where we remind the reader that $\dlt \omg = - \covnb^{\ell} \omg_{\ell}$ for a $1$-form $\omg$. 
\begin{lemma} [Linear div-curl system in $\bbR^{5}$] \label{lem:div-curl-5d}
Consider the div-curl system \eqref{eq:div-curl} in $(\bbR^{5}, \bfe)$, where $\bfe_{\alp \bt}$ is a smooth metric such that $\nrm{\bfe_{\alp \bt} - \bar{\bfe}_{\alp \bt}}_{L^{\infty} \cap \dot{H}^{\frac{5}{2}}} < \epse$ for some constant positive definite matrix $\bar{\bfe}_{\alp \bt}$. Suppose that $f, g$ are in $L^{2} \dot{H}^{\frac{1}{2}}$ and obey the compatibility condition
\begin{equation*}
	\ud g = 0.
\end{equation*}
Then for $\epse$ sufficiently small, there exists a unique solution $\omg$ to this problem such that $\omg \in L^{5}$ and $\nb \omg \in L^{2} \dot{H}^{\frac{1}{2}}$, which obeys
\begin{equation*}
	\nrm{\nb \omg}_{L^{2} \dot{H}^{\frac{1}{2}}} \aleq \nrm{f}_{L^{2} \dot{H}^{\frac{1}{2}}} + \nrm{g}_{L^{2} \dot{H}^{\frac{1}{2}}}.
\end{equation*}
\end{lemma}
Note that if $\omg \in L^{5}$ and $\nb \omg \in L^{2} \dot{H}^{\frac{1}{2}}$, then
\begin{equation} \label{eq:inhom-trace}
	\nrm{\omg}_{L^{5}} \aleq \nrm{\omg}_{L^{2} \dot{H}^{\frac{3}{2}}} +\nrm{\omg}_{L^{\infty} \dot{H}^{1}} \aleq \nrm{\nb \omg}_{L^{2} \dot{H}^{\frac{1}{2}}}.
\end{equation}
Indeed, the first inequality is simply the $4$-dimensional Sobolev inequality and interpolation, and the second inequality follows from applying the trace theorem to each $P_{k} \omg$ and square summing in $k$.
\begin{proof}
First, we treat the case $\bfe_{\alp \bt} = \bar{\bfe}_{\alp \bt}$, in which case \eqref{eq:div-curl} has constant coefficients. Let
\begin{equation*}
	\tilde{f} = (-\lap)^{-1} f, \quad \tilde{g}_{\alp \bt} = (-\lap)^{-1} g_{\alp \bt},
\end{equation*}
and define
\begin{equation*}
	\omg = \ud \tilde{f} + \dlt \tilde{g}.
\end{equation*}
By the compatibility condition $\ud g = 0$, as well as the relations $\ud^{2} = 0$ and $\dlt^{2} = 0$, it is easy to check that $\omg$ solves \eqref{eq:div-curl}. Any other solution differs from $\omg$ by a harmonic $1$-form in $L^{5}$, which must be zero; thus the uniqueness assertion follows. The desired estimate is then clear by Fourier transform.

In the general case, we use a simple perturbation argument. Note that \eqref{eq:div-curl} may be written as
\begin{equation*}
\left\{
\begin{aligned}
	\dlt \omg & = f + \mathrm{Err}_{\bfe} \omg, \\
	\ud \omg & = g,
\end{aligned}
\right.
\end{equation*}
where $\dlt$ denotes the constant coefficient divergence with respect to $\bar{\bfe}$ and
\begin{equation*}
\mathrm{Err}_{\bfe} \omg = (\bar{\bfe}^{\alp \bt} - \bfe^{\alp \bt}) \rd_{\alp} \omg_{\bt} - \frac{1}{\sqrt{\det \bfe}} \rd_{\alp} (\bfe^{\alp \bt} \sqrt{\det \bfe}) \omg_{\bt}.
\end{equation*}
By the standard Moser estimates in $L^{\infty} \cap \dot{H}^{\frac{5}{2}}$, which is an algebra, it is straightforward to establish (for $\epse < 1$)
\begin{equation*}
	\nrm{\bar{\bfe}^{\alp \bt} - \bfe^{\alp \bt}}_{L^{\infty} \cap \dot{H}^{\frac{5}{2}}} \aleq \epse, \quad
	\nrm{\frac{1}{\sqrt{\det \bfe}} \rd_{\alp} (\bfe^{\alp \bt} \sqrt{\det \bfe})}_{\dot{H}^{\frac{3}{2}}} \aleq \epse.
\end{equation*}
Combined with the embeddings $\dot{H}^{\frac{3}{2}} \cdot L^{2} \dot{H}^{\frac{3}{2}} \hookrightarrow L^{2} \dot{H}^{\frac{1}{2}}$ and $(L^{\infty} \cap \dot{H}^{\frac{5}{2}}) \cdot L^{2} \dot{H}^{\frac{1}{2}} \hookrightarrow L^{2} \dot{H}^{\frac{1}{2}}$, both of which follow from the standard Littlewood--Paley trichotomy with respect to $x^{1}, \ldots, x^{4}$, we have
\begin{equation*}
	\nrm{\mathrm{Err}_{\bfe} \omg}_{L^{2} \dot{H}^{\frac{1}{2}}} \aleq \epse \nrm{\nb \omg}_{L^{2} \dot{H}^{\frac{1}{2}}}.
\end{equation*}
Thus, for $\epse$ sufficiently small, \eqref{eq:div-curl} is solvable with $\nb \omg \in L^{2} \dot{H}^{\frac{1}{2}}$, as desired. \qedhere
\end{proof}

Next we consider the case $\calA_{(R', R)}$. We denote by $\nu$ the outward unit normal of $\calA_{(R', R)}$ on $\rd \calA_{(R', R)}$.
\begin{lemma} [Linear div-curl system in a $4$-dimensional annulus] \label{lem:div-curl-ann}
Consider the div-curl system \eqref{eq:div-curl} in $\calA_{(R', R)} \subset \bbR^{4}$ with the boundary condition
\begin{equation*}
	\iota_{\nu} \omg = 0 \qquad \hbox{ on } \rd \calA_{(R', R)}. 
\end{equation*}
Suppose that $f, g$ are in $L^{2}(\calA_{(R', R)})$ and obey the compatibility conditions
\begin{equation*}
	\int_{\calA_{(R', R)}} f = 0, \quad \ud g = 0.
\end{equation*}
Then there exists a unique solution $\omg \in H^{1}(\calA_{(R', R)})$ to this boundary value problem, which obeys
\begin{equation*}
	\nrm{\omg}_{H^{1}(\calA_{(R', R)})} \aleq_{R', R} \nrm{f}_{L^{2}(\calA_{(R', R)})} + \nrm{g}_{L^{2}(\calA_{(R', R)})}.
\end{equation*}
\end{lemma}
\begin{proof}
For simplicity, we write $\calA = \calA_{(R', R)}$ and omit the dependence of constants on $R'$ and $R$.
As in Lemma~\ref{lem:div-curl-5d}, we start by solving the following boundary value problem for $\tilde{g}$:
\begin{equation*}
\left\{
\begin{aligned}
	-\lap \tilde{g} & = g \quad \hbox{ in } \calA, \\
	(\iota_{\nu} \ud \tilde{g}, \iota_{\nu} \tilde{g}) &= 0 \quad \hbox{ on } \rd \calA.
\end{aligned}
\right. 
\end{equation*}
Here $-\lap = \dlt \ud + \ud \dlt$ denotes the Hodge Laplacian. By the solvability of the absolute boundary value problem for $2$-forms, we may find a unique solution $\tilde{g} \in H^{2}(\calA)$ to these problems; see, for instance, \cite[{Proposition~9.8}]{Tay1}. We remark that for uniqueness, we use the Hodge theorem and the fact that the second de Rham cohomology group of $\calA$ is trivial. We also note that
\begin{equation*}
	\ud \dlt \ud \tilde{g} = \ud (-\lap) \tilde{g} = \ud g = 0.
\end{equation*}
Testing $\ud \tilde{g}$ against the above equation, integrating $\ud$ by parts and using $\iota_{\nu} \ud \tilde{g} = 0$ to make the boundary terms vanish, it follows that $\dlt \ud \tilde{g} = 0$. 

Next, we solve the following Neumann boundary value problem for $\tilde{f}$:
\begin{equation*}
\left\{
\begin{aligned}
	-\lap \tilde{f} & = f \quad \hbox{ in } \calA, \\
	\iota_{\nu} \ud \tilde{f} & = - \iota_{\nu} \dlt \tilde{g} \quad \hbox{ on } \rd \calA, \\
	\int_{\calA} \tilde{f} &= 0.
\end{aligned}
\right.
\end{equation*}
To solve this problem with $\tilde{f} \in H^{2}(\calA)$, we need to verify the following the compatibility condition (which arises from integrating $-\lap \tilde{f} = f$ over $\calA$, integrating the LHS by parts and using the conditions on $f$ and $\iota_{\nu} \ud \tilde{f} \vert_{\rd \calA}$):
\begin{equation*}
	0 = \int_{\rd \calA} \iota_{\nu} \dlt \tilde{g} \  \ud \hbox{Vol}_{\rd \calA}.
\end{equation*}
Since $\iota_{\nu} \star 1$ is precisely the induced volume form $\ud \hbox{Vol}_{\rd \calA}$, for any sufficiently regular $3$-form $\eta$ defined in a neighborhood of $\rd \calA$ we have $\iota_{\nu} \star \eta \  \ud \hbox{Vol}_{\rd \calA} = i_{\rd \calA}^{\ast} \eta$, where $i_{\rd \calA}$ is the embedding $\rd \calA \hookrightarrow \overline{\calA}$. It follows that
\begin{equation*}
\iota_{\nu} \dlt \tilde{g} \  \ud \hbox{Vol}_{\rd \calA} = \ud (i_{\rd \calA}^{\ast} \star \tilde{g}), 
\end{equation*}
so that the compatibility condition holds by the Stokes theorem.

In conclusion, $\omg = \ud \tilde{f} + \dlt \tilde{g}$ gives a desired $H^{1}(\calA)$ solution to the div-curl system \eqref{eq:div-curl} with $\iota_{\nu} \omg \vert_{\rd \calA}= 0$; uniqueness follows from the Hodge theorem and the fact that the first de Rham cohomology group of $\calA$ is trivial.
\end{proof}

We are now ready to state and prove the perturbative solvability results for the nonlinear div-curl system \eqref{eq:div-curl-nonlin}.
\begin{proposition} [Nonlinear div-curl system in $\bbR^{5}$]\label{prop:div-curl-nonlin-5d}
Consider the nonlinear div-curl system \eqref{eq:div-curl-nonlin} in $(\bbR^{5}, \bfe)$, where $\bfe_{\alp \bt}$ is a smooth metric such that $\nrm{\bfe_{\alp \bt} - \bar{\bfe}_{\alp \bt}}_{L^{\infty} \cap \dot{H}^{\frac{5}{2}}} < \epse$ for some constant positive definite matrix $\bar{\bfe}_{\alp \bt}$ and $A$ obeys $A \in L^{5}$, $\nb A\in L^{2} \dot{H}^{\frac{1}{2}}$ and $\nrm{\covnb^{\alp} A_{\alp}}_{L^{2} \dot{H}^{\frac{1}{2}}} < \epsA$. Then for $\epse, \epsA$ sufficiently small, there exists a unique solution $\Omg$ to this problem such that $\Omg \in L^{5}$ and $\nb \Omg \in L^{2} \dot{H}^{\frac{1}{2}}$, which obeys
\begin{equation*}
	\nrm{\nb \Omg}_{L^{2} \dot{H}^{\frac{1}{2}}} \aleq \nrm{\covnb^{\alp} A_{\alp}}_{L^{2} \dot{H}^{\frac{1}{2}}}.
\end{equation*}
\end{proposition}
\begin{proof}
We aim to solve \eqref{eq:div-curl-nonlin} by iteration; however, the RHS of the $\ud \Omg$ equation may not satisfy the compatibility condition during the iteration procedure. To rectify this issue, we use a Leray-type projection operator. For a $2$-form $g$ defined in $\bbR^{5}$, we introduce the operators
\begin{equation*}
	\bbP^{df} g = \dlt (-\lap)^{-1} \ud g, \quad \bbP^{cf} = g - \bbP^{df} g.
\end{equation*}
where $\dlt$ and $(-\lap)^{-1}$ are defined with respect to the constant metric $\bar{\bfe}$; such a simple choice is allowed since the condition we need to ensure ($\ud (\cdot) = 0$) is independent of the metric.

Now we set up an iteration scheme by starting with $\Omg^{(0)} = 0$, and defining $\Omg^{(n)}$ by solving the following system: 
\begin{equation*}
\left\{
\begin{aligned}
	\covnb^{\alp} \Omg_{\alp}^{(n)} & = \covnb^{\alp} B_{\alp}, \\
	\ud \Omg^{(n)} & = \frac{1}{2} \bbP^{cf} [\Omg^{(n-1)} \wedge \Omg^{(n-1)}].
\end{aligned}
\right.
\end{equation*}
Using Lemma~\ref{lem:div-curl-5d}, it is straightforward to show that $\Omg^{(n)}$ has a limit $\Omg$ such that $\nrm{\nb \Omg}_{L^{2} \dot{H}^{\frac{1}{2}}} \aleq \nrm{\covnb^{\alp} B_{\alp}}_{L^{2} \dot{H}^{\frac{1}{2}}}$ and solves
\begin{equation*}
\left\{
\begin{aligned}
	\covnb^{\alp} \Omg_{\alp} & = \covnb^{\alp} B_{\alp}, \\
	\ud \Omg & = \frac{1}{2} \bbP^{cf} [\Omg \wedge \Omg] = \frac{1}{2} [\Omg \wedge \Omg] - Z,
\end{aligned}
\right.
\end{equation*}
where
\begin{equation*}
	Z = \frac{1}{2} \bbP^{df} [\Omg \wedge \Omg].
\end{equation*}
It remains to show that $Z = 0$. As a preparation, note that for any $1$-form $\omg$,
\begin{equation*}
	[[\omg \wedge \omg] \wedge \omg ] = 0,
\end{equation*}
which follows from the Jacobi identity for the Lie bracket. 
Thus, $Z$ obeys the identity
\begin{align*}
	Z &= \frac{1}{2} \dlt (-\lap)^{-1} \ud [\Omg \wedge \Omg]\\
	& = \frac{1}{2} \dlt (-\lap)^{-1} [[\Omg \wedge \Omg] \wedge \Omg] - \frac{1}{2} \dlt (-\lap)^{-1} [\bbP^{df} [\Omg \wedge \Omg] \wedge \Omg] \\
	& = - \dlt (-\lap)^{-1} [Z \wedge \Omg].
\end{align*}
But then since $Z \in L^{2} \dot{H}^{\frac{1}{2}}$ and
\begin{equation*}
	\nrm{Z}_{L^{2} \dot{H}^{\frac{1}{2}}} \aleq \nrm{\Omg}_{L^{5}} \nrm{Z}_{L^{2} \dot{H}^{\frac{1}{2}}},
\end{equation*}
we have $Z = 0$ provided that $\epsA$ is small enough, as desired.
\end{proof}

\begin{proposition} [Nonlinear div-curl system in a $4$-dimensional annulus]\label{prop:div-curl-nonlin-ann}
Consider the nonlinear div-curl system \eqref{eq:div-curl-nonlin} in $\calA_{(R', R)} \subset \bbR^{4}$ with the boundary condition
\begin{equation*}
	\iota_{\nu} \omg = 0 \qquad \hbox{ on } \rd \calA_{(R', R)}. 
\end{equation*}
Suppose that $A$ is in $H^{1}(\calA_{(R', R)})$ and obeys $\nrm{\rd^{\ell} A_{\ell}}_{L^{2}(\calA_{(R', R)})} < \epsA$.
Then for $\epsA$ sufficiently small (depending on $R', R$), there exists a unique solution $\Omg \in H^{1}(\calA_{(R', R)})$ to this boundary value problem, which obeys
\begin{equation*}
	\nrm{\Omg}_{H^{1}(\calA_{(R', R)})} \aleq_{R', R} \nrm{\rd^{\ell} A_{\ell}}_{L^{2}(\calA_{(R', R)})}.
\end{equation*}
\end{proposition}
\begin{proof}
Again, for simplicity, we write $\calA = \calA_{(R', R)}$ and omit the dependence of constants on $R', R$. As in the proof of Proposition~\ref{prop:div-curl-nonlin-5d}, the crucial step is to construct a suitable projection that enforces the compatibility condition. For sufficiently smooth $g$ ($g \in H^{2}(\calA)$ is enough), solve the boundary value problem
\begin{equation*}
\left\{
\begin{aligned}
-\lap u &= \ud g \quad \hbox{ in } \calA, \\
(\iota_{\nu} \ud u, \iota_{\nu} u) & = 0 \quad \hbox{ on } \rd \calA, \\
u & \perp \calH^{3}_{A}(\calA),
\end{aligned}
\right.
\end{equation*}
where $\calH^{3}_{A}(\calA)$ is the space of harmonic $3$-forms $\eta$ satisfying $\iota_{\nu} \eta = 0$; see, for instance, \cite[{Proposition~9.8}]{Tay1}. Note that the solvability condition $\ud g \perp \calH^{3}_{A}$ is clearly satisfied. Furthermore, observe that $\ud \dlt \ud u = \ud^{2} g = 0$; thus testing by $\ud u$ and using the boundary condition $\iota_{\nu} \ud u = 0$, it follows that $\dlt \ud u = 0$. Thus, if we define
\begin{equation*}
	\bbP^{df} g = \dlt u, \quad \bbP^{cf} = g - \bbP^{df} g,
\end{equation*}
then $\ud \bbP^{df} g = \ud g$ and $\ud \bbP^{cf} g = 0$. Moreover, by the $H^{1}$ estimate for the Hodge Laplacian,
\begin{equation*}
	\nrm{\bbP^{df} g}_{L^{2}(\calA)} \aleq \nrm{u}_{H^{1}(\calA)} \aleq \nrm{g}_{L^{2}(\calA)},
\end{equation*}
by which we may extend $\bbP^{df}$ to any $g \in L^{2}(\calA)$.

As in Proposition~\ref{prop:div-curl-nonlin-5d}, we now solve the system
\begin{equation*}
\left\{
\begin{aligned}
	\dlt \Omg & = \dlt B, \\
	\ud \Omg & = \frac{1}{2} \bbP^{cf} [\Omg \wedge \Omg],
\end{aligned}
\right.
\end{equation*}
with the boundary condition $\iota_{\nu} \Omg = 0$ by iteration. Note that $\iota_{\nu} B = 0$ on $\rd \calA$ ensures the compatibility condition $\int_{\calA} \dlt B = 0$. Then to finish the proof, it suffices to show that
\begin{equation*}
Z = \frac{1}{2} \bbP^{df} [\Omg \wedge \Omg],
\end{equation*}
which is a-priori in $L^{2}$, must vanish. Since $\bbP^{df} g$ is defined from $\ud g$, we may perform a similar computation as in Proposition~\ref{prop:div-curl-nonlin-5d} and conclude that $Z$ obeys a self-improving relation if $\epsA$ is sufficiently small; thus $Z = 0$ as desired. \qedhere
\end{proof}

\subsection{Uhlenbeck's lemmas}
Here, we record various lemmas that allows us to pass to the Coulomb gauge under a suitable gauge-independent smallness condition. We begin with the classical results proved by Uhlenbeck. In what follows, $\nu$ denotes the outward unit normal to $B_{R}$ on $\rd B_{R}$.
\begin{theorem} \label{t:uhl}
Let $A$ be a connection in a ball $B_{R} \subset \bbR^{4}$ that satisfies $A \in H^{1}(B_{R})$ and $\nrm{F}_{L^{2}(B_{R})} < \epsU$.
If $\epsU$ is sufficiently small, then there exists an admissible gauge transformation $O$, unique up to multiplication by a constant element of $\G$, such that $\tA = \calG(O) A$ obeys
    \begin{equation*}
      \rd^{\ell} \tA_{\ell} = 0 \hbox{ in } B_{R}, \qquad
      \nu^{\ell} \tA_{\ell} = 0 \hbox{ on } \rd B_{R}
    \end{equation*}
    and
    \begin{equation*}
      \nrm{\tA}_{\dot{H}^{1}(B_{R})} \aleq \nrm{F}_{L^{2}(B_{R})}.
    \end{equation*}
\end{theorem}

\begin{theorem} \label{t:uhl-ext}
Let $A$ be a connection in $\bbR^{4} \setminus B_{R}$ that satisfies $A \in \dot{H}^{1} \cap L^{4}(\bbR^{4} \setminus B_{R})$ and $\nrm{F}_{L^{2}(\bbR^{4} \setminus B_{R})} < \epsU$.
If $\epsU$ is sufficiently small, then there exists an admissible gauge transformation $O$, unique up
    to multiplication by a constant element of $\G$, such that $\tA = \calG(O) A$ obeys
    \begin{equation*}
      \rd^{\ell} \tA_{\ell} = 0 \hbox{ in } \bbR^{4} \setminus B_{R}, \qquad
      \nu^{\ell} \tA_{\ell} = 0 \hbox{ on } \rd B_{R}
    \end{equation*}
    and
    \begin{equation*}
      \nrm{\tA}_{\dot{H}^{1}(\bbR^{4} \setminus B_{R})} \aleq \nrm{F}_{L^{2}(\bbR^{4} \setminus B_{R})}.
    \end{equation*}
\end{theorem}

The first result is essentially \cite[Theorem~1.3]{MR648356}; see \cite[Theorem~3.11]{OTYM2.5} for the uniqueness assertion at the critical regularity. The second result is the combination of \cite[Theorem~4.5]{MR815194}, which is formulated on a punctured disk, and a conformal inversion procedures, which is also in \cite{MR815194}. See, also, Theorem~3.12 and the proof of Theorem~1.5 in \cite{OTYM2.5}.

We end with a result that concerns a connection on the unit $3$-sphere $\bbS^{3}$ whose curvature is small in $L^{2}$; note that this is a subcritical assumption. 
\begin{proposition} [Subcritical Uhlenbeck on $\bbS^{3}$] \label{p:uhl-S3}
Let $A$ be a connection in $\bbS^{3}$ that satisfies $A \in H^{1}_{\Tht}(\bbS^{3})$ and $\nrm{F}_{L^{2}_{\Tht}(\bbS^{3})} < \epsU$. If $\epsU$ is sufficiently small, then there exists a $H^{2}_{\Tht}$ gauge transformation $O$, unique up to multiplication by a constant element of $\G$, such that $\tA = \calG(O) A$ obeys $\covnb^{\Tht} A_{\Tht} = 0$ and
\begin{equation*}
	\nrm{\tA}_{H^{1}_{\Tht}(\bbS^{3})} \aleq \nrm{F}_{L^{2}_{\Tht}(\bbS^{3})}
\end{equation*}
\end{proposition}
This proposition is a slight strengthening of \cite[Theorem~2.5]{U2}; we include a sketch of the proof for completeness.
\begin{proof}
We cover $\bbS^{3}$ by two caps $\calO_{N}$ and $\calO_{S}$ centered at the north and the south poles, respectively, and apply the usual Uhlenbeck's lemma (Theorem~\ref{t:uhl}) to each; we denote the resulting representations by $A^{(N)}$ and $A^{(S)}$, respectively. In the intersection, $A^{(N)} = \calG(O^{(NS)}) A^{(S)}$ for some gauge transformation $O^{(NS)} \in H^{2}_{\Tht}(\calO_{N} \cap \calO_{S})$. By the Sobolev embedding, taking $\epsU$ small enough, the image of $O^{(NS)}$ is contained in a small ball near a constant element. Hence we may patch together $A^{(N)}$ and $A^{(S)}$ to obtain a global representation $\bar{A}$ such that $\nrm{\bar{A}}_{H^{1}_{\Tht}} \aleq \epsU$. Finally, applying a subcritical perturbative argument (see, for instance, \cite[Proof of Theorem~2.5]{U2}), we find a gauge transformation from $\bar{A}$ into the Coulomb gauge. The uniqueness assertion follows also from the same perturbative argument. \qedhere
\end{proof}

\bibliographystyle{ym}
\bibliography{ym}

\end{document}